\documentclass[english]{article}
\usepackage{lmodern}
\usepackage[T1]{fontenc}
\usepackage[latin9]{inputenc}
\usepackage{mathtools}
\usepackage{geometry}
\usepackage[usenames,dvipsnames,svgnames,table]{xcolor}
\usepackage{babel}
\usepackage{refstyle}
\usepackage{float}
\usepackage{mathtools}
\usepackage{amsmath}
\usepackage{amsthm}
\usepackage{amssymb}
\usepackage{graphicx}
\usepackage{caption}
\usepackage[all]{xy}
\usepackage[unicode=true,pdfusetitle,
 bookmarks=true,bookmarksnumbered=false,bookmarksopen=false,
 breaklinks=false,pdfborder={0 0 0},pdfborderstyle={},backref=false,colorlinks=true]
 {hyperref}

\makeatletter

\AtBeginDocument{\providecommand\propref[1]{\ref{prop:#1}}}
\AtBeginDocument{\providecommand\thmref[1]{\ref{thm:#1}}}
\AtBeginDocument{\providecommand\exaref[1]{\ref{exa:#1}}}
\AtBeginDocument{\providecommand\corref[1]{\ref{cor:#1}}}
\AtBeginDocument{\providecommand\defref[1]{\ref{def:#1}}}
\AtBeginDocument{\providecommand\remref[1]{\ref{rem:#1}}}
\AtBeginDocument{\providecommand\conjref[1]{\ref{conj:#1}}}
\AtBeginDocument{\providecommand\lemref[1]{\ref{lem:#1}}}
\AtBeginDocument{\providecommand\subsecref[1]{\ref{subsec:#1}}}
\RS@ifundefined{subsecref}
  {\newref{subsec}{name = \RSsectxt}}
  {}
\RS@ifundefined{thmref}
  {\def\RSthmtxt{theorem~}\newref{thm}{name = \RSthmtxt}}
  {}
\RS@ifundefined{lemref}
  {\def\RSlemtxt{lemma~}\newref{lem}{name = \RSlemtxt}}
  {}

\theoremstyle{plain}
\newtheorem{thm}{\protect\theoremname}[subsection]
\theoremstyle{definition}
\newtheorem{defn}[thm]{\protect\definitionname}
\newtheorem*{defn*}{\protect\definitionname}
\theoremstyle{definition}
\newtheorem{rem}[thm]{\protect\remarkname}
\theoremstyle{plain}

\theoremstyle{plain}
\newtheorem{conjecture}[thm]{\protect\conjecturename}
\theoremstyle{plain}
\newtheorem*{conjecture*}{\protect\conjecturename}
\newtheorem{lem}[thm]{\protect\lemmaname}
\theoremstyle{definition}

\theoremstyle{plain}
\newtheorem{prop}[thm]{\protect\propositionname}
\theoremstyle{plain}
\newtheorem*{prop*}{\protect\propositionname}
\newtheorem{cor}[thm]{\protect\corollaryname}
\theoremstyle{definition}
\newtheorem{example}[thm]{\protect\examplename}

\newtheorem*{rem*}{\protect\remarkname}
\newtheorem{theorem}{Theorem}

\@ifundefined{date}{}{\date{}}

\usepackage{dsfont}
\newref{prop}{name=Proposition~}
\newref{cor}{name=Corollary~}
\newref{rem}{name=Remark~}
\newref{def}{name=Definition~}
\newref{exa}{name=Example~}
\newref{thm}{name=Theorem~}
\newref{lem}{name=Lemma~}
\newref{conj}{name=Conjecture~}

\usepackage{hyperref}
\hypersetup{bookmarksnumbered,colorlinks=true}
\hypersetup{
     colorlinks   = true,
     linkcolor	  = DarkGreen,
     citecolor    = Purple
}

\makeatother

\providecommand{\conjecturename}{Conjecture}
\providecommand{\corollaryname}{Corollary}
\providecommand{\definitionname}{Definition}
\providecommand{\examplename}{Example}
\providecommand{\lemmaname}{Lemma}
\providecommand{\notationname}{Notation}
\providecommand{\propositionname}{Proposition}
\providecommand{\questionname}{Question}
\providecommand{\remarkname}{Remark}
\providecommand{\theoremname}{Theorem}

\usepackage{cjhebrew}
\DeclareFontFamily{U}{rcjhbltx}{}
\DeclareFontShape{U}{rcjhbltx}{m}{n}{<->s*[1.2]rcjhbltx}{}
\DeclareSymbolFont{hebrewletters}{U}{rcjhbltx}{m}{n}
\DeclareMathSymbol{\pretsadi}{\mathord}{hebrewletters}{118}

\makeatother

\newcommand{\cone}{\mathbin{\rotatebox[origin=c]{-90}{$\triangle$}}}
\newcommand{\es}{\operatorname{{\small \normalfont{\text{\O}}}}}
\newcommand{\iso}{\xrightarrow{\,\smash{\raisebox{-0.5ex}{\ensuremath{\scriptstyle\sim}}}\,}}
\DeclareMathOperator{\supp}{supp}

\newcommand{\Yo}{\text{\usefont{U}{min}{m}{n}\symbol{'110}}}

\DeclareFontFamily{U}{min}{}
\DeclareFontShape{U}{min}{m}{n}{<-> dmjhira}{}

\begin{document}
\global\long\def\iso{\overset{\sim}{\longrightarrow}}%
\global\long\def\into{\hookrightarrow}%
\global\long\def\onto{\twoheadrightarrow}%
\global\long\def\ss{\subseteq}%
\global\long\def\adj{\leftrightarrows}%
\global\long\def\oto#1{\xrightarrow{#1}}%
\global\long\def\ofrom#1{\xleftarrow{#1}}%
\global\long\def\bb#1{\mathbb{#1}}%
\global\long\def\es{\varnothing}%
\global\long\def\oo#1{\overset{\circ}{#1}}%
\global\long\def\cone{\triangleright}%
\global\long\def\from{\leftarrow}%
\global\long\def\nto{\rightsquigarrow}%

\global\long\def\map{\operatorname{Map}}%
\global\long\def\End{\operatorname{End}}%
\global\long\def\CMon{\operatorname{CMon}}%
\global\long\def\Sp{\operatorname{Sp}}%
\global\long\def\calg{\operatorname{CAlg}}%
\global\long\def\wcov{\operatorname{CAlg}^{\text{w.cov}}}%
\global\long\def\indwcov{\operatorname{CAlg}^{\text{ind(w.cov)}}}%
\global\long\def\cocalg{\operatorname{coCAlg}}%
\global\long\def\span{\operatorname{Span}}%
\global\long\def\cat{\operatorname{Cat}_{\infty}}%
\global\long\def\fun{\operatorname{Fun}}%
\global\long\def\nm{\operatorname{Nm}}%
\global\long\def\Id{\operatorname{Id}}%
\global\long\def\pt{\operatorname{pt}}%
\global\long\def\one{\mathds{1}}%
\global\long\def\bc{\operatorname{BC}}%
\global\long\def\red#1{\textcolor{red}{#1}}%
\global\long\def\white#1{\textcolor{white}{#1}}%
\global\long\def\can{\mathrm{can}}%
\global\long\def\tr{\operatorname{Tr}}%
\global\long\def\Ind{\operatorname{Ind}}%
\global\long\def\op{\text{\emph{op}}}%
\global\long\def\c{\colon}%
\global\long\def\aut{\operatorname{Aut}}%
\global\long\def\endo{\operatorname{End}}%
\global\long\def\alt{\operatorname{Alt}}%
\global\long\def\Fin{\operatorname{Fin}}%
\global\long\def\aug{\operatorname{Aug}}%
\global\long\def\pro{\operatorname{Pro}}%
\global\long\def\cts{\mathrm{cts}}%
\global\long\def\pb{\mathrm{pb}}%
\global\long\def\st{\mathrm{st}}%

\global\long\def\colim{\underrightarrow{\operatorname{lim}}\,}%
\global\long\def\lim{\underleftarrow{\operatorname{lim}\,}}%
\global\long\def\gal{\operatorname{gal}}%
\global\long\def\oplax{\operatorname{oplax}}%
\global\long\def\lax{\operatorname{lax}}%
\global\long\def\ev{\operatorname{ev}}%
\global\long\def\spec{\operatorname{Spec}}%
\global\long\def\LMod{\operatorname{LMod}}%
\global\long\def\BMod{\operatorname{BMod}}%
\global\long\def\RMod{\operatorname{RMod}}%
\global\long\def\Mod{\operatorname{Mod}}%
\global\long\def\alg{\operatorname{Alg}}%
\global\long\def\gal{\operatorname{Gal}}%
\global\long\def\pts{\operatorname{Pts}\,}%
\global\long\def\crng{\operatorname{CRing}\,}%
\global\long\def\set{\operatorname{Set}}%
\global\long\def\DM{\operatorname{DM}}%
\global\long\def\div{\operatorname{pDiv}}%
\global\long\def\wt{\operatorname{Wt}}%
\global\long\def\hopf{\operatorname{Hopf}}%
\global\long\def\fgrp{\operatorname{FGrp}}%
\global\long\def\cospec{\operatorname{coSpec}}%
\global\long\def\grpf{\operatorname{GrpFun}}%
\global\long\def\norm#1{||#1||}%
\global\long\def\sad#1{\scalebox{0.6}{\ensuremath{\oplus}}\text{-}#1}%
\global\long\def\psad#1{\scalebox{0.6}{\ensuremath{\oplus}}_{p}\text{-}#1}%
\global\long\def\hatotimes{\,\widehat{\otimes}\,}%
\global\long\def\Pic{\mathrm{Pic}}%
\global\long\def\GL{\mathrm{GL}}%
\global\long\def\supp{\mathrm{supp}}%
\global\long\def\et{\mathrm{et}}%
\global\long\def\pn#1{p_{(#1)}}%

\global\long\def\tsadi{\pretsadi}%

\global\long\def\acat{\mathrm{Cat}}%
\global\long\def\fin#1{#1\text{-}\mathrm{fin}}%
\global\long\def\htt{\mathrm{ht}}%
\global\long\def\Ht{\mathrm{Ht}}%
\global\long\def\MU{\text{\emph{MU}}}%
\global\long\def\cn{\mathrm{cn}}%
\global\long\def\mode{\mathrm{Mode}}%

\title{Ambidexterity and Height}
\author{Shachar Carmeli\thanks{Department of Mathematics, Weizmann Institute of Science.} \and Tomer M. Schlank\thanks{Einstein Institute of Mathematics, Hebrew University of Jerusalem.} \and Lior Yanovski\thanks{Max Planck Institute for Mathematics.}}

\maketitle
\begin{abstract}
We introduce and study the notion of \emph{semiadditive height} for
higher semiadditive $\infty$-categories, which generalizes the chromatic
height. We show that the higher semiadditive structure trivializes
above the height and prove a form of the redshift principle, in which
categorification increases the height by one. In the stable setting,
we show that a higher semiadditive $\infty$-category decomposes into
a product according to height, and relate the notion of height to semisimplicity properties of local systems. We place
the study of higher semiadditivity and stability in the general framework
of smashing localizations of $\Pr^{L}$, which we call \emph{modes}.
Using this theory, we introduce and study the universal stable $\infty$-semiadditive
$\infty$-category of semiadditive height $n$, and give sufficient conditions for
a stable $1$-semiadditive $\infty$-category to be $\infty$-semiadditive. 
\end{abstract}

\begin{figure}[H]
\centering{}\includegraphics[scale=0.174]{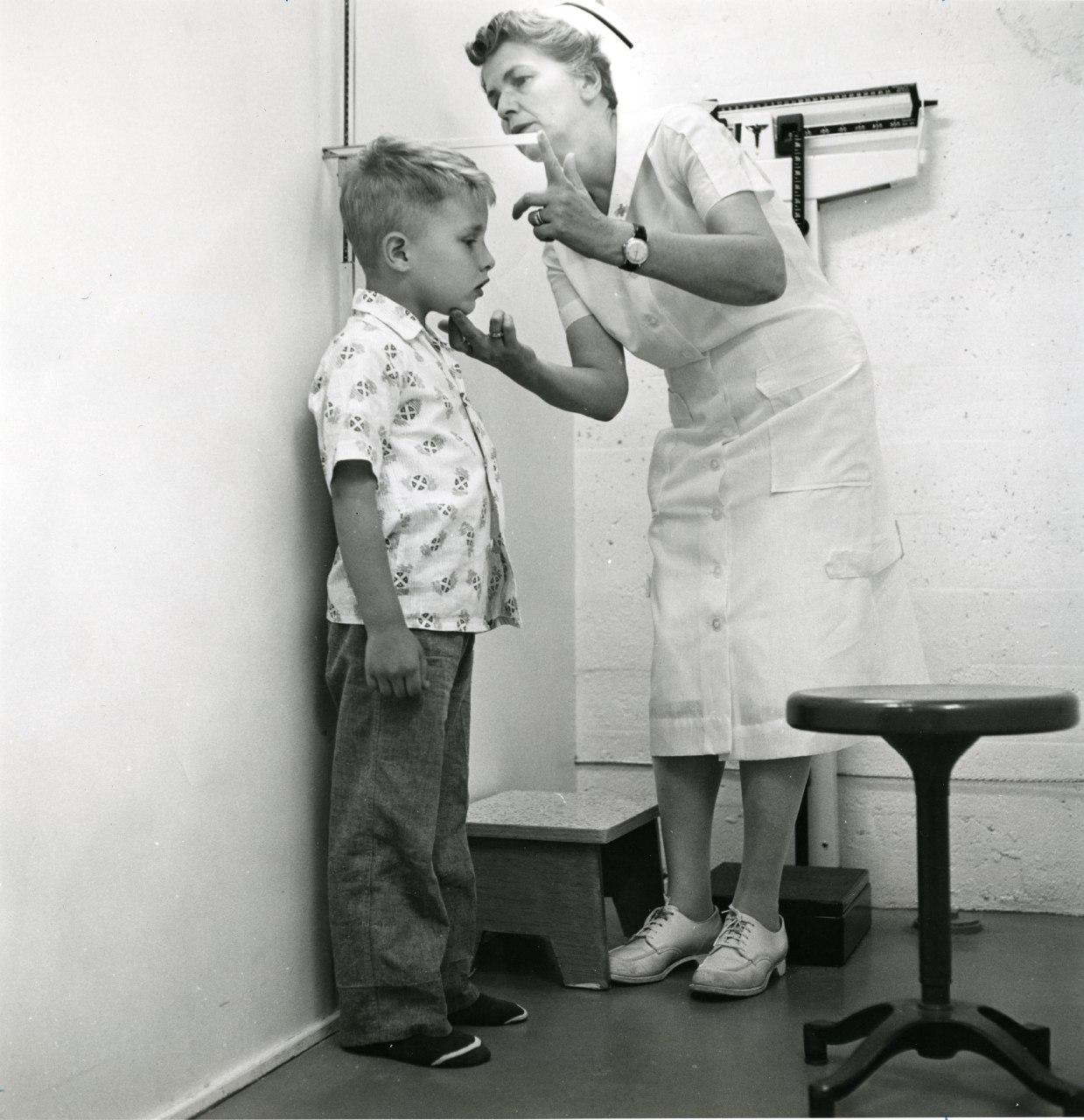}\caption*{Nurse measures height of child for Tecumseh Health study, 1959, HS15323, Alumni Association (University of Michigan) records, UM News Service, Bentley Historical Library, University of Michigan.
}
\end{figure}

\tableofcontents{}

\section{Introduction}

\subsection{Background \& Overview}

Chromatic homotopy theory springs from the deep and surprising connection
between the $\infty$-category of spectra and the stack of formal
groups. In particular, the height filtration on the latter is mirrored
by the ``chromatic height filtration'' on the former. This connection
begins with Quillen's work on the complex cobordism spectrum $\MU$,
showing that the ring $\pi_{*}\MU$ carries the universal formal group
law. Formal group laws admit a notion of ($p$-typical) height for
every prime $p$. This notion can be defined in terms of a certain
sequence of classes $v_{n}\in\pi_{2(p^{n}-1)}MU$ as follows: If $v_{0},\dots,v_{n-1}$
vanish, then the height is $\ge n$, and if $v_{n}$ is invertible,
then the height is $\le n$. This algebraic filtration has a spectrum
level manifestation in the form of the Morava $K$-theories $K(n)$, which
are certain $\MU$-algebras with the property $\pi_{*}K(n)\simeq\bb F_{p}[v_{n}^{\pm1}]$.
This suggests that the $K(n)$-s are concentrated at height exactly
$n$, and the corresponding Bousfield localizations $\Sp_{K(n)}\ss\Sp$
can then be considered as the ``monochromatic layers''\emph{ }of
the chromatic height filtration. The process of $K(n)$-localization
can be loosely thought of as completion with respect to $v_{0},\dots,v_{n-1}$
followed by the inversion of $v_{n}$. 

By the work of Hopkins, Devinatz, and Smith  (see \cite[Theorems 9 and 4.12]{nilp2}), the $v_{n}$-operations
can be inductively lifted to finite spectra, \emph{without} $\MU$-module
structure. More precisely, a finite $p$-local spectrum $F$ is said
to have type $n$, if $n$ is the lowest integer for which the $K(n)$-localization
of $F$ does not vanish. Given a type $n$ finite spectrum $F(n)$,
there exists a self map $\Sigma^{d}F(n)\to F(n)$, which induces a
power of $v_{n}$ on $K(n)$-homology. The cofiber of this self map
is then a type $(n+1)$ spectrum. This procedure allows us to construct a sequence
of $F(n)$-s as iterated Moore spectra, which we may suggestively
write as follows:
\[
F(n)\coloneqq\bb S/(p^{r_{0}},v_{1}^{r_{1}},\dots,v_{n-1}^{r_{n-1}}).
\]
Just as localization with respect to $\bb S/p^{r}$ (for any $r$)
has the effect of $p$-completion, one can think of the localization
with respect to $F(n)$, as completion with respect to $v_{0},\dots,v_{n-1}$.
Furthermore, localization with respect to the spectrum $T(n)=F(n)[v_{n}^{-1}]$,
can be thought of as completion with respect to $v_{0},\dots,v_{n-1}$,
followed by the inversion of $v_{n}$. It is known that the $K(n)$-localization
factors through the $T(n)$-localization, and that they coincide for
$\MU$-modules (and also in general when $n=0,1$). Furthermore, the
localizations $\Sp_{T(n)}$ turn out to be independent of all the
choices and thus naturally constitute another, potentially \emph{larger},
candidate for the ``monochromatic layers'' of the chromatic height
filtration. While the question of whether the inclusion $\Sp_{K(n)}\ss\Sp_{T(n)}$
is strict for $n\ge2$ is open (known as the Telescope Conjecture),
both candidates for the ``monochromatic layers'' play a pivotal
role in homotopy theory. 

The localizations $\Sp_{K(n)}$ and $\Sp_{T(n)}$ are known to possess
several rather special and remarkable properties. Among them, the
vanishing of the Tate construction for finite group actions (\cite{Kuhn,GState,HState,ClausenAkhil}).
In \cite{HopkinsLurie}, Hopkins and Lurie reinterpret this Tate vanishing property
as $1$\emph{-semiadditivity}, and vastly generalize it by showing
that the $\infty$-categories $\Sp_{K(n)}$ are $\infty$\emph{-semiadditive}.
In turn, this is exploited to obtain new structural results for $\Sp_{K(n)}$.
In \cite[Theorem B]{Ambi2018}, the authors extended on \cite{HopkinsLurie} by classifying all
the higher semiadditive localizations of $\Sp$ with respect to homotopy
rings. First, for all such localizations, $1$-semiadditivity was
shown to be equivalent to $\infty$-semiadditivity. Second, the telescopic
localizations $\Sp_{T(n)}$, for various primes $p$ and heights $n$,
were shown to be precisely the maximal examples of such localizations
(while the $\Sp_{K(n)}$ are the minimal). Concisely put, in the stable\emph{
}world, the higher semiadditive property singles out precisely the
monochromatic localizations, which are parameterized by the chromatic
height.

In this paper, we introduce a natural notion of\emph{ semiadditive
height} for higher semiadditive $\infty$-categories, which in the
examples $\Sp_{T(n)}$ and $\Sp_{K(n)}$ reproduces the usual chromatic
height $n$, without appealing to the theory of formal groups. We
then proceed to show that the semiadditive height is a fundamental
invariant of a higher semiadditive $\infty$-category, which controls
many aspects of its higher semiadditive structure, and the behavior
of local systems on $\pi$-finite spaces valued in it. We also show
that the semiadditive height exhibits a compelling form of the ``redshift
principle'', where categorification has the effect of increasing
the height exactly by one. When restricting to \emph{stable} $\infty$-categories,
we show that higher semiadditive $\infty$-categories decompose completely
according to the semiadditive height, which accounts for the monochromatic
nature of the higher semiadditive localizations of $\Sp$. Finally,
building on the work of Harpaz \cite{Harpaz}, we introduce and study universal constructions
of stable $\infty$-semiadditive $\infty$-categories of height $n$,
and initiate their comparison with the chromatic examples. 

The present work should be viewed as part of a more extensive program that
aims to place chromatic phenomena in the categorical context of the
interaction between higher semiadditivity and stability. Apart from
providing new tools for the study of $\Sp_{T(n)}$, we believe that this
approach can elucidate the chromatic picture and unfold
the rich and intricate structure hidden within. 

\subsection{Main Results}

\subsubsection{Height Theory}

Recall that \emph{ambidexterity} is a property of a space $A$ with
respect to an $\infty$-category $\mathcal{C}$, that allows ``integrating''
$A$-families of morphisms between pairs of objects in $\mathcal{C}$
in a canonical way \cite[Construction 4.0.7]{HopkinsLurie}. In particular, integrating the constant
$A$-family on the identity morphism of each object, produces a natural
endomorphism $|A|$ of the identity functor of $\mathcal{C}$. We
call $|A|$ the $\mathcal{C}$\emph{-cardinality} of $A$, and think
of it as multiplication by the ``size of $A$'' (the actual meaning of which depends on $\mathcal{C}$). 

An $\infty$-category $\mathcal{C}$ is called $m$-semiadditive if
every $m$-finite space is $\mathcal{C}$-ambidextrous. Our notion
of semiadditive height is defined in terms of cardinalities of such
spaces. For starters, let us begin with a $0$-semiadditive (i.e.
semiadditive) $p$-local $\infty$-category $\mathcal{C}$. If $p$
is \emph{invertible} in $\mathcal{C}$, then $\mathcal{C}$ is rational
and we consider it to be of ``height $0$''. In contrast, if all
objects of $\mathcal{C}$ are $p$\emph{-complete}, we consider it
to be of ``height $>0$''. To proceed, let us assume that $\mathcal{C}$
is $m$-semiadditive for some $m\ge0$. In such a case, we can consider
the $\mathcal{C}$-cardinalities of Eilenberg-MacLane spaces:
\[
p=|C_{p}|,|BC_{p}|,|B^{2}C_{p}|,\dots,|B^{m}C_{p}|.
\]
The definition of semiadditive height uses the maps $|B^{n}C_{p}|$
in a manner which is analogous to how the $v_{n}$-self maps are used
in the definition of the chromatic height:
\begin{defn*}[Semiadditive Height, \ref{def:Height_Obj}, \ref{def:Height_Cat}]
\label{def:Intro_Height}For every $0\le n\le m$, we write
\begin{enumerate}
\item $\Ht(\mathcal{C})\le n$, if $|B^{n}C_{p}|$ is invertible in $\mathcal{C}$.
\item $\Ht(\mathcal{C})>n$, if $\mathcal{C}$ is complete with respect
to $|C_{p}|,|BC_{p}|,\dots,|B^{n}C_{p}|$\footnote{by \propref{Height_Sense}, $\Ht(\mathcal{C})>n$, if and only if
$\mathcal{C}$ is $|B^{n}C_{p}|$-complete.}.
\item $\Ht(\mathcal{C})=n$, if $\mathcal{C}$ is of height $\le n$ and
$>n-1$.
\end{enumerate}
\end{defn*}

To show that the semiadditive height of $\Sp_{T(n)}$ and $\Sp_{K(n)}$
is indeed $n$, we need to get a handle on the $T(n)$-local and $K(n)$-local
cardinalities of the Eilenberg-Maclane spaces $B^{k}C_{p}$. In \cite[Lemma 5.3.3]{Ambi2018},
we have already shown that
\[
|B^{k}C_{p}| =p^{\binom{n-1}{k}}\quad\in\quad\pi_{0}E_{n},
\]
for the $\infty$-category of $K(n)$-local $E_{n}$-modules. Thus, this $\infty$-category is of height $n$. Since tensoring with $E_{n}$ is conservative on
$K(n)$-local spectra, this also readily implies that $\Sp_{K(n)}$
is of height $n$. However, to show that $\Sp_{T(n)}$ is of height
$n$, one has to know that the map $\pi_{0}\bb S_{T(n)}\to\pi_{0}\bb S_{K(n)}$,
induced by $K(n)$-localization, detects invertibility of elements.
This result was established in \cite[Propostion 5.1.17]{Ambi2018} using the notion of ``nil-conservativity''. Thus, we get that $\Sp_{T(n)}$ is of height $n$ as well. 

The notion of semiadditive height allows us to contextualize various
aspects of the $\infty$-categories $\Sp_{K(n)}$ and $\Sp_{T(n)}$
pertaining to the chromatic height. At the bottom of the hierarchy,
the $\infty$-category $\Sp_{T(0)}=\Sp_{K(0)}=\Sp_{\bb Q}$ can be
shown to be $\infty$-semiadditive by elementary arguments. This is
strongly related to the fact that all connected $\pi$-finite spaces
are $\bb Q$-acyclic and the cardinality of any (non-empty) $\pi$-finite
space $A$ is invertible. Thus, the higher semiadditive structure
of $\Sp_{\bb Q}$ is in a sense ``trivial''. The higher semiadditivity
of $\Sp_{K(n)}$ and $\Sp_{T(n)}$ for $n\ge1$ is more subtle precisely
because not all connected $\pi$-finite spaces are acyclic, and not
all cardinalities are invertible. One might roughly say, that the
complexity of the higher semiadditive structure grows with the height.
Our first main result formalizes this as follows:
\begin{theorem}[Bounded Height, \ref{thm:Height_Everything_p_Local}]
\label{thm:Intro_Bounded_Height}Let $\mathcal{C}$
be an $n$-semiadditive $p$-local $\infty$-category, which admits
all $\pi$-finite limits and colimits. If $\mathcal{C}$ is of height
$\le n$, then 
\begin{enumerate}
\item $\mathcal{C}$ is $\infty$-semiadditive. 
\item For every $(n-1)$-connected nilpotent $\pi$-finite space $A$, the
map $|A|$ is invertible.
\item For every $n$-connected $\pi$-finite space $A$ and $X\in\mathcal{C}$,
the fold map $A\otimes X\to X$ is invertible. 
\item For every principal fiber sequence of $\pi$-finite spaces 
\[
F\to A\to B,
\]
if $F$ is $(n-1)$-connected and nilpotent, then $|A|=|F|\cdot|B|$.
\end{enumerate}
\end{theorem}

Informally speaking, \thmref{Intro_Bounded_Height} states that the
invertibility of $|B^{n}C_{p}|$ has the effect of ``trivializing''
the higher semiadditive structure at levels $\ge n$. In particular,
it shows that it exists, which is point (1). From point (2), we deduce
that having height $\le n$ implies having height $\le n+1$, so the conditions are of decreasing strength as the terminology suggests.
Point (3) articulates a useful categorical consequence (and, in fact, a characterization) of having height $\le n$, which does not refer
directly to the higher semiadditive structure. This can be seen as
a generalization of \cite[Theorem E]{Ambi2018}, which is essentially the special case
$\mathcal{C}=\Sp_{T(n)}$. Finally, point (4) can be used to reduce
the computation of the $\mathcal{C}$-cardinalities of nilpotent $\pi$-finite
spaces to those of $n$-finite ones, under the assumption that $\mathcal{C}$
is of height $\le n$. The case $n=0$ produces an explicit formula,
which recovers Baez-Dolan's classical homotopy cardinality (\exaref{Homotopy_Cardinality}).
We note that the possible failure of point (4) for the principal fiber
sequence $B^{n-1}C_{p}\to\pt\to B^{n}C_{p}$, is precisely the obstruction
for $\mathcal{C}$ to have height $\le n-1$.

In their work on algebraic $K$-theory of ring spectra, Ausoni and
Rognes have discovered a phenomena which they dubbed ``chromatic
redshift''. Roughly speaking, it is the tendency of $K(R)$, which
is a spectrum constructed from the $\infty$-category of perfect $R$-modules,
to be of chromatic complexity larger by one, than the ring spectrum
$R$ (appropriately measured). While more precise conjectures regarding
this phenomena were subsequently formulated and studied, a conceptual
source for the chromatic redshift phenomena seems to remain unrevealed.
Our next result concerns an analogue of the redshift phenomena for
the \emph{semiadditive} height. In this context the increase by one
in height is a formal consequence of \emph{categorification}.\emph{
}To state this formally, we first note that the definition of semiadditive
height makes sense for individual objects. Namely, an object $X$
in an $\infty$-semiadditive $\infty$-category $\mathcal{C}$ is
of height $\le n$ for some $n$, if $|B^{n}C_{p}|$ acts invertibly
on $X$ and of height $>n$, if it is complete with respect to $|C_{p}|,|BC_{p}|,\dots,|B^{n}C_{p}|$.
Second, we exploit the fact that the $\infty$-category $\acat^{\sad{\infty}}$,
of $\infty$-semiadditive $\infty$-categories and $\pi$-finite colimit
preserving functors, is itself $\infty$-semiadditive. Thus, given
an $\infty$-semiadditive $\infty$-category $\mathcal{C}$, we can
consider the height of $\mathcal{C}$ being lower equal (resp. greater
than) $n$, as an \emph{object} of $\acat^{\sad{\infty}}$ which we
shall denote by $\htt(\mathcal{C})\le n$ (resp. $\htt(\mathcal{C})>n$).
\begin{theorem}
[Semiadditive Redshift, \ref{thm:Redshift}]\label{thm:Intro_Redshift}Let $\mathcal{C}$
be an $\infty$-semiadditive $\infty$-category. We have that $\Ht(\mathcal{C})\le n$
(resp. $\Ht(\mathcal{C})>n$), if and only if $\htt(\mathcal{C})$ $\le n+1$
(resp. $\htt(\mathcal{C}) >n+1$).
\end{theorem}

The higher semiadditive structure of $\acat^{\sad{\infty}}$ is essentially
given by taking colimits over $\pi$-finite spaces. Hence, \thmref{Intro_Redshift}
is closely related to point (3) of \thmref{Intro_Bounded_Height}.
As a concrete example, we can consider for a $T(n)$-local ring spectrum
$R$, the $\infty$-category of $T(n)$-local left $R$-modules. The
space of objects of this $\infty$-category is a commutative monoid
for the direct sum operation. Moreover, the higher semiadditivity
of the $\infty$-category of modules endows this space with a \emph{higher}
commutative monoid structure in the sense of \cite[Definition 5.10]{Harpaz}. As a consequence
of \thmref{Intro_Redshift}, this higher commutative monoid is of
height $\le n+1$ in the $\infty$-category of higher commutative
monoids. In a future work, we shall investigate the implications of
this to the chromatic redshift in algebraic $K$-theory in the sense
of Ausani-Rognes. 

Our main interest in the notion of higher semiadditivity is in its
application to \emph{stable} $\infty$-categories. As it turns out,
the two properties of higher semiadditivity and stability interact
in a highly non-trivial way. First and foremost, in the  presence of stability, the higher semiadditive structure turns out to decompose completely
according to height:
\begin{theorem}
[Height Decomposition, \ref{thm:Height_Decomposition}]\label{thm:Intro_Decomposition}Let $\mathcal{C}$
be a stable idempotent complete $m$-semiadditive $\infty$-category
for some $m\in\bb N$. There is a canonical equivalence
\[
\mathcal{C}\simeq\mathcal{C}_{0}\times\cdots\times\mathcal{C}_{m-1}\times\mathcal{C}_{>m-1},
\]
were $\mathcal{C}_{0},\dots,\mathcal{C}_{m-1}$ and $\mathcal{C}_{>m-1}$
are the full subcategories of objects of height $0,\dots,m-1$ and
$>m-1$ respectively.\footnote{We also treat the case $m=\infty$, which is somewhat more subtle.}
\end{theorem}

This result sheds light on the ``monochromatic nature'' of higher semiadditive
phenomena in the stable world. Loosely speaking, the fact that the
monochromatic layers, which have different heights, glue non-trivially
(by means of the chromatic fracture square), obstructs the higher
semiadditivity of non-monochromatic localizations of spectra. 

In view of \thmref{Intro_Decomposition}, it makes sense to focus
our attention on stable $\infty$-categories $\mathcal{C}$ of height
exactly $n$. In \cite[Section 5.4]{HopkinsLurie} it is shown that the behavior of local systems
of $K(n)$-local spectra on a $\pi$-finite space $A$, strongly depends
on the level of connectedness of $A$ compared with $n$. We show that
some of these results hold for general stable $\infty$-categories
$\mathcal{C}$ of height exactly $n$. First of all, from \thmref{Intro_Bounded_Height}(3),
it can be deduced that for an $n$-connected $\pi$-finite space $A$,
the inclusion functor $\mathcal{C}\into\mathcal{C}^{A}$ of constant
local-systems is \emph{fully faithful}. The right orthogonal complement
$\mathcal{C}^{\perp}\ss\mathcal{C}^{A}$ consists of local-systems
whose global sections object (i.e. limit over $A$) vanishes. We prove
the following:
\begin{theorem}
[Semisimplicity, \ref{thm:Height_Semisimple}]\label{thm:intro_Semisimple}Let $\mathcal{C}$ be
a stable $\infty$-semiadditive $\infty$-category such that  $\Ht(\mathcal{C}) =n$,
and let $A$ be an $n$-connected $\pi$-finite space. There is a
canonical equivalence $\mathcal{C}^{A}\simeq\mathcal{C}\times\mathcal{C}^{\perp}$.
\end{theorem}

This result can be seen as a generalization of the ``semisimplicity'' of $\Sp_{K(n)}$-valued local systems on
$n$-connected $\pi$-finite spaces (compere \cite{LurieRep}). We also provide
an explicit formula for the composition $\mathcal{C}^{A}\onto\mathcal{C}\into\mathcal{C}^{A}$,
as a ``symmetrization'' of the action of the $\infty$-group $G=\Omega A$.
Intuitively, the ``order'' of $G$, by which one has to divide, is
precisely the $\mathcal{C}$-cardinality of $G$, which is invertible
by the assumption on the height of $\mathcal{C}$ and the connectivity
of $A$ (\thmref{Intro_Bounded_Height}).

Based on the classification of higher semiadditive localizations of
$\Sp$ with respect to homotopy rings in \cite[Theorem B]{Ambi2018}, the authors proposed
the conjecture that every stable $p$-local presentable $1$-semiadditive
$\infty$-category is in fact $\infty$-semiadditive \cite[Conjecture 1.1.5]{Ambi2018}.
In this paper, we prove a partial result in the direction of this
conjecture. Given a stable $p$-local presentable $\infty$-category
$\mathcal{C}$, we say that an object $X\in\mathcal{C}$ is of \emph{finite
stable height} if there exists a non-zero finite $p$-local spectrum
$F$, such that $F\otimes X=0$. We also denote by $\mathcal{C}_{\infty^{\st}}\ss\mathcal{C}$
the full subcategory of objects $Y\in\mathcal{C}$, for which $\map(X,Y)=\pt$
for all $X$ of finite stable height. 
\begin{theorem}
[Bounded Bootstrap, \ref{thm:Bounded_Bootstrap}]\label{thm:Intro_Bounded_Bootstrap}Let $\mathcal{C}$
be a stable $p$-local presentable $\infty$-category. If $\mathcal{C}$
is $1$-semiadditive and $\mathcal{C}_{\infty^{\st}}=0$, then it
is $\infty$-semiadditive. Moreover, in this case $\mathcal{C}\simeq\prod_{n\in\bb N}\mathcal{C}_{n}$.
\end{theorem}

The condition $\mathcal{C}_{\infty^{\st}}=0$ is satisfied if for
example for every $X,Y\in\mathcal{C}$, the mapping spectrum $\hom(X,Y)$
is $L_{n}^{f}$-local for some integer $n$. The proof of \thmref{Intro_Bounded_Bootstrap},
relies on the theory of \emph{modes}, which we shall review next. 

\subsubsection{Mode Theory}

In \cite[Proposition 4.8.1.15]{ha}, Lurie introduced a symmetric monoidal structure on the
$\infty$-category $\Pr^{L}$ of presentable $\infty$-categories
and colimit preserving functors. Moreover, he showed that many familiar
properties of presentable $\infty$-categories can be characterized
as having a (necessarily unique) module structure over certain idempotent
algebras in $\Pr^{L}$ \cite[Section 4.8.2]{ha}. We call such idempotent presentable
$\infty$-categories \emph{modes}. This notion was also considered
in \cite{GepUniv} from the perspective of \emph{smashing} \emph{localizations}
of $\Pr^{L}$. Given a mode $\mathcal{M}$, it is a \emph{property}
of a presentable $\infty$-category $\mathcal{C}$ to have a \emph{structure}
of a module over $\mathcal{M}$. The terminology is inspired by the
idea that modes classify the possible ``modes of existence'' in
which mathematical objects can occur, manifest, and behave. Most notably,
the property of \emph{stability} is equivalent to having a module
structure over $\Sp$. Consequently, every stable presentable $\infty$-category
is canonically enriched in $\Sp$ and colimit preserving functors
between stable presentable $\infty$-categories preserve this enrichment.
This structure naturally plays a significant role in the study of
stable $\infty$-categories. In \cite[Lemma 5.20]{Harpaz}, Harpaz showed that $m$-semiadditivity
is similarly characterized by having a module structure over the idempotent
algebra $\CMon_{m}$ of $m$\emph{-commutative} \emph{monoids}. The
case $m=0$ recovers the usual $\infty$-category of commutative (i.e.
$\bb E_{\infty}$) monoids in spaces, which classifies ordinary semiadditivity.
The mapping spaces of an $m$-semiadditive $\infty$-category obtain
a canonical $m$-commutative monoid structure, by analogy with the
$\Sp$-enrichment of stable $\infty$-categories. 

In the final section of this paper, we develop the theory of modes further and apply it to the study of height in stable presentable
higher semiadditive $\infty$-categories. First, by the general theory
of modes, $\CMon_{m}\otimes\Sp$ is also a mode, which classifies
the property of being at the same time stable and $m$-semiadditive.
Furthermore, using \thmref{Intro_Decomposition}, we show:
\begin{theorem}[\ref{thm:Tsadi_n_Mode}]
\label{thm:Intro_Tsadi_n}For every $n\ge0$, there exists a mode $\tsadi_{n}$\footnote{The letter $\tsadi$ (pronounced ``tsadi'') is the first letter
in the Hebrew word for ``color''. The notation was chosen to indicate
the close relationship with chromatic homotopy theory.}, which classifies the property of being stable, $p$-local, $\infty$-semiadditive
and of height $n$.
\end{theorem}

It is natural to compare $\tsadi_{n}$ with $\Sp_{T(n)}$, which is
in a sense the universal $p$-local height $n$ localization of spectra.
Since $\Sp_{T(n)}$ is also $\infty$-semiadditive and of semiadditive
height $n$, the theory of modes implies the existence of a unique
colimit preserving symmetric monoidal functor $L\colon\tsadi_{n}\to\Sp_{T(n)}$.
In the case $n=0$, the functor $L\colon\tsadi_{0}\to\Sp_{T(0)}$
is an equivalence and hence $\tsadi_{0}\simeq\Sp_{\bb Q}$ (\exaref{Tsadi0}). In general,
we show that $L$ exhibits $\Sp_{T(n)}$ as a \emph{smashing} \emph{localization}
of $\tsadi_{n}$ in the sense that $L$ admits a fully faithful right
adjoint $\Sp_{T(n)}\into\tsadi_{n}$ and there is a canonical isomorphism
$LX\simeq\bb S_{T(n)}\otimes X$ for all $X\in\tsadi_{n}$ (\corref{Tsadi_Tn_Smashing}).
For $n\ge1$, the $\infty$-category $\tsadi_{n}$ also resembles
$\Sp_{T(n)}$ in that the unique colimit preserving symmetric monoidal
functor $\Sp\to\tsadi_{n}$ vanishes on all bounded above spectra
(\propref{Tsadi_Bounded_Vanishing}), and that the right adjoint of
the unique colimit preserving symmetric monoidal functor $\mathcal{S}\to\tsadi_{n}$
is conservative (\corref{Tsadi_Conservative}). We consider $\tsadi_{n}$
to be a natural extension of $\Sp_{T(n)}$, which is a universal home for phenomena of height $n$. 

In a previous draft of this paper, we proposed the conjecture that for every $n\ge0$, the unique colimit preserving symmetric monoidal functor $L\colon\tsadi_{n}\to\Sp_{T(n)}$ 
is an equivalence. However, this conjecture was soon \emph{disproved} by Allen Yuan already in the case $n=1$. More precisely, using the Segal Conjecture (now a theorem \cite{carlsson1984}), he has constructed a higher commutative monoid structure of height $1$ on the $p$-complete sphere, as an object of the $\infty$-category of $p$-complete spectra. The details and some interesting applications of this example will appear in a separate paper by him.  

Finally, the theory of modes allows us not only to analyze the implications
of certain properties of presentable $\infty$-categories, but also
to \emph{enforce} them in a universal way. For every mode $\mathcal{M}$
and a presentable $\infty$-category $\mathcal{C},$ we can view $\mathcal{M}\otimes\mathcal{C}$
as the universal approximation of $\mathcal{C}$ by a presentable
$\infty$-category which satisfies the property classified by $\mathcal{M}$.
For example, $\Sp\otimes\,\mathcal{C}\simeq\Sp(\mathcal{C})$ is the
stabilization of $\mathcal{C}$ \cite[Example 4.8.1.23]{ha}, and similarly, $\CMon_{m}\otimes\,\mathcal{C}$
is the ``$m$-semiadditivization'' of $\mathcal{C}$ \cite[Corollary  5.18]{Harpaz}.
As alluded to above, the non-trivial gluing in the chromatic fracture
square, prevents $L_{n}^{f}\Sp$ from being higher semiadditive for
$n\ge1$. Employing the additive $p$-derivation $\delta$ on the
rings $\pi_{0}\bb S_{T(n)}$ constructed in \cite[Section 4]{Ambi2018}, we show that
forcing even $1$-semiadditivity on $L_{n}^{f}\Sp$, has the effect
of ``dissolving the glue'' in the chromatic fracture squares:
\begin{theorem}
[1-Semiadditive Splitting, \ref{thm:1Sad_Decomposition}]\label{thm:Intro_Chrom_Decomp}For every
$n\ge0$, there is a unique equivalence of presentably symmetric monoidal
$\infty$-categories
\[
\CMon_{1}\otimes L_{n}^{f}\Sp\simeq\prod_{k=0}^{n}\Sp_{T(k)}.
\]
\end{theorem}

In particular, we see that forcing $1$-semiadditivity on $L_{n}^{f}\Sp$
makes it automatically $\infty$-semiadditive. Noticing that both
sides of \thmref{Intro_Chrom_Decomp} are \emph{modes}, we can reinterpret
it in terms of the properties classified by them. Namely, that every
$1$-semiadditive stable presentable $\infty$-category whose mapping
spectra are $L_{n}^{f}$-local, is $\infty$-semiadditive. With some additional effort, we deduce from it the stronger statement of \thmref{Intro_Bounded_Bootstrap}. 

\subsection{Organization}

We shall now outline the content of each section of the paper.

In section 2, we recall and expand the theory of higher semiadditivity.
We discuss the notion of cardinality for a $\pi$-finite space in
a higher semiadditive $\infty$-category and the corresponding notion
of amenability. We then give several examples of these notions in
various higher semiadditive $\infty$-categories, and relate the notion
of amenability to the behavior of local systems, through
the notion of acyclicity. 

In section 3, we discuss the main notion of this paper, that of \emph{height}
in a higher semiadditive $\infty$-category, defined in terms of the
cardinalities of Eilenberg-MacLane spaces. We show that the higher
semiadditive structure trivializes above the height (\thmref{Intro_Bounded_Height})
and exhibits a redshift principle of increasing by one under categorification
(\thmref{Intro_Redshift}).

In section 4, we study semiadditivity and height for \emph{stable} $\infty$-categories.
After a general discussion on recollement, we show that a stable higher
semiadditive $\infty$-category splits as a product according to height
(\thmref{Intro_Decomposition}). We then study local systems valued
in a stable higher semiadditive $\infty$-category of height $n$ and show
how the notion of height is related to the phenomenon of semisimplicity
of local systems (\thmref{intro_Semisimple}). Finally, we use nil-conservative functors to show that semiadditive and chromatic height coincide for monochromatic localizations of spectra.

In section 5, we study the theory of modes, i.e. that of idempotent
algebras in the category of presentable $\infty$-categories. We show
how algebraic operations on modes, such as tensor product and localization,
translate into operations on the properties of presentable $\infty$-categories
classified by them. We then show that the main notions studied in
this paper, higher semiadditivity and height, together with the more
classical notion of chromatic height, are all encoded by modes (e.g.
\thmref{Intro_Tsadi_n}). Using this theory, we study the interaction
between the chromatic and the semiadditive heights through the interactions
between the corresponding modes. In particular, we prove \thmref{Intro_Chrom_Decomp}
and deduce from it \thmref{Intro_Bounded_Bootstrap}. 

\subsection{Conventions}

Throughout the paper, we work in the framework of $\infty$-categories
(a.k.a. quasicategories), and in general follow the notation of \cite{htt}
and \cite{ha}. We shall also use the following terminology and notation
most of which is consistent with \cite{Ambi2018}:
\begin{enumerate}
\item We slightly diverge from  \cite{htt} and \cite{ha} in the following points:
\begin{enumerate}
\item We use the term \emph{isomorphism} for an invertible morphism in an
$\infty$-category (i.e. equivalence).
\item We denote by $\mathcal{C}^{\simeq}\ss\mathcal{C}$ the maximal $\infty$-subgroupoid
of an $\infty$-category $\mathcal{C}$. 
\item We write $\Pr$ for the $\infty$-category of presentable $\infty$-categories
and colimit preserving functors denoted in \cite{htt} by $\Pr^{L}$.
\item We denote by $\acat_{\st} \subset \acat$  the subcategory spanned by stable $\infty$-categories and exact functors.
Similarly,  we denote by $\Pr_{\st} \subseteq \Pr$  the full subcategory spanned by stable presentable $\infty$-categories.
\end{enumerate}
\item We say that a space $A\in\mathcal{S}$ is 
\begin{enumerate}
\item $m$-finite for $m\ge-2$, if $m=-2$ and $A$ is contractible, or
$m\ge-1$, the set $\pi_{0}A$ is finite and all the fibers of the
diagonal map $\Delta\colon A\to A\times A$ are $(m-1)$-finite\footnote{For $m\ge0$, this is equivalent to $A$ having finitely many components,
each of them $m$-truncated with finite homotopy groups.}. 
\item $\pi$-finite or $\infty$-finite, if it is $m$-finite for some integer
$m\ge-2$. For $-2\le m\le\infty$, we denote by $\mathcal{S}_{\fin m}\ss\mathcal{S}$
the full subcategory spanned by $m$-finite spaces.
\item $p$-space, if all the homotopy groups of $A$ are $p$-groups.
\end{enumerate}
\item Given an $\infty$-category $\mathcal{C}\in\cat$,
\begin{enumerate}
\item For every map of spaces $A\oto qB$, we write $q^{*}\colon\mathcal{C}^{B}\to\mathcal{C}^{A}$
for the pullback functor and $q_{!}$ and $q_{*}$ for the left and
right adjoints of $q^{*}$ whenever they exist.
\item Whenever convenient we suppress the canonical equivalence of $\infty$-categories
$\mathcal{S}_{/\pt}\iso\mathcal{S}$ by identifying a space $A$ with
the terminal map $A\oto q\pt$. In particular, for every $\infty$-category
$\mathcal{C}$, we write $A^{*}$ for $q^{*}$ and similarly $A_{!}$
and $A_{*}$ for $q_{!}$ and $q_{*}$ whenever they exist.
\item For every $X\in\mathcal{C}$ we write $X[A]$ for $A_{!}A^{*}X$ and
denote the fold (i.e. counit) map by $X[A]\oto{\nabla}X$. Similarly,
we write $X^{A}$ for $A_{*}A^{*}X$ and denote the diagonal (i.e.
unit) map by $X\oto{\Delta}X^{A}$.
\end{enumerate}
\item Given a map of spaces $q\colon A\to B$, we denote for every $b\in B$,
the homotopy fiber of $q$ over $b$ by $q^{-1}(b)$. We say that 
\begin{enumerate}
\item an $\infty$-category $\mathcal{C}$ admits all $q$-limits (resp.
$q$-colimits) if it admits all limits (resp. colimits) of shape $q^{-1}(b)$
for all $b\in B$. 
\item a functor $F\colon\mathcal{C}\to\mathcal{D}$ preserves $q$-colimits
(resp. $q$-limits) if it preserves all colimits (resp. limits) of
shape $q^{-1}(b)$ for all $b\in B$.
\end{enumerate}
\item For every $-2\le m\le\infty$,
\begin{enumerate}
\item by $m$-finite (co)limits we mean (co)limits indexed by an $m$-finite space. 
\item We let $\acat_{m\text{-fin}}\subset\cat$ (resp. $\acat^{m\text{-fin}}\subset\cat$)
be the subcategory spanned by $\infty$-categories which admit $m$-finite
colimits (resp. limits) and functors preserving them. 
\item For $\mathcal{C},\mathcal{D}\in\acat_{m\text{-fin}}$ (resp. $\acat^{m\text{-fin}}$)
we write $\fun_{\fin m}(\mathcal{C},\mathcal{D})$ (resp. $\fun^{\fin m}(\mathcal{C},\mathcal{D})$)
for the full subcategory of $\fun(\mathcal{C},\mathcal{D})$ spanned
by the $m$-finite colimit (resp. limit) preserving functors.
\item We let $\acat^{\sad m} \subset \cat$ be the subcategory  spanned by the
$m$-semiadditive $\infty$-categories and $m$-semiadditive (i.e. $m$-finite colimit preserving) functors.
\item Given an $\infty$-operad $\mathcal{O}$, we say that $\mathcal{C}\in\alg_{\mathcal{O}}(\cat)$ is compatible with $\mathcal{K}$-indexed colimits for some collection of $\infty$-categories $\mathcal{K}$, if $\mathcal{C}$ admits $\mathcal{K}$-indexed colimits and every tensor operation $\otimes\colon\mathcal{C}^{n}\to\mathcal{C}$
of $\mathcal{O}$ preserves $\mathcal{K}$-indexed colimits in each variable. 

\item An $m$-semiadditively $\mathcal{O}$-monoidal
$\infty$-category is an $\mathcal{O}$-monoidal $m$-semiadditive
$\infty$-category which is compatible with $m$-finite colimits.  
\end{enumerate}
\item If $\mathcal{C}$ is a monoidal $\infty$-category and $\mathcal{D}$
is an $\infty$-category enriched in $\mathcal{C}$, we write $\hom_{\mathcal{D}}^{\mathcal{C}}(X,Y)$
for the $\mathcal{C}$-mapping object of $X,Y\in\mathcal{D}$. We
omit the subscript or superscript when they are understood from the
context. In particular, when $\mathcal{C}$ is closed, $\hom_{\mathcal{C}}(X,Y)$
means $\hom_{\mathcal{C}}^{\mathcal{C}}(X,Y)$. For every $\infty$-category $\mathcal{C}$ we have  $\hom_{\mathcal{C}}^{\mathcal{S}}(X,Y) = \map_{\mathcal{C}}(X,Y)$.
\end{enumerate}

\subsection{Acknowledgments}

We would like to thank Tobias Barthel, Clark Barwick,  Agn{\`e}s Beaudry,
Jeremy Hahn, Gijs Heuts, Mike Hopkins, and Tyler Lawson for useful discussions. We thank Allen Yuan for sharing with us his ideas regarding the topics of this paper. We also like to thank the entire Seminarak group, especially Shay Ben Moshe, for useful comments on the paper's first draft. 

The first author is supported by the Adams Fellowship Program of the Israel
Academy of Sciences and Humanities. 
The second author is supported by ISF1588/18 and BSF 2018389.  

\section{Semiadditivity}

In this section, we collect general facts regarding the notion of ambidexterity
and its implications. We begin by reviewing some background material
and most importantly (re)introduce the notion of \emph{cardinality}
for ambidextrous $\pi$-finite spaces. We provide a variety of examples
of cardinality in both the stable and the unstable settings, including those
of relevance to chromatic homotopy theory. Of particular importance are the \emph{amenable }spaces, whose cardinality is invertible. We
continue the study of such spaces, which we began in \cite[Section 3]{Ambi2018}, and
in particular, establish its implications for the behavior of certain
$\pi$-finite limits and colimits. In the next section, the amenability
of Eilenberg-MacLane spaces will play a central role in the definition
of ``semiadditive height'' for higher semiadditive $\infty$-categories,
which is the main subject of this paper. 

\subsection{Preliminaries}

In this subsection, we review some basic definitions and facts regarding
ambidexterity from \cite[Section 4]{HopkinsLurie}, cardinality from \cite{Ambi2018}, and higher
commutative monoids from  \cite[Section 5.2]{Harpaz}. This subsection serves mainly
to set up notation, terminology, and a convenient formulation of
fundamental results. 

\subsubsection{Ambidexterity}

Recall from \cite[Section 4.1]{HopkinsLurie} the definition of\emph{ }ambidexterity:
\begin{defn}
\label{def:Ambidextrous}Let $\mathcal{C}\in\cat$. A $\pi$-finite
map $A\oto qB$ is called:
\begin{enumerate}
\item \textbf{weakly} $\mathcal{C}$\textbf{-ambidextrous} if it is an isomorphism,
or $A\oto{\Delta_{q}}A\times_{B}A$ is $\mathcal{C}$-ambidextrous.
\item $\mathcal{C}$\textbf{-ambidextrous} if it is weakly $\mathcal{C}$-ambidextrous,
$\mathcal{C}$ admits all $q$-limits and $q$-colimits and the the
norm map $q_{!}\oto{\nm_{q}}q_{*}$ is an isomorphism. 
\end{enumerate}
\end{defn}

\defref{Ambidextrous} should be understood inductively on the level
of truncatedness of $q$. A $(-2)$-finite map, i.e. an isomorphism,
is always $\mathcal{C}$-ambidextrous. If $q$ is $m$-finite, then
the diagonal map 
\[
A\oto{\Delta_{q}}A\times_{B}A
\]
is $(m-1)$-finite and the ambidexterity of $\Delta_{q}$ allows in
turn the definition of $\nm_{q}$ by \cite[Construction 4.1.8]{{HopkinsLurie}} (see also \cite[Definition  3.1.3]{Ambi2018}). 
\begin{rem}
\label{rem:Ambi_Space}A map $A\oto qB$ is $\mathcal{C}$-ambidextrous
if and only if all the fibers of $q$ are $\mathcal{C}$-ambidextrous
spaces \cite[Corollary 4.2.6(2), Corollary 4.3.6]{HopkinsLurie}. Moreover, the fibers of the diagonal $A\to A\times A$
are the path spaces of $A$. In other words, $A$ is weakly $\mathcal{C}$-ambidextrous
if and only if the path spaces of $A$ are $\mathcal{C}$-ambidextrous.
Thus, $\mathcal{C}$-ambidexterity is ultimately a property of \emph{spaces}.
\end{rem}

By \cite{HopkinsLurie}, the property of ambidexterity has the following useful
characterization, which avoids the explicit inductive construction
of the norm map:
\begin{prop}
\label{prop:Ambi_Criterion}Let $\mathcal{C}$ be an $\infty$-category
and let $A\oto qB$ be a $\pi$-finite map. The map $q$ is $\mathcal{C}$-ambidextrous
if and only if the following hold:
\begin{enumerate}
\item $q$ is weakly $\mathcal{C}$-ambidextrous.
\item $\mathcal{C}$ admits all $q$-limits and $q$-colimits.
\item Either $q_{*}$ preserves all $q$-colimits or $q_{!}$ preserves
all $q$-limits.
\end{enumerate}
\end{prop}

\begin{proof}
By \remref{Ambi_Space}, we may assume that $B=\pt$, in which case
it is essentially \cite[Proposition 4.3.9]{HopkinsLurie} and its dual \cite[Remark 4.3.10]{HopkinsLurie}. We note that while
the claim in \cite{HopkinsLurie} is stated under the stronger assumption that
$\mathcal{C}$ admits, and $q_{*}$ (resp. $q_{!}$) preserves, all
small colimits (resp. limits), the proof uses only $q$-colimits (resp.
$q$-limits).
\end{proof}
As a consequence, we can easily deduce that ambidexterity enjoys the
following closure properties with respect to the $\infty$-category:
\begin{prop}
\label{prop:Ambi_Closure}Let $\mathcal{C}$ be an $\infty$-category
and let $A$ be a $\pi$-finite $\mathcal{C}$-ambidextrous space. The space
$A$ is also $\mathcal{D}$-ambidextrous for:
\begin{enumerate}
\item $\mathcal{D}=\mathcal{C}^{\op}$ the opposite $\infty$-category of
$\mathcal{C}$.
\item $\mathcal{D}=\fun(\mathcal{I},\mathcal{C})$ for an $\infty$-category $\mathcal{I}$.
\item $\mathcal{D}\ss\mathcal{C}$ containing the final object and closed
under $\Omega_{a}^{k}A$-limits for all $a\in A$ and $k\ge0$.
\item $\mathcal{D}\ss\mathcal{C}$ containing the initial object and closed
under $\Omega_{a}^{k}A$-colimits for all $a\in A$ and $k\ge0$.
\end{enumerate}
\end{prop}

\begin{proof}
First, (4) follows from (3) and (1), so it suffices to consider (1)-(3).
In all cases, we proceed by induction on $m$, so we may assume by
induction that $A$ is weakly $\mathcal{D}$-ambidextrous. By \propref{Ambi_Criterion},
it suffices to verify that $A$-limits and $A$-colimit exist in $\mathcal{D}$
and that the functor $A_{*}$ preserves $A$-colimits or that $A_{!}$
preserves $A$-limits. For (1), the claim follows from the fact that
limits in $\mathcal{C}^{\op}$ are computed as colimits in $\mathcal{C}$
and vice versa. For (2), we use the fact that limits and colimits
in $\fun(\mathcal{I},\mathcal{C})$ are computed pointwise. For (3), since $A$-limits
in $\mathcal{C}$ coincide with $A$-colimits in $\mathcal{C}$, it
follows that $A$-limits and $A$-colimits are computed in $\mathcal{D}$
in the same way as in $\mathcal{C}$.
\end{proof}

\subsubsection{Cardinality}

The main feature of ambidexterity is that it allows us to integrate
families of morphisms in $\mathcal{C}$. That is, given a $\mathcal{C}$-ambidextrous
map $A\oto qB$ and $X,Y\in\mathcal{C}^{B}$ we have a map (see \cite[Definition 2.1.11]{Ambi2018})
\[
\int_{q}\colon\map_{\mathcal{C}}(q^{*}X,q^{*}Y)\to\map_{\mathcal{C}}(X,Y).
\]
When $B=\pt$, we can think of an element of $\map_{\mathcal{C}}(q^{*}X,q^{*}Y)$
as a map $A\oto f\map_{\mathcal{C}}(X,Y)$, and of $\int_{A}f\in\map_{\mathcal{C}}(X,Y)$
as the sum of $f$ over the points of $A$. In particular, we can
integrate the \emph{identity} morphism:
\begin{defn}
\label{def:Cardinality}Let $\mathcal{C}\in\cat$ and let $A\oto qB$
be a $\mathcal{C}$-ambidextrous map. We have a natural transformation
$\Id_{\mathcal{C}^{B}}\oto{|q|_{\mathcal{C}}}\Id_{\mathcal{C}^{B}}$
given by the composition 
\[
\Id_{\mathcal{C}^{B}}\oto{u_{*}}q_{*}q^{*}\oto{\nm_{q}^{-1}}q_{!}q^{*}\oto{c_{!}}\Id_{\mathcal{C}^{B}}.
\]
For a $\mathcal{C}$-ambidextrous space $A$, we write $\Id_{\mathcal{C}}\oto{|A|_{\mathcal{C}}}\Id_{\mathcal{C}}$
and call $|A|_{\mathcal{C}}$ the $\mathcal{C}$\textbf{-cardinality}
of $A$.
\end{defn}

The name ``\emph{$\mathcal{C}$}-cardinality''\emph{ }can be explained
as follows. For a given object $X\in\mathcal{C}$, the map $X\oto{|A|_{X}}X$
equals $\int_{A}\Id_{X}$. Thus, we think of $|A|_{X}$ as the sum
of the identity of $X$ with itself ``$A$ times''. Or in other
words, as the result of multiplying by the ``cardinality of $A$''
on $X$. The basic example which motivates the terminology and notation
is the following:
\begin{example}
\label{exa:Finite_Cardinality}Let $\mathcal{C}$ be a semiadditive
$\infty$-category. For a finite set $A$, viewed as a $0$-finite
space, the operation $|A|_{\mathcal{C}}$ is simply the multiplication
by the natural number, which is the cardinality of $A$ in the usual sense. 
\end{example}

\begin{rem}
\label{rem:Cardinality_Fiberwise}For a general $\mathcal{C}$-ambidextrous
map $A\oto qB$, the transformation $|q|_{\mathcal{C}}$ can be understood
as follows. Let $B\oto X\mathcal{C}$ be a $B$-family of objects
in $\mathcal{C}$. For each $b\in B$, let us denote the fiber of
$q$ at $b$ by $A_{b}$ and the evaluation of $X$ at $b$ by $X_{b}$.
By \cite[Proposition 3.1.13]{Ambi2018}, the map $X\oto{|q|_{\mathcal{C}}}X$ acts on each $X_{b}$
by $|A_{b}|_{\mathcal{C}}$. In other words, $|q|_{\mathcal{C}}$
is just the $B$-family of $\mathcal{C}$-cardinalities of the $B$-family
of spaces $A_{b}$.
\end{rem}

For a $\mathcal{C}$-ambidextrous space $A$, the $A$-limits and
$A$-colimits in $\mathcal{C}$ are canonically isomorphic. This can
be used to show the following:
\begin{prop}
\label{prop:A_Semiadd}Let $\mathcal{C},\mathcal{D}\in\cat$ and let
$A$ be a $\mathcal{C}$- and $\mathcal{D}$-ambidextrous space. A
functor $F\colon\mathcal{C}\to\mathcal{D}$ preserves all $A$-limits
if and only if it preserves all $A$-colimits. Moreover, if $F$ preserves
all $A$-(co)limits, then $F(|A|_{\mathcal{C}})=|A|_{\mathcal{D}}.$
\end{prop}

\begin{proof}
\cite[Corollary 3.2.4]{Ambi2018}  shows that $F$ preserves $m$-finite limits if and only
if it preserves $m$-finite colimits and \cite[Corollary 3.2.7]{Ambi2018} shows that in this
case $F(|A|_{\mathcal{C}})=|A|_{\mathcal{D}}$ for every $m$-finite
space $A$. One easily checks that all the arguments are valid for an individual
space $A$ as well. 
\end{proof}
The $\mathcal{C}$-cardinality is additive in the following sense:
\begin{prop}
\label{prop:Cardinality_Additivity}Let $\mathcal{C}\in\cat$ and
$A\oto qB$ be a map of spaces. If $B$ and $q$ are $\mathcal{C}$-ambidextrous
then $A$ is $\mathcal{C}$-ambidextrous and for every $X \in \mathcal{C}$,
\[
|A|_{X}=\int_{B}|q|_{B^*X}.
\]
\end{prop}

Informally, \propref{Cardinality_Additivity} says that the cardinality
of the total space $A$ is the ``sum over $B$'' of the cardinalities
of the fibers $A_{b}$ of $q$. 
\begin{proof}
This follows from the Higher Fubini's Theorem (\cite[Propostion 2.1.15]{Ambi2018}) applied
to the identity morphism. 
\end{proof}
\begin{rem}
\label{rem:Cardinality_Arithmetic}By \propref{Cardinality_Additivity}
and \remref{Cardinality_Fiberwise}, for every pair of $\mathcal{C}$-ambidextrous
spaces $A$ and $B$ we have 
\[
|A\times B|_{\mathcal{C}}=|A|_{\mathcal{C}}|B|_{\mathcal{C}} \quad \in \quad \End(\Id_{\mathcal{C}}).
\]
Similarly, by \cite[Remark 4.4.11]{HopkinsLurie}, for every $\mathcal{C}$-ambidextrous
space $A$ we have 
\[
|A|_{\mathcal{C}}=\sum_{a\in\pi_{0}A}|A_{a}|_{\mathcal{C}} \quad \in \quad \End(\Id_{\mathcal{C}}) .
\]
\end{rem}

The various naturality properties enjoyed by the operations $|A|_{\mathcal{C}}$
allow for useful abuses of notation:
\begin{enumerate}
\item Given an $A$-colimit preserving functor $F\colon\mathcal{C}\to\mathcal{D}$,
if $A$ is $\mathcal{C}$- and $\mathcal{D}$-ambidextrous, we get
by \propref{A_Semiadd} that $F(|A|_{\mathcal{C}})=|A|_{\mathcal{D}}$.
We therefore write $|A|$, dropping the subscript $\mathcal{C}$,
whenever convenient. As a special case, let $\mathcal{C}_{\circ}\ss\mathcal{C}$
be a full subcategory which is closed under $A$-(co)limits. The cardinality $|A|_{\mathcal{C}_{\circ}}$ coincides with the restriction of $|A|_{\mathcal{C}}$ to $\mathcal{C}_{\circ}$.

\item When $\mathcal{C}$ is monoidal and the tensor product preserves $A$-colimits
in each variable, the action of $|A|$ on any object $X\in\mathcal{C}$
can be identified with tensoring $\one\oto{|A|}\one$ with $X$ (see
\cite[Lemma 3.3.4]{Ambi2018}). We therefore sometimes identify $|A|$ with an element
of $\pi_{0}\one\coloneqq\pi_{0}\map(\one,\one)$. 
\item Furthermore, for $R\in\alg(\mathcal{C})$, the map $R\oto{|A|}R$
can be also identified with multiplication by the image of $|A|\in\pi_{0}\one$
under the map $\pi_{0}\one\to\pi_{0}R$, which we also denote by $|A|$. 
\end{enumerate}
All these abuses of notation are compatible with standard conventions
when $A$ is a finite set (see \exaref{Finite_Cardinality}). 

\subsubsection{Higher commutative monoids}

Of particular interest are $m$\textbf{-semiadditive} $\infty$-categories,
i.e. those for which all $m$-finite spaces are ambidextrous. For
$m=0$, we recover the ordinary notion of a semiadditive $\infty$-category.
The central feature of semiadditive $\infty$-categories is the existence
of a canonical summation operation on their spaces of morphisms, endowing
them with a commutative monoid structure. In  \cite[Section  5.2]{Harpaz}, an analogous theory of $m$-commutative monoids is developed and applied
to the study of $m$-semiadditivity for all $-2 \leq m < \infty$. In this section, we recall from \cite{Harpaz} a part of this theory of higher commutative monoids and extend it to the case $m=\infty$.

\begin{defn}{\cite[Definition  5.10, Proposition 5.14]{Harpaz}}\label{def:cmon_fin}
Let $-2 \leq m < \infty $, for $\mathcal{C}\in\acat^{\fin m}$, the $\infty$-category of $m$\textbf{-commutative}
\textbf{monoids} in $\mathcal{C}$ is given by
\[
\CMon_{m}(\mathcal{C})\coloneqq\fun^{m\text{-fin}}(\span(\mathcal{S}_{\fin m})^{\op},\mathcal{C}).
\]
In the case $\mathcal{C}=\mathcal{S}$, we simply write $\CMon_{m}$ and refer
to its objects as $m$-commutative monoids\footnote{For $m=0$, one indeed recovers the usual notion of a commutative
(i.e. $\bb E_{\infty}$) monoids in spaces by comparison with Segal
objects \cite[Section 2.4.2]{ha}.}.
\end{defn}
In the case $m=-2$, evaluating at $\pt$, the unique object of $\span(\mathcal{S}_{\fin {(-2)}})$, gives an equivalence  $\CMon_{-2}(\mathcal{C}) \simeq \mathcal{C}$.

\begin{rem}
We can understand the definition of $\CMon_{m}$ as follows. An object
$X\in\CMon_{m}$ consists of the ``underlying space'' $X(\pt)$,
together with a collection of coherent operations for summation of
$m$-finite families of points in it. Indeed, for every $A\in\mathcal{S}_{\fin m}$,
there is a canonical equivalence $X(A)\simeq X(\pt)^{A}$. Given $A\to B$
in $\mathcal{S}_{\fin m}$, the ``right way'' map $X^{B}\to X^{A}$
is given simply by restriction, while the ``wrong way'' map $X^{A}\to X^{B}$
encodes integration along the fibers. The functoriality with respect
to composition of spans encodes the coherent associativity and commutativity
of these integration operations.
\end{rem}

Higher commutative monoids of different levels are related by ``forgetful functors''.
\begin{prop}\label{prop:Rm_CMon}
Let $-2 \leq m < \infty $ and let  $\mathcal{C}\in\acat^{\fin {(m+1)}}$.
 The restriction along the inclusion functor $$\iota_m\colon\span(\mathcal{S}_{\fin m}) \hookrightarrow \span(\mathcal{S}_{\fin {(m+1)}})$$ induces a limit preserving  functor 
$$\iota_m^*\colon \CMon_{m+1}(\mathcal{C}) \to \CMon_{m}(\mathcal{C}).$$
\end{prop}
\begin{proof}
It suffices to show that the $\iota_m$ preserves $m$-finite colimits. By \cite[Corollary 2.16]{Harpaz}, it suffices  to show  that  the composite $$F\colon \mathcal{S}_{\fin m} \to  \span(\mathcal{S}_{\fin m}) \hookrightarrow \span(\mathcal{S}_{\fin {(m+1)}})$$ preserves $m$-finite colimits. Indeed,  $F$ factors also as the composite $$\mathcal{S}_{\fin m} \hookrightarrow \mathcal{S}_{\fin {(m+1)}}   \to \span(\mathcal{S}_{\fin {(m+1)}}).$$ The first functor clearly preserves  $m$-finite colimits while the second one preserves  $m$-finite colimits by \cite[Proposition 2.12]{Harpaz}. \end{proof}

We now extend the definition of $\CMon_{m}$ to $m=\infty$. 

\begin{defn}
	 For $\mathcal{C}\in \acat^{\fin \infty}$, we denote
	\[
	\CMon_{\infty}(\mathcal{C})\coloneqq \lim_{m} \CMon_m(\mathcal{C}).
	\footnote{In fact one can define $\CMon_{\infty}(\mathcal{C})$ by substituting $m=\infty$ in  \defref{cmon_fin} and obtain an equivalent notion (we shall not use this fact).}
	\]
	As above we write $\CMon_{\infty}:= \CMon_{\infty}(\mathcal{S})$. 
\end{defn}

When $\mathcal{C}$ is presentable, $  \CMon_{m}(\mathcal{C}) $ is  presentable for all $m$, by \cite[Lemma 5.17]{Harpaz}. Moreover,  $\CMon_{\infty}(\mathcal{C})$ can then be described as the colimit of $\  \CMon_{m}(\mathcal{C}) $ in $\Pr$.

\begin{lem}\label{lem:CMon_Pr}
For $\mathcal{C} \in \Pr$,  the forgetful functors 
$$\iota_m^* \colon \CMon_{m+1}(\mathcal{C}) \to \CMon_{m}(\mathcal{C}),$$
admit left adjoints and  the colimit of the sequence   
\[
\mathcal{C}  \simeq \CMon_{-2} \xrightarrow{} \CMon_{-1}(\mathcal{C}) \xrightarrow{} \cdots \to \CMon_{m}(\mathcal{C}) \xrightarrow{} \cdots 
\]
in $\Pr$ is $\CMon_{\infty}(\mathcal{C}) $. In particular, $\CMon_{\infty}(\mathcal{C}) $ is presentable.
\end{lem} 
\begin{proof}
Since $\CMon_{m}(\mathcal{C}) $ is  presentable for all $m$, by the adjoint functor theorem \cite[Corollary 5.5.2.9]{htt}, the functor $\iota_m^*$ admits a left adjoint if and only if it is accessible and limit preserving. The functor $\iota_m^*$ is $\kappa$-accessible for any $\kappa$ large enough such that  $(m+1)$-finite limits commute with 
 $\kappa$-filtered colimits in $\mathcal{C}$, and by \propref{Rm_CMon}, is limit preserving. The second claim follows from the description of colimts in $\Pr$ (see \cite[Theorem 5.5.3.18, Corollary 5.5.3.4]{htt}). 
\end{proof}

In the theory of $m$-semiadditivity, the $\infty$-category $\CMon_{m}$
plays an analogous role to that of the $\infty$-category $\CMon$
of commutative (i.e. $\bb E_{\infty}$) monoids in spaces, in the the
theory of ordinary (i.e. $0$) semiadditivity. In particular, the
mapping space between every two objects in an $m$-semiadditive $\infty$-category
has a canonical $m$-commutative monoid structure. To see this, we
begin by recalling the fundamental universal property of $\CMon_{m}$
from  \cite{Harpaz}. For each $\mathcal{C}\in\acat^{\fin m}$, we have
a forgetful functor $\CMon_{m}(\mathcal{C})\to \CMon_{-2}(\mathcal{C}) = \mathcal{C}$, given
by evaluation at $\pt\in\mathcal{S}_{\fin m}$.
\begin{prop}\label{prop:CMon_Universal_Property} Let $-2 \leq m \leq \infty$. For every  $\mathcal{C}\in\acat^{\sad m}$
and $\mathcal{D}\in\acat^{\fin m}$, post-composition with the forgetful
functor induces an equivalence of $\infty$-categories
\[
\fun^{\fin m}(\mathcal{C},\CMon_{m}(\mathcal{D}))\simeq \fun^{\fin m}(\mathcal{C},\mathcal{D}).
\]
\end{prop}
\begin{proof}
The case $m <\infty$ is proved in \cite[Proposition 5.14]{Harpaz}. For $m=\infty$ we have
\[
\fun^{\fin \infty}(\mathcal{C},\CMon_{\infty}(\mathcal{D}))  \simeq   \lim_k \fun^{\fin k}(\mathcal{C},\CMon_{\infty}(\mathcal{D})) \simeq   \]
\[  \lim_k \lim_{\ell} \fun^{\fin k}(\mathcal{C}, \CMon_{\ell}(\mathcal{D})) \simeq  \lim_k \fun^{\fin k}(\mathcal{C},\CMon_{k}(\mathcal{D})) \simeq  \]
\[ \lim_k \fun^{\fin k}(\mathcal{C},\mathcal{D}) \simeq \fun^{\fin \infty}(\mathcal{C},\mathcal{D}) . \]
\end{proof}

As a corollary,  for every $m$-semiadditive
$\infty$-category we have a unique lift of the Yoneda embedding to a
$\CMon_{m}$-enriched Yoneda embedding.
\begin{cor}
\label{cor:CMon_Yoneda} Let $-2 \leq m \leq \infty$. For every $\mathcal{C}\in\acat^{\sad m}$,
there is a unique fully faithful $m$-semiadditive functor
\[
\Yo^{\CMon_{m}}\colon\mathcal{C}\into\fun(\mathcal{C}^{\op},\CMon_{m})
\]
whose composition with the forgetful functor $\CMon_{m}\to\mathcal{S}$
is the Yoneda embedding. 
\end{cor}

\begin{proof}
Taking $\mathcal{D}=\mathcal{S}$ in \propref{CMon_Universal_Property},
shows that the ordinary Yoneda embedding 
\[
\Yo \colon\mathcal{C}\into\fun^{\fin m}(\mathcal{C}^{\op},\mathcal{S})\ss\fun(\mathcal{C}^{\op},\mathcal{S})
\]
lifts uniquely to a fully faithful $m$-finite limit preserving functor
\[
\Yo^{\CMon_{m}}\colon\mathcal{C}\into\fun^{\fin m}(\mathcal{C}^{\op},\CMon_{m})\ss\fun(\mathcal{C}^{\op},\CMon_{m}).
\]
\end{proof}
The $\CMon_{m}$-enriched Yoneda embedding $\Yo^{\CMon_{m}}$ corresponds
to a functor

\[
\hom^{\CMon_{m}}(-,-)\colon\mathcal{C}\times\mathcal{C}^{\op}\to\CMon_{m},
\]
whose composition with the forgetful functor $\CMon_{m}\to\mathcal{S}$
is the functor $\map_{\mathcal{C}}(-,-)$. Thus, we obtain a canonical
structure of an $m$-commutative monoid on each mapping space in $\mathcal{C}$.
Informally, the ``wrong way'' maps for $A\oto qB$, in the higher
commutative monoid structure on $\map_{\mathcal{C}}(X,Y)$, are given
by integration
\[
\int_{q}\colon\map_{\mathcal{C}}(X,Y)^{A}\to\map_{\mathcal{C}}(X,Y)^{B}.
\]

\begin{rem}
It is overwhelmingly likely that an $m$-semiadditive $\infty$-category
$\mathcal{C}$ can be canonically enriched in $\CMon_{m}$ (for e.g. in the sense of \cite{Rune} or  \cite{Hinich}), such that the $\CMon_{m}$-valued mapping objects coincide
with our definition above. In case $\mathcal{C}$ is further assumed
to be presentable, this follows from the fact that $\mathcal{C}$
is left tensored over $\CMon_{m}$ (see  \cite[Lemma 5.20]{Harpaz}, \propref{CMon_Tensor} and \cite[Proposition 4.2.1.33]{ha}).
\end{rem}
\subsection{Examples}
We now review some examples of $m$-semiadditive $\infty$-categories
and the behavior of cardinalities of $m$-finite spaces in them. 

\subsubsection{Universal}

It is proved in  \cite{Harpaz}, that the following is the universal example
of an $m$-semiadditive $\infty$-category. In particular, it shows
that in general, the operations $|A|$ need not reduce to something
``classical'':
\begin{example}
[Universal case]\label{exa:Universal_Cardinality}By  \cite[Corollary 5.7]{Harpaz}, for $-2 \leq m < \infty $ the
symmetric monoidal $\infty$-category of spans $\mathcal{C}=\span(\mathcal{S}_{\fin m})$
is the universal $m$-semiadditive $\infty$-category. For every $A\in\mathcal{S}_{\fin m}$,
we have 
\[
|A|_{\pt}=(\pt\from A\to\pt)\quad\in\quad\pi_{0}\map_{\span(\mathcal{S}_{\fin m})}(\pt,\pt).
\]
Moreover, $\pi_{0}\map_{\span(\mathcal{S}_{\fin m})}(\pt,\pt)$ is
the set of isomorphism classes of $m$-finite spaces with the semiring
structure given by (see \remref{Cardinality_Arithmetic}):
\[
|A|+|B|=|A\sqcup B|,\quad|A|\cdot|B|=|A\times B|.
\]
\end{example}

Informally, the universality of \exaref{Universal_Cardinality} is
reflected in its construction as follows. The collection of spaces
$\mathcal{S}_{\fin m}\ss\mathcal{S}$ is generated under $m$-finite
colimits from the point $\pt\in\mathcal{S}$. The ``right way''
maps in $\span(\mathcal{S}_{\fin m})$ encode the usual covariant
functoriality of these colimits. The ``wrong way'' maps in $\span(\mathcal{S}_{\fin m})$
encode the contravariant functoriality arising from ``integration
along the fibers''. 
 
A closely related example is the $\infty$-category $\CMon_{m}$ of
$m$-commutative monoids, which is shown in  \cite[Corollary  5.19, Corollary 5.21]{Harpaz}, to be the universal
\emph{presentable} $m$-semiadditive $\infty$-category. The Yoneda
embedding induces a fully faithful $m$-semiadditive (symmetric monoidal)
functor 
\[
\span(\mathcal{S}_{\fin m})\into\CMon_{m},
\]
taking each $m$-finite space $A$ to the ``free $m$-commutative
monoid'' on $A$. From this we get that cardinalities in $\CMon_{m}$
are computed essentially in the same way as in $\span(\mathcal{S}_{\fin m})$\footnote{The relation between $\span(\mathcal{S}_{\fin m})$ and $\CMon_{m}$
is somewhat analogues to the relation between the $\infty$-category
$\Sp^{\omega}$ of finite spectra and the $\infty$-category $\Sp$
of all spectra. }. In section 5, we shall discuss more systematically the universality
of $\CMon_{m}$ (see \propref{CMon_Tensor}).

\subsubsection{Rational}

There are however some situations in which the operations $|A|$ can
be expressed in terms of classical invariants. 
\begin{example}
[Homotopy cardinality]\label{exa:Homotopy_Cardinality}For a $\pi$-finite
space $A$, Baez and Dolan \cite{Baez} define the \textbf{homotopy}
\textbf{cardinality} of $A$ to be the following non-negative rational
number
\begin{equation}
|A|_{0}\coloneqq\sum_{a\in\pi_{0}(A)}\prod_{n\ge1}|\pi_{n}(A,a)|^{(-1)^{n}}\quad\in\quad\bb Q_{\ge0}.\label{eq:Rational_Cardinality}
\end{equation}
This notion can be seen as a special case of the cardinality of a
$\pi$-finite space in a higher semiadditive $\infty$-category as
follows. We say that an $\infty$-category $\mathcal{C}$ is \textbf{semirational}
if it is 0-semiadditive and for each $n\in\bb N$, multiplication
by $n$ is invertible in $\mathcal{C}$ (e.g. $\mathcal{C}=\Sp_{\bb Q}$
or $\mathrm{Vec}_{\bb Q}$). We shall see that a semirational $\infty$-category which admits all 1-finite colimits is
automatically $\infty$-semiadditive and for every $\pi$-finite space
$A$, we have 
\[
|A|_{\mathcal{C}}=|A|_{0}\quad\in\quad\bb Q_{\ge0}\ss\End(\Id_{\mathcal{C}}).
\]
We note that formula (\ref{eq:Rational_Cardinality}) is completely
determined by the additivity of the cardinality under coproducts and
the following ``multiplicativity property'': For every fiber sequence
of $\pi$-finite spaces
\[
F\to A\to B,
\]
such that $B$ is connected, we have $|A|=|F||B|$\footnote{This follows from the long exact sequence in homotopy groups, and
is reminiscent of the ``additivity property'' of the Euler characteristic. }. 
\end{example}

\begin{rem}
\label{rem:Cardinality_Height_0}In fact, we shall prove in \propref{Cardinality_Height_0}
a somewhat sharper result. Let $\mathcal{C}$ be a $0$-semiadditive
$\infty$-category, which admits $\pi$-finite colimits, and let $A$
be a $\pi$-finite space. If $A$ satisfies the following condition:
\begin{itemize}
\item [$(*)$] The orders of the homotopy groups of $A$ are invertible
on all objects of $\mathcal{C}.$
\end{itemize}
Then $A$ is $\mathcal{C}$-ambidextrous and $|A|_{\mathcal{C}}=|A|_{0}$.
\end{rem}

From the perspective of the theory we are about to develop, semirational
$\infty$-categories are $0$-semiadditive $\infty$-categories of
``height 0''. One of our goals is to generalize the above phenomena
to ``higher heights'' (see \propref{Height_Bootstrap} and \propref{Height_Cardinality}). 

\subsubsection{Chromatic}

Examples of $\infty$-semiadditive $\infty$-categories of ``higher
height'' appear naturally in chromatic homotopy theory. For a given
prime $p$ and $0\le n<\infty$, let $K(n)$ be the Morava
$K$-theory spectrum of height $n$ at the prime $p$. One of the main results
of \cite{HopkinsLurie} is that the localizations $\Sp_{K(n)}$ are $\infty$-semiadditive.
In particular, we can consider $K(n)$-local cardinalities of $\pi$-finite
spaces. For $n=0$, we have $\Sp_{K(n)}\simeq\Sp_{\bb Q}$ and we
recover the homotopy cardinality (\exaref{Homotopy_Cardinality}).
Similarly, since $\Sp_{K(n)}$ is $p$-local for all $n$, for every
$\pi$-finite space $A$ whose homotopy groups have cardinality prime
to $p$, the $K(n)$-local cardinality of $A$ coincides with its
homotopy cardinality for all $n$ (see \remref{Cardinality_Height_0}).
In particular, it is independent of $n$. However, for $n\ge1$ the
prime $p$ is not invertible in $\Sp_{K(n)}$. Thus, there are $\pi$-finite
spaces $A$ (e.g. $\pi$-finite $p$-spaces), which are ambidextrous
even though they do not satisfy condition $(*)$ of \remref{Cardinality_Height_0}.
For such spaces $A$, the $K(n)$-local cardinality does depend on
$n$ and in general does not (and can not) agree with the homotopy
cardinality.\footnote{Note that the rationalization functor $L_{\bb Q}\colon\Sp_{K(n)}\to\Sp_{\bb Q}$
does not preserve colimits in general and so does not preserve cardinalities.
It does however preserve colimits which are indexed on $\pi$-finite
spaces whose homotopy groups have order prime to $p$.} 

To study the $K(n)$-local cardinalities of $\pi$-finite spaces, it
is useful to consider their image in Morava $E$-theory. For every
integer $n\ge1$, we let $E_{n}$ be the Morava $E$-theory associated
with some formal group of height $n$ over $\overline{\bb F}_{p}$,
viewed as an object of $\calg(\Sp_{K(n)})$. In particular, we have
a (non-canonical) isomorphism 
\[
\pi_{*}E_{n}\simeq\bb W(\overline{\bb F}_{p})[[u_{1},\dots,u_{n-1}]][u^{\pm1}]\qquad|u_{i}|=0,\quad|u|=2.
\]

\begin{example}
[Chromatic cardinality]\label{exa:Chromatic_Cardinality}The $\infty$-category
$\Theta_{n}=\Mod_{E_{n}}(\Sp_{K(n)})$ is $\infty$-semiadditive \cite[Theorem 5.3.1]{Ambi2018}
and hence one can consider cardinalities of $\pi$-finite spaces in
$\pi_{0}E_{n}$. We define the ($p$-typical) \textbf{height $n$
cardinality} of $A$ to be 
\[
|A|_{n}\coloneqq|A|_{\Theta_{n}}\quad\in\quad\pi_{0}E_{n}.
\]
It makes sense to consider $\overline{\bb Q}$ as $E_{0}$, in which
case we recover the homotopy cardinality (\exaref{Homotopy_Cardinality}).
The technology of \cite{Ell3} allows one to derive a rather explicit
formula for $|A|_{n}$, for heights $n>0$ as well. Let $\widehat{L}_{p}A\coloneqq\map({B\bb Z}_{p},A)$
be the $p$-adic free loop space of $A$. One can show that the element
$|A|_{n}\in\pi_{0}E_{n}$ belongs to the subring $\bb Z_{(p)}\ss\pi_{0}E_{n}$
and satisfies $|A|_{n}=|\widehat{L}_{p}A|_{n-1}$. Applying this relation
inductively we obtain the formula
\begin{equation}
|A|_{n}=|\map({B\bb Z}_{p}^{n},A)|_{0}\quad\in\quad\bb Z_{(p)}.\label{eq:Loop_Numbers}
\end{equation}
If $A$ happens to be a $p$-space, then $\widehat{L}_{p}A$ coincides
with the ordinary free loop space $LA\coloneqq\map(S^{1},A)$. Thus,
$|A|_{n}$ can be computed as the homotopy cardinality of the space
of maps from the $n$-dimensional torus to $A$. 
\end{example}

We shall not get here into the details of how formula (\ref{eq:Loop_Numbers})
is deduced from the results of \cite{Ell3}, as we shall only need the
following special case: 
\begin{prop}
\label{prop:Cardinality_EM}For all $k,n\ge0$ we have $|B^{k}C_{p}|_{n}=p^{\binom{n-1}{k}}$.\footnote{For height $n=0$ this should be interpreted via the identity $\binom{-1}{k}=(-1)^{k}.$}
\end{prop}

This was proved independently in \cite[Lemma 5.3.3]{Ambi2018} by relating the cardinality
to the symmetric monoidal dimension. However, we shall use the general
formula (\ref{eq:Loop_Numbers}) in some examples to illustrate interesting
phenomena.

While the structure of the rings $\pi_{0}\bb S_{K(n)}$ is not entirely
understood in general, it follows from \cite{BobkovaG} and \cite{Agnes2017}
that:
\begin{prop}
\label{prop:Goerss_Irena}For all $p$ and $n$, the image of the
unit map $\pi_{0}\bb S_{K(n)}\oto u\pi_{0}E_{n}$ is $\bb Z_{p}\ss\pi_{0}E_{n}$
and the kernel is precisely the nil-radical. 
\end{prop}

\begin{proof}
Let $\Gamma_{n}$ be the Morava stabilizer group associated with $E_{n}$.
We have an action of $\Gamma_{n}$ on $E_{n}$ by commutative algebra
maps and thus, the map $u$ factors through the fixed points $(\pi_{0}E_{n})^{\Gamma_{n}}\ss\pi_{0}E_{n}$.
By \cite[Lemma 1.33]{BobkovaG}, we have 
\[
(\pi_{0}E_{n})^{\Gamma_{n}}=H_{c}^{0}(\Gamma;\pi_{0}E_{n})=\bb Z_{p}\ss\pi_{0}E_{n}.
\]

By \cite[Theorem 2.3.5]{Agnes2017}, the $E_{\infty}$-page of the descent spectral sequence
\[
H_{c}^{s}(\Gamma;(E_{n})_{t})\implies\pi_{t-s}(\bb S_{K(n)})
\]
has a horizontal vanishing line. Since the spectral sequence is multiplicative,
this implies that all elements in $\pi_{0}\bb S_{K(n)}$ with positive
filtration degree are nilpotent. Finally, since $\bb S_{K(n)}$ admits a ring map from the $p$-complete sphere, the map   $\pi_{0}\bb S_{K(n)}\oto u\bb Z_p$ is surjective.
\end{proof}
Thus, for every $\pi$-finite space $A$, the identity 
\[
|A|_{\Sp_{K(n)}}=|A|_{n}
\]

holds \emph{up to nilpotents}. We do not know, however, whether it holds
in $\pi_{0}\bb S_{K(n)}$.

\subsubsection{Categorical}

Another family of examples of higher semiadditive $\infty$-categories
arises from category theory itself.
\begin{prop}
For every  $-2 \leq m \leq \infty $ the $\infty$-category  $\acat_{m\text{-fin}}$  is $m$-semiadditive. 
\end{prop}
\begin{proof}
The case $m<\infty$ is exactly \cite[Proposition 5.26]{Harpaz}. We now wish to show that  $\acat_{\infty\text{-fin}}$ is $k$-semiadditive for every $k < \infty $. By \cite[Remark 4.8.1.6]{ha}, both  $\acat_{k\text{-fin}}$ and $\acat_{\infty\text{-fin}}$   admit closed symmetric monoidal structures, and by \cite[Proposition  4.8.1.3]{ha}, there exists  a symmetric monoidal functor $\mathcal{P}\colon \acat_{k\text{-fin}} \to  \acat_{\infty\text{-fin}}$. 
By \cite[Remark 4.8.1.8]{ha} and   \cite[Proposition 5.3.6.2(2)]{htt},  $\mathcal{P}$ admits a right adjoint and thus preserves colimits. Hence,  $\acat_{\infty\text{-fin}}$ is $k$-semiadditive by \cite[Corollary 3.3.2(2)]{Ambi2018}.\end{proof}

\begin{example}[Categorical cardinality]\label{exa:Categorical_Cardinality} Let $-2 \leq m \leq \infty $ and let $\mathcal{C} \in \acat_{m\text{-fin}}$. For every  $m$-finite space $A$ the $m$-semiadditive structure of  $\acat_{m\text{-fin}}$ gives rise to a functor 
$|A|_{\mathcal{C}}\colon\mathcal{C}\to\mathcal{C}.$
 When $m < \infty$ it is shown in \cite[Section  5.2]{Harpaz} that $|A|_{\mathcal{C}}$ is given by taking the constant colimit on $A$. That is, it takes an object $X\in\mathcal{C}$ to the object $X[A]\in\mathcal{C}$.  Since  the forgetful functor $\acat_{\infty\text{-fin}} \to  \acat_{m\text{-fin}}$ preserves limits, and hence m-semiadditive, the same holds for $m=\infty$.  This is very suggestive of the idea that ``multiplication by $|A|$
on $\mathcal{C}$'' is given by ``summing each object $X\in\mathcal{C}$
with itself $A$ times''.  A closely related example is discussed
in \cite[Example 4.3.11]{HopkinsLurie}, where it is shown that $\Pr$ is $\infty$-semiadditive
(in fact, every $m$-truncated space, not necessarily $\pi$-finite,
is $\Pr$-ambidextrous).
\end{example}


\begin{rem}
There is a different approach to the higher semiadditivity of $\acat_{m\text{-fin}}$,
based on the notion of \emph{ambidextrous} \emph{adjunctions} of $(\infty,2)$-categories. We sketch the argument to demonstrate
the role of higher categorical structures as a useful perspective
on ambidexterity phenomena. Given an $m$-finite space $A$, the adjunction
\[
A^{*}\colon\acat_{\fin m}\adj\acat_{\fin m}^{A}\colon A_{*}
\]
can be naturally promoted to an adjunction of $(\infty,2)$-categories.
Moreover, the unit $u$ and counit $c$ of $A^{*}\dashv A_{*}$, as
$1$-morphisms in the respective $(\infty,2)$-categories of endofunctors,
can be shown to have left adjoints $u^{L}$ and $c^{L}$ respectively.
Thus, we are in a situation which is dual to the notion of an \emph{ambidextrous}
\emph{adjunction} of  \cite[Definition 2.1]{HSSS2018}. It follows by an elementary
argument that $u^{L}$ and $c^{L}$ exhibit $A_{*}$ as a \emph{left}
adjoint of $A^{*}$ (see \cite[Remark 2.2]{HSSS2018}). Hence, by \propref{Ambi_Criterion},
the space $A$ is $\acat_{m\text{-fin}}$-ambidextrous and so $\acat_{\fin m}$
is $m$-semiadditive. 
\end{rem}


The functor that takes an $\infty$-category to its opposite induces
an equivalence $\acat_{\fin m}\simeq\acat^{\fin m}$. Hence, the $\infty$-category
$\acat^{\fin m}$ is $m$-semiadditive as well and the higher semiadditive
structure is given by taking \emph{limits}. For every $\infty$-category
$\mathcal{C}$ with finite (co)products, the (co)product endows the
space of objects $\mathcal{C}^{\simeq}$ with a \emph{(co)Cartesian}
commutative monoid structure. Using the $m$-semiadditivity of $\acat_{\fin m}$
and $\acat^{\fin m}$ together with the $\CMon_{m}$-enriched Yoneda
embedding provided by \corref{CMon_Yoneda}, this too can be generalized
to all $0 \leq m \leq \infty$. Given $\mathcal{C}\in\acat_{\fin m}$, by the $m$-semiadditivity of $\acat_{\fin m}$, the mapping
space $\map_{\fin m}(\mathcal{S}_{\fin m},\mathcal{C})$ admits a
canonical structure of an $m$-commutative monoid. On the other hand,
since $\mathcal{S}_{\fin m}$ is freely generated from a point under
$m$-finite colimits \cite[Notation 4.8.5.2]{ha}, we have 
\[
\map_{\fin m}(\mathcal{S}_{\fin m},\mathcal{C})\simeq\map(\pt,\mathcal{C})\simeq\mathcal{C}^{\simeq}.
\]

\begin{defn}
\label{def:CoCart_CMon}For $\mathcal{C}\in\acat_{\fin m}$, we refer
to the above $m$-commutative monoid structure on the space of objects
$\mathcal{C}^{\simeq}$ as the \textbf{coCartesian }structure. A completely
analogous construction endows the space of objects of each $\mathcal{C}\in\acat^{\fin m}$
with a \textbf{Cartesian} $m$-commutative monoid structure.
\end{defn}

As explained in \cite[Section 5.2]{Harpaz}, 
the integration operations for the (co)Cartesian
$m$-commutative monoid structure on $\mathcal{C}^{\simeq}$ are given
by taking $m$-finite (co)limits. Finally, the $\infty$-category
$\acat^{\sad m}$ of $m$-semiadditive $\infty$-categories and $m$-semiadditive
functors is a \emph{full} subcategory of both $\acat_{\fin m}$ and
$\acat^{\fin m}$. Moreover, we have the following:
\begin{prop}
\label{prop:Sad_Sad}Let $-2 \leq m  \leq \infty$. The full subcategory $\acat^{\sad m}\ss\acat_{\fin m}$
(resp. $\acat^{\fin m}$) is closed under colimits and in particular
is $m$-semiadditive. 
\end{prop}
\begin{proof}
The functor $(-)^{\mathrm{op}} \colon \acat_{\fin m}\to \acat^{\fin m}$ that takes an $\infty$-category to its opposite is an equivalence. By  \propref{Ambi_Closure}, $(-)^{\mathrm{op}}$ restricts to an involution  of  $\acat^{\sad m}$. It thus suffices to consider only the inclusion  $\iota\colon \acat^{\sad m}\hookrightarrow \acat^{\fin m}$. By  \propref{CMon_Universal_Property}, $\iota$ admits a right adjoint given by $\CMon_m(-)$ and thus preserve colimits.
\end{proof}
%

\subsection{Amenability}

Let $A\oto qB$ be a $\mathcal{C}$-ambidextrous map. Recall from
\cite[Definition 3.1.7]{Ambi2018}, that $q$ is called $\mathcal{C}$\textbf{-amenable} if
$|q|_{\mathcal{C}}$ is \emph{invertible}. As with ambidexterity,
amenability is a fiber-wise property (\cite[Corollary 3.1.16]{Ambi2018}). Thus, we shall be
mainly interested in $\mathcal{C}$-amenable \emph{spaces.} i.e. those
whose $\mathcal{C}$-cardinality is invertible.
\begin{rem}
We warn the reader not to confuse the condition that the natural transformation
$|q|_{\mathcal{C}}$ is invertible with the condition that the natural
transformation $\nm_{q}$ is invertible, which is equivalent to $q$
being $\mathcal{C}$-ambidextrous (which is itself a prerequisite
for defining $|q|_{\mathcal{C}}$).
\end{rem}

In this section, we extend some results from \cite{Ambi2018} regarding amenability.
The main point is that while $A$-ambidexterity allows us to sum over
$A$-families of maps, $A$-amenability allows us to \emph{average}
over $A$-families of maps, which in turn facilitates ``transfer
arguments'' along $A$. We shall explore how this condition affects
the higher semiadditive structure. 

\subsubsection{Closure properties}

For a map of spaces, the condition of $\mathcal{C}$-amenability,
as the condition of $\mathcal{C}$-ambidexterity, is fiber-wise.
However, \emph{unlike} $\mathcal{C}$-ambidextrous maps, $\mathcal{C}$-amenable
maps are not closed under composition. To understand the situation,
it suffices to consider the case $A\oto qB\to\pt$. By the additivity
of cardinality (\propref{Cardinality_Additivity}), we get $|A|=\int_{B}|q|$. Assume for simplicity that $B$ is connected and that the fiber of
$q$ is $F$. The transformation $|q|$ equals $|F|$ at each point
$b\in B$. Thus, it is tempting to presume that $|A|=|B|\cdot|F|$
and hence if both $|B|$ and $|F|$ are invertible, then so is $|A|$.
However, for this reasoning to hold we need to know that $|q|$ is
\emph{constant }on $B$ with value $|F|$. Alas, in general $|q|$
is not constant, even when $|F|$ is invertible, and $|A|$ need not
equal $|B|\cdot|F|$ (see \exaref{Amenable_Composition}). On the
positive side, we show that $|q|$ must be constant if we require
in addition to the invertibility of $|F|$ that $F\to A\oto qB$ is a \emph{principal} fiber sequence. 
\begin{defn}
We call a map $A\oto qB$ of spaces \textbf{principal} if it can be
extended to a fiber sequence $A\oto qB\oto fE$.
\end{defn}

We note that for a principal map all the fibers are isomorphic even
if the target is not connected. 
\begin{prop}
\label{prop:Amenable_Descent}Let $\mathcal{C}\in\cat$ and let $A\oto qB$
be a principal $\mathcal{C}$-ambidextrous map of $\mathcal{C}$-ambidextrous
spaces with fiber $F$. For $X\in\mathcal{C}$, if $|F|_{X}$ is invertible,
then 
\[
|q|_{B^{*}X}=|F|_{B^{*}X}\quad\in\quad\End(B^{*}X),
\]
and 
\[
|A|_{X}=|F|_{X}|B|_{X}.
\]
\end{prop}

\begin{proof}
The base change of $q$ along itself is a map $A\times_{B}A\oto{\tilde{q}}A$,
which is a principal map with a section. Therefore $\widetilde{q}$
is isomorphic to the projection $F\times A\oto{\pi_{A}}A.$ Hence,
$q\times_{B}q\simeq q\times_{B}\pi_{B}$, where $F\times B\oto{\pi_{B}}B$
is the projection. We get from  \cite[Corollary  3.1.14]{Ambi2018} that
\[
|q|^{2}=|q\times_{B}q|=|q\times_{B}\pi_{B}|=|q||\pi_{B}|\quad\in\quad\End(\Id_{\mathcal{C}^{B}}).
\]
 If $|F|_{X}$ is invertible, then $|q|_{B^{*}X}$ is invertible,
and thus by the above, we get (\cite[Proposition 3.1.13]{Ambi2018}) 
\[
|q|_{B^{*}X}=|\pi_{B}|_{B^{*}X}=B^{*}(|F|_{X})=|F|_{B^{*}X}\quad\in\quad\End(B^{*}X).
\]
We can now integrate along $B$ and get 
\[
|A|_{X}=\int\limits _{B}|q|_{B^{*}X}=\int\limits _{B}B^{*}(|F|_{X})=|F|_{X}|B|_{X}\quad\in\quad\End(X).
\]
\end{proof}
As a simple application, we deduce that when the homotopy groups of
a $\pi$-finite space have invertible cardinality in $\mathcal{C}$,
the notion of cardinality degenerates to the homotopy cardinality
(\remref{Cardinality_Height_0}).
\begin{prop}
\label{prop:Cardinality_Height_0}Let $\mathcal{C}\in\acat^{\sad 0}$
and let $A$ be a $\pi$-finite space. Assume that $\mathcal{C}$ admits $A$-colimits and the order of each
homotopy group of $A$ is invertible in $\End(\Id_{\mathcal{C}})$. Then, $A$
is $\mathcal{C}$-ambidextrous and $|A|_{\mathcal{C}}=|A|_{0}$. In
particular, if $A$ is connected, then it is $\mathcal{C}$-amenable.
\end{prop}

\begin{proof}
By \remref{Cardinality_Arithmetic}, we may assume that $A$ is connected.
We prove the claim for all connected $m$-finite $A$ by induction
on $m\ge0$, where the case $m=0$ is trivial. We prove the claim
for some $m>0$, assuming it holds for $m-1$. Choose a base point
in $A$ and consider the connected component $(\Omega A)^{\circ}\ss\Omega A$
of the identity loop. The space $(\Omega A)^{\circ}$ satisfies the
assumptions of the inductive hypothesis. Hence, $(\Omega A)^{\circ}$
is amenable and $|(\Omega A)^{\circ}|_{\mathcal{C}}=|(\Omega A)^{\circ}|_{0}$.
Now, $\Omega A$ is just the coproduct of $|\pi_{1}A|$ copies of
$(\Omega A)^{\circ}$. Hence, by \remref{Cardinality_Arithmetic}
and the inductive hypothesis, we have 
\[
|\Omega A|_{\mathcal{C}}=|\pi_{1}A|\cdot|(\Omega A)^{\circ}|_{\mathcal{C}}=|\pi_{1}A|\cdot|(\Omega A)^{\circ}|_{0}=|\Omega A|_{0}.
\]
Moreover, $|\Omega A|_{\mathcal{C}}$ is invertible as $|\pi_{1}A|$
is invertible by assumption and $|(\Omega A)^{\circ}|_{\mathcal{C}}$
is invertible by the inductive hypothesis. That is, $\Omega A$ is
$\mathcal{C}$-amenable. Finally, consider the principal fiber sequence
\[
\Omega A\to\pt\to A.
\]
Since $\Omega A$ is $\mathcal{C}$-amenable, by \cite[Proposition  3.1.17]{Ambi2018}, the space
$A$ is $\mathcal{C}$-ambidextrous and by \propref{Amenable_Descent},
we have 
\[
|A|_{\mathcal{C}}=|\Omega A|_{\mathcal{C}}^{-1}=|\Omega A|_{0}^{-1}=|A|_{0}.
\]
\end{proof}
From \propref{Amenable_Descent}, we also deduce that $\mathcal{C}$-amenable
maps are partially closed under composition:
\begin{cor}
\label{cor:Amenable_Extensions}Let $\mathcal{C}\in\cat$ and let
$A\oto fB\oto gC$ be a pair of composable maps of spaces. If $f$
and $g$ are $\mathcal{C}$-amenable and $f$ is principal, then $g\circ f$
is $\mathcal{C}$-amenable. 
\end{cor}

\begin{proof}
We can check that $g\circ f$ is $\mathcal{C}$-amenable, by pulling
back along $\pt\to C$ for every point of $C$. In other words, we
can assume that $C=\pt$. Taking $F$ to be the fiber of $f$, we
have a principal fiber sequence $F\to A\to B$ where $F$ and $B$
are $\mathcal{C}$-amenable; thus the result follows from \propref{Amenable_Descent}. 
\end{proof}

\subsubsection{Counter examples}

We conclude this subsection with a discussion of the necessity of
the conditions in \propref{Amenable_Descent}. For starters, if $|F|$
is not invertible, then the identity $|A|=|F||B|$ is (very much) false in
general. The following examples show that the condition on the fiber
sequence to be \emph{principal} can also not be dropped. The first
example shows that $\mathcal{C}$-cardinality need not be multiplicative
even when the fiber and base space are $\mathcal{C}$-amenable. Moreover,
in such case, the total space need not even be $\mathcal{C}$-amenable
and so in particular, $\mathcal{C}$-amenable maps are not closed
under composition. 
\begin{example}
\label{exa:Amenable_Composition}Let $p$ be an odd prime and let
$\Theta_{n}\coloneqq\Mod_{E_{n}}(\Sp_{K(n)})$. We consider the map
$B^{2}C_{p}\oto fB^{4}C_{p}$ classifying the cup-square operation
$x\mapsto x\cup x$ on mod $p$ cohomology, and the associated fiber
sequence 
\[
F\to B^{2}C_{p}\oto fB^{4}C_{p}.
\]

The only non-trivial homotopy groups of $F$ are $\pi_{2}F\simeq\pi_{3}F\simeq C_{p}$,
but the Postnikov invariant represented by $f$ in $H^{4}(B^{2}C_{p};C_{p})$
is non-zero. Using \exaref{Chromatic_Cardinality}, we have $|F|_{n}=|L^{n}F|_{0}$
and we can compute it using the fiber sequence 
\[
L^{n}F\to L^{n}B^{2}C_{p}\to L^{n}B^{4}C_{p},
\]
via the induced long exact sequence on homotopy groups. The only complication
arises at the level of $\pi_{0}$, where we need to compute the size
of the kernel of the cup-square map
\[
\pi_{0}L^{n}B^{2}C_{p}=H^{2}(T^{n};C_{p})\to H^{4}(T^{n};C_{p})=\pi_{0}L^{n}B^{4}C_{p}.
\]
Namely, the number of $n$-dimensional 2-forms over $C_{p}$ that
square to zero. Since this is the number of 2-forms of rank lower
or equal 1, one can write down an explicit combinatorial formula for
it. This leads to the following explicit formula:
\[
|F|_{n}=p^{\binom{n-1}{3}}\cdot\frac{p^{3-n}+p^{n}-p-1}{p^{2}-1}.
\]
In particular, taking $n=4$, we get $|F|_{4}=p^{3}+p-1$, which is
an \emph{invertible} element in $\pi_{0}E_{4}$. It follows that  $F$ is $\Theta_4$-amenable.
Nevertheless, 
\[
|F|_{4}|B^{4}C_{p}|_{4}=p^{3}+p-1
\]
which differs from 
\[
|B^{2}C_{p}|_{4}=p^{\binom{3}{2}}=p^{3}.
\]

Moreover, $F$ and $B^{4}C_{p}$ are both $\Theta_4$-amenable,
but $B^{2}C_{p}$ is not. Thus, the maps $B^{2}C_{p}\oto fB^{4}C_{p}$
and $B^{4}C_{p}\to\pt$ are $\Theta_4$-amenable, but their composition
is not.
\end{example}

The next example shows that $\mathcal{C}$-cardinality need not be
multiplicative when the fiber and \emph{total} space are $\mathcal{C}$-amenable.
Moreover, the base space in this case need not be $\mathcal{C}$-amenable
and so in particular the class of $\mathcal{C}$-amenable maps does
not satisfy ``left cancellation'' (compare \cite[Theorem 2.4.5]{Ambi2018}).
\begin{example}
\label{exa:Amenable_Cancelation}Let $p=2$ and $\Theta_{1}=\Mod_{E_{1}}(\Sp_{K(1)})$.
Consider the (non-principal) fiber sequence 
\[
\Sigma_{3}/C_{2}\to BC_{2}\to B\Sigma_{3}.
\]

It can be shown using \exaref{Chromatic_Cardinality}, that for $p=2$
we have 
\[
|B\Sigma_{3}|_{1}=|\widehat{L}_{p}B\Sigma_{3}|_{0}=\frac{2}{3}.
\]
In particular,
\[
|\Sigma_{3}/C_{2}||B\Sigma_{3}|=2\neq1=|BC_{2}|_{1}.
\]
Moreover, $\Sigma_{3}/C_{2}$ and $BC_{2}$ are $\Theta_{1}$-amenable,
but $B\Sigma_{3}$ is not. Thus, the map $BC_{2}\to B\Sigma_{3}$
and the composition $BC_{2}\to B\Sigma_{3}\to\pt$ are $\Theta_{1}$-amenable,
but $B\Sigma_{3}\to\pt$ is not.
\end{example}

\subsection{\label{subsec:Acyclicity}Acyclic Maps}

In this subsection, we show that the amenability of a loop-space is
equivalent to the ``triviality'' of limits and colimits over its
classifying space. This characterization is interesting in that it
does not directly involve the higher semiadditive structure. 

\subsubsection{Definitions \& basic properties}

We begin by introducing the notions of ``acyclicity'' and ``triviality'', which are not immediately related to the theory of ambidexterity:
\begin{defn}
\label{def:Acyclic_Map}Let $A\oto qB$ be a map of spaces and let
$\mathcal{C}$ be an $\infty$-category. We say that $q$ is $\mathcal{C}$\textbf{-acyclic}
(resp. $\mathcal{C}$\textbf{-trivial}) if $\mathcal{C}$ admits all
$q$-limits and $q$-colimits and $q^{*}\colon\mathcal{C}^{B}\to\mathcal{C}^{A}$
is fully faithful (resp. an equivalence). 
\end{defn}

It is a standard fact about adjoints, that $q^{*}$ is fully faithful,
if and only if the counit $q_{!}q^{*}\oto{c_{!}^{q}}\Id_{\mathcal{C}^{B}}$
is an isomorphism and if and only if the unit $\Id_{\mathcal{C}^{B}}\oto{u_{*}^{q}}q_{*}q^{*}$
is an isomorphism. Similarly, $q^{*}$ is an equivalence, if furthermore
$q^{*}q_{*}\oto{c_{*}^{q}}\Id_{\mathcal{C}^{A}}$ is an isomorphism,
or equivalently $\Id_{\mathcal{C}^{A}}\oto{u_{!}^{q}}q^{*}q_{!}$
is an isomorphism. Like ambidexterity and amenability, acyclicity
and triviality are fiber-wise conditions:
\begin{prop}
\label{prop:Acyclicity_Fiberwise}Let $\mathcal{C}\in\cat$. A map
of spaces $A\oto qB$ is $\mathcal{C}$-acyclic (resp. $\mathcal{C}$-trivial)
if and only if each fiber of $q$ is $\mathcal{C}$-acyclic (resp.
$\mathcal{C}$-trivial).
\end{prop}

\begin{proof}
For each map $\widetilde{B}\oto fB$ we can form the following pullback
square of spaces and the induced commutative square of $\infty$-categories
\[
\xymatrix{\widetilde{A}\ar[d]^{\tilde{q}}\ar[rr]^{\tilde{f}} &  & A\ar[d]^{q}\\
\widetilde{B}\ar[rr]^{f} &  & B & ,
}
\qquad\xymatrix{\mathcal{C}^{B}\ar[d]^{q^{*}}\ar[rr]^{f^{*}} &  & \mathcal{C}^{\widetilde{B}}\ar[d]^{\tilde{q}^{*}}\\
\mathcal{C}^{A}\ar[rr]^{\tilde{f}^{*}} &  & \mathcal{C}^{\widetilde{A}}.
}
\]
By \cite[Lemma 2.2.3]{Ambi2018}, we have 
\[
f^{*}(c_{!}^{q})=c_{!}^{\tilde{q}}\quad\in\quad\map(\tilde{q}_{!}\tilde{q}^{*},\Id_{\mathcal{C}^{\tilde{B}}}).
\]
\[
\tilde{f}^{*}(u_{!}^{q})=u_{!}^{\tilde{q}}\quad\in\quad\map(\Id_{\mathcal{C}^{\tilde{A}}},\tilde{q}^{*}\tilde{q}_{!}).
\]
Thus, if $q$ is $\mathcal{C}$-acyclic (resp. $\mathcal{C}$-trivial),
then $\tilde{q}$ is $\mathcal{C}$-acyclic (resp. $\mathcal{C}$-trivial).
If $f$ (and hence also $\tilde{f}$) is surjective, then by the conservativity
of $f^{*}$ (and $\tilde{f}^{*}$), the converse holds as well. Namely,
if $\tilde{q}$ is $\mathcal{C}$-acyclic (resp. $\mathcal{C}$-trivial),
then $q$ is $\mathcal{C}$-acyclic (resp. $\mathcal{C}$-trivial).
Applying this to any section of $B\to\pi_{0}B=\widetilde{B}$ yields
the claim. 
\end{proof}
The collections of $\mathcal{C}$-acyclic and $\mathcal{C}$-trivial
spaces are also closed under extensions:
\begin{cor}
\label{cor:Acyclic_Spaces_Closure}Let $\mathcal{C}\in\cat$ and let
$A\oto qB$. If $B$ is $\mathcal{C}$-acyclic (resp. $\mathcal{C}$-trivial)
and all the fibers of $q$ are $\mathcal{C}$-acyclic (resp. $\mathcal{C}$-trivial),
then $A$ is $\mathcal{C}$-acyclic (resp. $\mathcal{C}$-trivial). 
\end{cor}

\begin{proof}
By \propref{Acyclicity_Fiberwise}, it suffices to show that $\mathcal{C}$-acyclic
(resp. $\mathcal{C}$-trivial) maps are closed under composition,
which is clear from the definition. 
\end{proof}
In presence of a compatible monoidal structure, one can check the
acyclicity property on the unit:
\begin{lem}
\label{lem:Acyclic_Monoidal}Let
$A$ be a space and let $\mathcal{C}\in\alg(\cat)$ which is compatible with $A$-colimits. The space $A$ is $\mathcal{C}$-acyclic
if and only if the fold map $\one[A]\oto{\nabla}\one$ is an isomorphism.
\end{lem}

\begin{proof}
Assume $\one[A]\oto{\nabla}\one$ is an isomorphism. By assumption,
for every $X\in\mathcal{C}$, tensoring the isomorphism $\one[A]\oto{\nabla}\one$
with $X$ gives the fold map $X[A]\oto{\nabla}X$. Hence, $A_! A^*X=X[A]\oto{\nabla}X$
is an isomorphism for all $X\in\mathcal{C}$. 
\end{proof}
The following is the prototypical example of acyclicity:
\begin{example}
[Bousfield]\label{exa:Acyclic_Bousfield}For $E\in\Sp$, we can consider
the $\infty$-category $\Sp_{E}$ of $E$-local spectra. By \lemref{Acyclic_Monoidal},
a space is $\Sp_{E}$-acyclic if and only if it is $E$-acyclic in
the sense of Bousfield, i.e. has the $E$-homology of a point.
\end{example}

\begin{rem}
\label{rem:Acyclic_Warning}For an $\infty$-category $\mathcal{C}$,
a space $A$ is $\mathcal{C}$-acyclic if the following equivalent
conditions hold:
\begin{enumerate}
\item The fold map $X[A]\oto{\nabla}X$ is an isomorphism for all $X\in\mathcal{C}$. 
\item The diagonal map $X\oto{\Delta}X^{A}$ is an isomorphism for all $X\in\mathcal{C}$. 
\end{enumerate}
We warn the reader that for an \emph{individual} object $X$, it can
happen that the map $X\oto{\Delta}X^{A}$ is an isomorphism, but $X[A]\oto{\nabla}X$
is not (and vise versa). As a trivial example, consider $\mathcal{C}=\mathcal{S}$
with $X=\pt$ and any $A\neq\pt$. 
\end{rem}

\subsubsection{Relation to amenability}

Under suitable ambidexterity assumptions, the notion of $\mathcal{C}$-acyclicity
turns out to be closely related to that of $\mathcal{C}$-amenability.
To begin with, recall that for a $\mathcal{C}$-ambidextrous space
$A$ and an object $X\in\mathcal{C}$, the map $|A|_{X}$ is given
by the composition 
\[
X\oto{\Delta}X^{A}\oto{\nm_{A}^{-1}}X[A]\oto{\nabla}X.
\]
Thus, if $A$ is $\mathcal{C}$-acyclic, then it is in particular
$\mathcal{C}$-amenable. However, there is a deeper connection between
acyclicity and amenability, which we first state on an object-wise
level:
\begin{prop}
\label{prop:Amenable_Acyclic_Object} Let $\mathcal{C}\in\cat$ and
let $A$ be a connected $\mathcal{C}$-ambidextrous space. For every
$X\in\mathcal{C}$, the following are equivalent: 
\begin{enumerate}
\item The fold map $X[A]\oto{\nabla}X$ is an isomorphism. 
\item The diagonal $X\oto{\Delta}X^{A}$ is an isomorphism.
\item $|\Omega A|_{X}$ is invertible.
\end{enumerate}
Moreover, $|A|_{X}$ is then the inverse of $|\Omega A|_{X}$.
\end{prop}

\begin{proof}
Let $\pt\oto eA$ be a base point. First, we show that if the diagonal
map $X\oto{\Delta}X^{A}$ is an isomorphism, then $|A|_{X}|\Omega A|_{X}=\Id_{X}$.
We begin by reducing the claim to the fact that the map
\[
|e|_{A^{*}X}\colon A^{*}X\to A^{*}X
\]
equals the constant map on $|\Omega A|_{X}$. Indeed, given that,
by integrating along $A$, we get (\cite[Propositions 2.1.15 and 3.1.13]{Ambi2018})
\[
\Id_{X}=\int\limits _{A\circ e}\Id_{e^{*}A^{*}X}=\int\limits _{A}\left(\int\limits _{e}\Id_{e^{*}A^{*}X}\right)=\int\limits _{A}|e|_{A^{*}X}=\int\limits _{A}A^{*}(|\Omega A|_{X})=|A|_{X}|\Omega A|_{X}.
\]

Now, recall that the diagonal $X\oto{\Delta}X^{A}=A_{*}A^{*}X$ is
the unit $u_{*}^{A}$ of the adjunction $A^{*}\colon\mathcal{C}\adj\mathcal{C}^{A}\colon A_{*}.$
Thus, if $u_{*}^{A}$ is an isomorphism at $X$, then the map 
\[
\map(X,X)\to\map(A^{*}X,A^{*}X)
\]
is an isomorphism. Since $e^{*}A^{*}=\Id$, it follows by 2-out-of-3
that the map
\[
\map(A^{*}X,A^{*}X)\to\map(e^{*}A^{*}X,e^{*}A^{*}X)=\map(X,X)
\]
is an isomorphism as well. Thus, it suffices to check that the maps
$|e|_{A^{*}X}$ and $A^{*}(|\Omega A|_{X})$ coincide after applying
$e^{*}$. The pullback square of spaces
\[
\xymatrix{\Omega A\ar[d]\ar[r] & \pt\ar[d]^{e}\\
\pt\ar[r]^{e} & A
}
\]
gives by \cite[Proposition 3.1.13]{Ambi2018}
\[
e^{*}(|e|_{A^{*}X})=|\Omega A|_{e^{*}A^{*}X}=|\Omega A|_{X}.
\]

This concludes the proof of (2)$\implies$(3) and that $|A|_{X}$
is the inverse of $|\Omega A|_{X}$. A completely symmetric argument
using the adjunction $A_{!}\dashv A^{*}$ instead of $A^{*}\dashv A_{*}$
shows that (1)$\implies$(3) and that $|A|_{X}$ is the inverse of
$|\Omega A|_{X}$. The implication (3)$\implies$(1) in the case that
$|\Omega A|_{X}$ is invertible for \emph{all} $X\in\mathcal{C}$
is given by \cite[Proposition  3.1.18]{Ambi2018}. One easily checks that all the arguments are,
in fact, object-wise. Alternatively, one can run the argument on the
full subcategory of objects on which $|\Omega A|$ is invertible.
In particular, each of the conditions (1), (2), and (3) implies that
$|A|_{X}$ is the inverse of $|\Omega A|_{X}$ and so, in particular,
is invertible. Finally, the composition
\[
X\oto{\Delta}X^{A}\oto{\nm_{A}}X[A]\oto{\nabla}X
\]
is $|A|_{X}$. Thus, by 2-out-of-3, we also have the implication (1)$\implies$(2). 
\end{proof}
From this we get:
\begin{cor}
\label{cor:Amenable_Acyclic}Let $\mathcal{C}\in\cat$ and let $A$
be a connected space. The following conditions are equivalent:
\begin{enumerate}
\item $A$ is $\mathcal{C}$-acyclic and $\mathcal{C}$-ambidextrous. 
\item $\Omega A$ is $\mathcal{C}$-amenable and $\mathcal{C}$ admits $A$-colimits.
\end{enumerate}
In which case $|A|=|\Omega A|^{-1}$.
\end{cor}

\begin{proof}
This follows from \propref{Amenable_Acyclic_Object} and the fact
that if $\Omega A$ is $\mathcal{C}$-amenable, then $A$ is $\mathcal{C}$-ambidextrous
by \cite[Proposition 3.1.17]{Ambi2018}.
\end{proof}
We note that the assumption of $\mathcal{C}$-ambidexterity in (1)
of \corref{Amenable_Acyclic} can not be relaxed to \emph{weak} $\mathcal{C}$-ambidexterity.
The following is a simple counter-example:
\begin{example}
Let $\mathcal{C}=\mathrm{Vec}_{\bb F_{p}}$ be the $1$-category of
$\bb F_{p}$-vector spaces and $A=BC_{p}$. It is clear that $A$
is weakly $\mathcal{C}$-ambidextrous (as $\mathrm{Vec}_{\bb F_{p}}$
is semiadditive) and that $BC_{p}$ is $\mathcal{C}$-acyclic. However,
$BC_{p}$ is not $\mathcal{C}$-ambidextrous (and $C_{p}$ is not
$\mathcal{C}$-amenable). 
\end{example}

\begin{rem}
The notion of acyclicity can be considered for a general (not necessarily
$\pi$-finite) space in any $\infty$-category, without any assumptions
on ambidexterity. However, in the presence of ambidexterity, \corref{Amenable_Acyclic}
allows us to deduce the acyclicity of a given $\pi$-finite space
from the amenability of its loop space. This strategy was already
employed in the proof of \cite[Theorem E]{Ambi2018} (exploiting the $\infty$-semiadditivity
of $\Sp_{T(n)}$). 
\end{rem}

We conclude this subsection by showing an analogue of the equivalence
of (1) and (3) in \propref{Amenable_Acyclic_Object}, for the notions
of $\mathcal{C}$-acyclicity and $\mathcal{C}$-triviality. As before,
it is clear that a $\mathcal{C}$-trivial space is, in particular, $\mathcal{C}$-acyclic,
but there is a better statement:
\begin{prop}
\label{prop:Acyclic_Trivial} Let $A$ be a connected space and let
$\mathcal{C}$ be an $\infty$-category which admits $\Omega A$-colimits.
Then, $A$ is $\mathcal{C}$-trivial if and only if $\Omega A$ is
$\mathcal{C}$-acyclic.
\end{prop}

\begin{proof}
Let $\pt\oto eA$ be a base point. The composition
\[
\mathcal{C}\oto{A^{*}}\mathcal{C}^{A}\oto{e^{*}}\mathcal{C}
\]
is the identity. This implies that $A^{*}$ is an equivalence if and
only if $e^{*}$ is. Moreover, it also implies that $e^{*}$ is essentially
surjective and hence an equivalence if and only if it is fully faithful. Namely, if and only if $e$ is $\mathcal{C}$-acyclic. Since $\mathcal{C}$-acyclicity
is a fiber-wise condition (\propref{Acyclicity_Fiberwise}), $e$
is $\mathcal{C}$-acyclic if and only if $\Omega A$ is $\mathcal{C}$-acyclic.
\end{proof}

\section{Height}

In this section, we introduce the notion of ``semiadditive height'' for objects in higher semiadditive $\infty$-categories, which is the  central object of study in this paper. 
We establish here the most general
properties of this notion, while those related to stability will be
differed to the next section. 

\subsection{Semiadditive Height}

The definition of height for a higher semiadditive $\infty$-category
depends on a choice of a prime $p$. In fact, it suffices to have
a certain ``$p$-typical'' version of $m$-semiadditivity, in which
one requires ambidexterity only for $m$-finite $p$-spaces. To emphasize
the relevant structure, and for some future applications, we shall
develop the basic theory of height in this level of generality. However,
for an $\infty$-category $\mathcal{C}$ which is $0$-semiadditive
and $p$-local, the ``$p$-typical'' version of higher semiadditivity
will turn out to be equivalent to ordinary higher semiadditivity (\propref{p_Local_n_Semiadd}).
Thus, for the applications considered in this paper, this point is
of minor importance. 

\subsubsection{$p$-Typical semiadditivity}

We begin with a definition of the following ``$p$-typical'' version
of higher semiadditivity:
\begin{defn}
\label{def:pTyp_Semiadd}Let $p$ be a prime and $0\le m\le\infty$.
We say that
\begin{enumerate}
\item An $\infty$-category $\mathcal{C}$ is $p$\textbf{-typically} $m$\textbf{-semiadditive}
if all $m$-finite $p$-spaces are $\mathcal{C}$-ambidextrous. 
\item A functor $F\colon\mathcal{C}\to\mathcal{D}$ between such is \textbf{$p$-typically
$m$-semiadditive} if it preserves all $m$-finite $p$-space colimits.
\item An $\mathcal{O}$-monoidal $\infty$-category $\mathcal{C}$, for
some $\infty$-operad\emph{ $\mathcal{O}$,} is \textbf{$p$-typically
$m$-semiadditively $\mathcal{O}$-monoidal} if it is $p$-typically
$m$-semiadditive and is compatible with  $m$-finite $p$-space colimits.
\end{enumerate}
We denote by $\acat^{\psad m}\subset\cat$ the subcategory of $p$-typically
$m$-semiadditive $\infty$-categories and $p$-typically $m$-semiadditive
functors. 
\end{defn}

It is clear that an $m$-semiadditive $\infty$-category or functor
are also $p$-typically $m$-semiadditive for every prime $p$. It
is useful to know that to verify $p$-typical $m$-semiadditivity,
it suffices to consider only the ``building blocks'':
\begin{prop}
\label{prop:pTyp_Semiad}Let $0\le m\le\infty$. 
\begin{enumerate}
\item An $\infty$-category $\mathcal{C}\in\acat^{\sad 0}$ is $p$-typically
$m$-semiadditive if and only if $B^{k}C_{p}$ is $\mathcal{C}$-ambidextrous
for all $k=1,\dots,m$.
\item For $\mathcal{C},\mathcal{D}\in\cat^{\psad m}$, a $0$-semiadditive
functor $F\colon\mathcal{C}\to\mathcal{D}$ is $p$-typically $m$-semiadditive
if and only if it preserves $B^{k}C_{p}$-(co)limits for all $k=1,\dots,m$.
\end{enumerate}
\end{prop}

\begin{proof}
(1) The ``only if'' part is clear. Conversely, we need to show that
if $B^{k}C_{p}$ is $\mathcal{C}$-ambidextrous for all $k=1,\dots,m$,
then every $m$-finite $p$-space $A$ is $\mathcal{C}$-ambidextrous.
Since $\mathcal{C}$ is $0$-semiadditive, we are reduced to the case
that $A$ is connected. The Postnikov tower of $A$ can be refined
to a tower of principal fibrations
\[
A=A_{r}\to\dots\to A_{1}\to A_{0}=\pt,
\]
such that the fiber of each $A_{i}\to A_{i-1}$ is of the form $B^{k_{i}}C_{p}$
for some $1\le k_{i}\le m$. To show that $A$ is $\mathcal{C}$-ambidextrous,
it suffices to show that each $A_{i}\to A_{i-1}$ is $\mathcal{C}$-ambidextrous
(as $\mathcal{C}$-ambidextrous maps are closed under composition).
Finally, since $\mathcal{C}$-ambidexterity is a fiber-wise condition,
this follows from the fact that $B^{k_{i}}C_{p}$ is $\mathcal{C}$-ambidextrous. 

(2) Follows by an analogous argument to (1).
\end{proof}
In a $p$-typically $m$-semiadditive $\infty$-category $\mathcal{C}$,
one can discuss cardinalities of $m$-finite $p$-spaces. As one might
expect, the Eilenberg-MacLane spaces $B^{n}C_{p}$ play a fundamental
role, and so deserve a special notation:
\begin{defn}
Let $\mathcal{C}\in\acat^{\psad m}$. For every integer $0\le n\le m$,
we define $\pn n^\mathcal{C}\coloneqq|B^{n}C_{p}|$ as a natural endomorphism
of the identity functor of $\mathcal{C}$. We shall omit the superscript ``$\mathcal{C}$'', whenever $\mathcal{C}$ is clear from the context.
\end{defn}

A fundamental example to keep in mind is the following:
\begin{example}
For $\Theta_{n}=\Mod_{E_{n}}(\Sp_{K(n)})$, we have by \propref{Cardinality_EM}:
\[
\pn k^{\Theta_n}\coloneqq|B^{k}C_{p}|_{n}=p^{\binom{n-1}{k}}
\]
for all $n,k\ge0$.
\end{example}

\subsubsection{Semiadditive height}

In what follows it will be convenient to use the following terminology:
\begin{defn}
\label{def:Div_Comp}Let $\mathcal{C}\in\cat$ and let $\alpha\colon\Id_{\mathcal{C}}\to\Id_{\mathcal{C}}$
be a natural endomorphism. An object $X\in\mathcal{C}$ is called
\begin{enumerate}
\item \textbf{$\alpha$-divisible} if $\alpha_{X}$ is invertible.
\item \textbf{$\alpha$-complete} if $\map(Z,X)\simeq\pt$ for all $\alpha$-divisible
$Z$. 
\end{enumerate}
We denote by $\mathcal{C}[\alpha^{-1}]$ and $\widehat{\mathcal{C}}_{\alpha}$
the full subcategories of $\mathcal{C}$ spanned by the \textbf{$\alpha$}-divisible
and $\alpha$-complete objects respectively. 
\end{defn}

Using the operations $\pn n$ we can now define the semiadditive height:
\begin{defn}
\label{def:Height_Obj}Let $\mathcal{C}\in\acat^{\psad m}$ and let
$0\le n\le m<\infty$. For every $X\in\mathcal{C}$ we define and
denote the (semiadditive) \textbf{height} of $X$ as follows:
\begin{enumerate}
\item $\htt_{\mathcal{C}}(X)\le n$, if $X$ is $\pn n^{\mathcal{C}}$-divisible. 
\item $\htt_{\mathcal{C}}(X)>n$, if $X$ is $\pn n^{\mathcal{C}}$-complete. 
\item $\htt_{\mathcal{C}}(X)=n$, if $\htt_{\mathcal{C}}(X)\le n$ and $\htt_{\mathcal{C}}(X)>n-1$. 
\end{enumerate}
We also extend the definition to $n=m=\infty$
as follows. For every $X\in\mathcal{C}$, we write $\htt_{\mathcal{C}}(X)=\infty$
if and only if $\htt_{\mathcal{C}}(X)>k$ for all $k\in\bb N$.
Additionally, by convention $-1 < \htt_{\mathcal{C}}(X) \leq \infty$ for all $X$, and $\htt_{\mathcal{C}}(X)\le-1$ or $\htt_{\mathcal{C}}(X) > \infty$
if and only if $X=0$.  We
shall drop the subscript $\mathcal{C}$ in $\htt_{\mathcal{C}}$,
when the $\infty$-category is clear from the context.
\end{defn}

\begin{rem}
We emphasize that the notation $\htt(X)\le n$ (and similarly $\htt(X)>n$
etc.) asserts that $X$ satisfies a certain property, and does not
mean that $\htt(X)$ is a well-defined number, which can be compared
with $n$. We note that the only object in $\mathcal{C}$, which can
simultaneously have height $\le n$ and $>n$ is the zero object. 
\end{rem}
The motivating example for the definition of height is the following:
\begin{example}
\label{exa:p_Complete}Let $\mathcal{C}$ be a $0$-semiadditive $\infty$-category.
An object $X\in\mathcal{C}$ is of height $0$ if and only if $p=\pn 0$
acts invertibly on $X$, and of height $>0$ if it is $p$-complete.
\end{example}

The first thing to show is that the notion of height behaves as the
terminology suggests:
\begin{prop}
\label{prop:Height_Sense}Let $\mathcal{C}\in\acat^{\psad m}$ and
let $0\le n_{0}\le n_{1}\le m$ be some integers. For every $X\in\mathcal{C}$,
\begin{enumerate}
\item If $\htt(X)\le n_{0}$, then $\htt(X)\le n_{1}$.
\item If $\htt(X)>n_{1}$, then $\htt(X)>n_{0}$.
\end{enumerate}
\end{prop}

\begin{proof}
To prove (1), it suffices to show that if $\htt(X)\le n$ for some
$n\le m-1$, then $\htt(X)\le n+1$. We consider the principal fiber
sequence
\[
B^{n}C_{p}\to\pt\to B^{n+1}C_{p}.
\]
All maps and spaces in this sequence are $\mathcal{C}$-ambidextrous
by assumption. Since $\htt(X)\le n$, we have that $|B^{n}C_{p}|_{X}$
is invertible. By \propref{Amenable_Descent}, we get 
\[
|B^{n+1}C_{p}|_{X}|B^{n}C_{p}|_{X}=|\pt|_{X}=\Id_{X}.
\]
Thus, $|B^{n+1}C_{p}|_{X}$ is invertible as well, and hence $\htt(X)\le n+1$.
Claim (2) now follows from (1) by definition. 
\end{proof}
\begin{rem}
In any stable $\infty$-category $\mathcal{C}$, a non-zero
object can have height $>0$ for at most one prime $p$. In particular,
if $\mathcal{C}$ is $p$-local, every object has height $0$ for
every prime $\ell\neq p$ (in fact, this is `if an only if').
Nevertheless, in this case the notion of \emph{$\ell$-height} with
respect to different primes $\ell$ allows one to treat ``prime to
$p$ phenomena'' as ``height $0$ phenomena'' for primes $\ell\neq p$. 
\end{rem}

It is also useful to consider the corresponding subcategories of objects
having height in a certain range:
\begin{defn}
\label{def:Height_Cat}Let $\mathcal{C}\in\acat^{\psad m}$ and let
$0\le n\le m\le\infty$. We define 
\[
\mathcal{C}_{\le n}=\mathcal{C}[\pn n^{-1}]\quad,\quad\mathcal{C}_{>n}=\widehat{\mathcal{C}}_{\pn n},
\]
\[
\mathcal{C}_{n}=\mathcal{C}_{\le n}\cap\mathcal{C}_{>n-1}=\widehat{\mathcal{C}[\pn n^{-1}]}_{\pn{n-1}}=\widehat{\mathcal{C}}_{\pn{n-1}}[\pn n^{-1}].
\]
the full subcategories of $\mathcal{C}$ spanned by objects of height
$\le n$, $>n$ and $n$ respectively. We also write $\Ht(\mathcal{C})\le n$,
$>n$ or $n$, if $\mathcal{C}=\mathcal{C}_{\le n}$, $\mathcal{C}_{>n}$
or $\mathcal{C}_{n}$ respectively. 
\end{defn}

The above defined subcategories are themselves $p$-typically $m$-semiadditive:
\begin{prop}
\label{prop:Height_Pieces_Sad}Let $\mathcal{C}\in\acat^{\psad m}$.
For every $n=0,\dots,m$, the full subcategories $\mathcal{C}_{\le n},$
$\mathcal{C}_{>n}$ and $\mathcal{C}_{n}$ are closed under limits
in $\mathcal{C}$. In particular, they are $p$-typically $m$-semiadditive,
and are furthermore $m$-semiadditive if $\mathcal{C}$ is.
\end{prop}

\begin{proof}
An object $X\in\mathcal{C}$ belongs to $\mathcal{C}_{\le n}$ if
and only if $X$ is $\pn n$-divisible. Thus $\mathcal{C}_{\le n}$
is closed under all limits which exist in $\mathcal{C}$. By definition,
$\mathcal{C}_{>n}$ and $\mathcal{C}_{\infty}$ are closed under limits
in $\mathcal{C}$ as well. Thus, for $n<\infty$ we have that $\mathcal{C}_{n}=\mathcal{C}_{\le n}\cap\mathcal{C}_{>n-1}$
is also closed under limits in $\mathcal{C}$. Finally, by \propref{Ambi_Closure}(3),
it follows that all these subcategories are $p$-typically $m$-semiadditive
and are furthermore $m$-semiadditive if $\mathcal{C}$ is.
\end{proof}
Next, we consider the behavior of height with respect to higher semiadditive
functors. It turns out that the height can only go down:
\begin{prop}
\label{prop:Height_Functoriality}Let $F\colon\mathcal{C}\to\mathcal{D}$
be a map in $\acat^{\psad m}$. For all $X\in\mathcal{C}$ and $0\le n\le m$,
if $\htt_{\mathcal{C}}(X)\le n$ then $\htt_{\mathcal{D}}(F(X))\le n$.
If $F$ is conservative, then the converse holds as well. 
\end{prop}

\begin{proof}
This follows immediately from the fact that $F$ maps $\pn n^{\mathcal{C}}$
to $\pn n^{\mathcal{D}}$.
\end{proof}
In contrast, the following example shows that a higher semiadditive
functor need not preserve \emph{lower} \emph{bounds} on height:
\begin{example}
\label{exa:Height_Decrease}The $0$-semiadditive functor $L_{\bb Q}\colon\Sp_{(p)}\to\Sp_{\bb Q}$
maps the $p$-complete sphere $\widehat{\bb S}_{p}$, which is of
height $>0$, to a non-zero object $\bb Q\otimes\widehat{\bb S}_{p}$
of height $0$. 
\end{example}

For an inclusion of a full subcategory, we can say a bit more:
\begin{prop}
\label{prop:Height_Subcategory}Let $\mathcal{C}\in\acat^{\psad m}$
and let $\mathcal{C}^{\prime}\ss\mathcal{C}$ be a full subcategory
closed under $m$-finite $p$-space (co)limits. Given $X\in\mathcal{C}^{\prime}$
and $0\le n\le m$, we have
\begin{enumerate}
\item $\htt_{\mathcal{C}^{\prime}}(X)\le n$ if and only if $\htt_{\mathcal{C}}(X)\le n$.
\item $\htt_{\mathcal{C}}(X)>n$ implies $\htt_{\mathcal{C}^{\prime}}(X)>n$.
\end{enumerate}
\end{prop}

\begin{proof}
(1) follows from \propref{Height_Functoriality} applied to the inclusion
$\mathcal{C}^{\prime}\into\mathcal{C}$. For (2), if $\htt_{\mathcal{C}}(X)>n$,
then for every $Z\in\mathcal{C}_{\le n}$, we have $\map(Z,X)\simeq\pt$.
We now observe that by (1), we have $\mathcal{C}_{\le n}^{\prime}=\mathcal{C}^{\prime}\cap\mathcal{C}_{\le n}$.
Thus, for every $Z^{\prime}\in\mathcal{C}_{\le n}^{\prime}$, we have
$\map(Z^{\prime},X)\simeq\pt$, which by definition means $\htt_{\mathcal{C}^{\prime}}(X)>n$.
\end{proof}
In presence of a $p$-typically $m$-semiadditively monoidal structure,
an upper bound on the height of the unit implies an upper bound on
the height of the $\infty$-category:
\begin{cor}
\label{cor:Height_Monoidal}Let $\mathcal{C}$ be $p$-typically $m$-semiadditively
monoidal $\infty$-category. For every $0\le n\le m$, we have $\Ht(\mathcal{C})\le n$
if and only if $\htt_{\mathcal{C}}(\one)\le n$.
\end{cor}

\begin{proof}
Given $X\in\mathcal{C}$, the functor $X\otimes(-)\colon\mathcal{C}\to\mathcal{C}$
is $p$-typically $m$-semiadditive. Thus, the claim follows from
\propref{Height_Functoriality}.
\end{proof}

\subsection{Bounded Height}

In this subsection, we study the implications for a higher semiadditive
$\infty$-category of having bounded height. These results generalize
the previously discussed facts regarding the $\infty$-semiadditive
structure of semirational $\infty$-categories (i.e. of height $0$),
and fall under the slogan that ``the higher semiadditive structure
is trivial above the height''. As a by-product, we shall see that
for a $0$-semiadditive $p$-local $\infty$-category, there is no
difference between $m$-semiadditivity and $p$-typical $m$-semiadditivity. 

\subsubsection{Amenability \& acyclicity}

The results on amenability and acyclicity from \subsecref{Acyclicity}
imply the following equivalent characterizations of height
$\le n$:
\begin{prop}
\label{prop:Height_Characterization}Let $\mathcal{C}\in\acat^{\psad n}$,
which admits $B^{n+1}C_{p}$-(co)limits. The following properties
are equivalent:
\begin{enumerate}
\item $\Ht(\mathcal{C})\le n$ (i.e. $B^{n}C_{p}$ is $\mathcal{C}$-amenable).
\item $B^{n+1}C_{p}$ is $\mathcal{C}$-acyclic.
\item $B^{n+2}C_{p}$ is $\mathcal{C}$-trivial.
\end{enumerate}
\end{prop}

We can therefore characterize the height of an $\infty$-category
in ways that do not make an explicit reference to the higher semiadditive
structure.
\begin{proof}
The equivalence of (1) and (2) follows from \corref{Amenable_Acyclic}
and the equivalence of (2) and (3) from \propref{Acyclic_Trivial}.
\end{proof}
The following can be seen as a ($p$-typical) generalization of the
fact that a semirational $\infty$-category is automatically $\infty$-semiadditive: 
\begin{prop}
\label{prop:Height_Bootstrap}Let $\mathcal{C}\in\acat^{\psad n}$
such that $\Ht(\mathcal{C})\le n$, and assume $\mathcal{C}$ admits
$B^{n+1}C_{p}$-(co)limits. Then, $\mathcal{C}$ is $p$-typically
$\infty$-semiadditive.
\end{prop}

\begin{proof}
By \propref{pTyp_Semiad}(1), it suffice to prove that 
\begin{itemize}
\item [$(*)$]$B^{k}C_{p}$ is $\mathcal{C}$-amenable (and ,in particular,
$\mathcal{C}$-ambidextrous) and $\mathcal{C}$ admits all $B^{k+1}C_{p}$-colimits.
\end{itemize}
holds for all $k\ge n$. We shall prove this by induction on $k$.
The base case $k=n$ is given by assumption. Assume $(*)$ holds for
some $k\ge n$ and consider the fiber sequence 
\[
B^{k}C_{p}\to\pt\to B^{k+1}C_{p}.
\]
By the inductive hypothesis, $|B^{k}C_{p}|$ is invertible and $\mathcal{C}$
admits $B^{k+1}C_{p}$-(co)limits. Therefore by \cite[Propostion 3.1.17]{Ambi2018}, the space
$B^{k+1}C_{p}$ is $\mathcal{C}$-ambidextrous. Moreover, by \propref{Amenable_Descent},
$|B^{k+1}C_{p}|$ is invertible as well. Finally, by \propref{Height_Characterization}(3),
the diagonal functor $\mathcal{C}\to\mathcal{C}^{B^{k+2}C_{p}}$ is
an equivalence and hence in particular $\mathcal{C}$ admits $B^{k+2}C_{p}$-(co)limits. 
\end{proof}
We now show that claims (1)-(3) of \propref{Height_Characterization}
extend to much wider classes of spaces:
\begin{prop}
\label{prop:Height_Everything}Let $\mathcal{C}\in\cat^{\psad{\infty}}$
such that $\Ht(\mathcal{C})\le n$. For every $\pi$-finite $p$-space
$A$, the following hold:
\begin{enumerate}
\item If $A$ is $(n-1)$-connected, then $A$ is $\mathcal{C}$-amenable. 
\item If $A$ is $n$-connected, then $A$ is $\mathcal{C}$-acyclic. 
\item If $A$ is $(n+1)$-connected, then $A$ is $\mathcal{C}$-trivial.
\end{enumerate}
\end{prop}

\begin{proof}
For (1), let $A$ be an $(n-1)$-connected $\pi$-finite $p$-space.
The Postnikov tower of $A$ can be refined to a tower of principal
fibrations 
\[
A=A_{r}\to\dots\to A_{1}\to A_{0}=\pt,
\]
such that the fiber of each $A_{i}\to A_{i-1}$ is of the form $B^{k_{i}}C_{p}$
for some $k_{i}\ge n$. Since we assumed $\Ht(\mathcal{C})\le n$,
we also have $\Ht(\mathcal{C})\le k_{i}$ (\propref{Height_Sense}),
and hence all the spaces $B^{k_{i}}C_{p}$ are $\mathcal{C}$-amenable.
By \corref{Amenable_Extensions}, the class of $\mathcal{C}$-amenable
spaces is closed under principal extensions, and therefore $A$ is
$\mathcal{C}$-amenable. Now, (2) follows from \corref{Amenable_Acyclic}
and (1) applied to $\Omega A$. Similarly, (3) follows from \propref{Acyclic_Trivial}
and (2) applied to $\Omega A$.
\end{proof}
\begin{example}
Let $\mathcal{C}$ be semirational, and so in particular of height $0$
(such as $\mathrm{Vec}_{\bb Q}$ or $\Sp_{\bb Q}$). By \exaref{Homotopy_Cardinality},
for every $\pi$-finite $p$-space $A$, the cardinality $|A|$ is
a sum of positive rational numbers. Thus, $|A|$ is invertible if
and only if $A$ is \emph{non-empty} (i.e. $(-1)$-connected). The
map $X\to X^{A}$ is an equivalence for all $X$, if and only if $A$
is \emph{connected},\emph{ }and $\mathcal{C}\to\mathcal{C}^{A}$ is
an equivalence if and only if $A$ is \emph{simply-connected}. 
\end{example}

\subsubsection{Cardinality}

Recall from \exaref{Homotopy_Cardinality}, that the formula for the
homotopy cardinality (\ref{eq:Rational_Cardinality}) could be deduced
solely from the ``multiplicativity property'' with respect to fiber
sequences. For higher heights, we have the following analogue: 
\begin{prop}
\label{prop:Height_Cardinality}Let $\mathcal{C}\in\acat^{\psad{\infty}}$
such that $\Ht(\mathcal{C})\le n$. Given a principal fibration of
$\pi$-finite $p$-spaces 
\[
F\to A\to B,
\]
if $F$ is $(n-1)$-connected, then $|A|=|F|\cdot|B|$. In particular,
\[
\pn k=\pn n^{(-1)^{k-n}},\quad\forall k\ge n.
\]
\end{prop}

\begin{proof}
By \propref{Height_Everything}(1), the space $F$ is $\mathcal{C}$-amenable.
Thus, by \propref{Amenable_Descent}, we have $|A|=|F|\cdot|B|$.
By an inductive application of this to the fiber sequence
\[
B^{k}C_{p}\to\pt\to B^{k+1}C_{p},
\]
we obtain the formula $\pn k=\pn n^{(-1)^{k-n}}$ for all $k\ge n$.
\end{proof}
Furthermore, using \propref{Height_Cardinality} together with a principal
refinement of the Postnikov tower (as in the proof of \propref{Height_Bootstrap}),
one can reduce the computation of the $\mathcal{C}$-cardinality of
all $\pi$-finite $p$-spaces to those of connected $n$-finite $p$-spaces.

\subsubsection{The $p$-local Case}

In many situations of interest, the $\infty$-category $\mathcal{C}$
under consideration is $0$-semiadditive, $p$-local and admits all
$1$-finite colimits. In this case, $\mathcal{C}$ is automatically
$\ell$-typically $\infty$-semiadditive for all primes $\ell\neq p$
(\propref{Height_Bootstrap}). Moreover, in this case there is essentially
no difference between higher semiadditivity and $p$-typical higher
semiadditivity:
\begin{prop}
\label{prop:p_Local_n_Semiadd}Let $\mathcal{C}$ be a $0$-semiadditive
$p$-local $\infty$-category which admits all $1$-finite limits
and colimits. The $\infty$-category $\mathcal{C}$ is $p$-typically
$m$-semiadditive if and only if it is $m$-semiadditive. 
\end{prop}

\begin{proof}
Assume that $\mathcal{C}$ is $p$-typically $m$-semiadditive. By
\cite[Proposition 4.4.16]{HopkinsLurie} and the fact that the space $BC_{p}$ is $\mathcal{C}$-ambidextrous,
we get that $\mathcal{C}$ is $1$-semiadditive. To get higher semiadditivity, we first have by assumption that
$\mathcal{C}$ is $p$-typically $m$-semiadditive. Moreover, since
$\mathcal{C}$ is $p$-local, it is $\ell$-typically $m$-semiadditive
for all $\ell\neq p$ (\propref{Height_Bootstrap}). Thus, $B^{k}C_{\ell}$
is $\mathcal{C}$-ambidextrous for all primes $\ell$ and all integers
$k=2,\dots m$. By inductive application of \cite[Proposition 4.4.19]{HopkinsLurie}, it follows
that $\mathcal{C}$ is $m$-semiadditive. We note that in both \cite[Proposition 4.4.16]{HopkinsLurie}
and \cite[Proposition 4.4.19]{HopkinsLurie}, one assumes that $\mathcal{C}$ admits all small limits
and colimits. However, the proofs use only the limits and colimits
which we assumed in the statement. 
\end{proof}
When applied to $p$-local $\infty$-categories, the main results
of this subsection can be summarized as follows: 
\begin{thm}
\label{thm:Height_Everything_p_Local}Let $\mathcal{C}$ be a $0$-semiadditive
$p$-local $\infty$-category, which admits all $(n+1)$-finite limits
and colimits. If $\mathcal{C}$ is $p$-typically $n$-semiadditive
such that $\Ht(\mathcal{C})\le n$, then $\mathcal{C}$ is $\infty$-semiadditive.
Moreover, for every $\pi$-finite space $A$:
\begin{enumerate}
\item If $A$ is $(n-1)$-connected and nilpotent, then $A$ is $\mathcal{C}$-amenable. 
\item If $A$ is $n$-connected, then $A$ is $\mathcal{C}$-acyclic. 
\item If $A$ is $(n+1)$-connected, then $A$ is $\mathcal{C}$-trivial.
\end{enumerate}
\end{thm}

\begin{proof}
Since $\Ht(\mathcal{C})\le n$ and $\mathcal{C}$ admits $B^{n+1}C_{p}$-limits
and colimits, it follows that $\mathcal{C}$ is $p$-typically $\infty$-semiadditive
(\propref{Height_Bootstrap}). Since $\mathcal{C}$ is $p$-local,
it follows that $\mathcal{C}$ is in fact $\infty$-semiadditive (\propref{p_Local_n_Semiadd}).
For (1), we observe that if $A$ is $\pi$-finite and nilpotent, then
$A=\prod_{\ell}A_{(\ell)}$ where $\ell$ ranges over primes and $A_{(\ell)}$
is a $\pi$-finite $\ell$-space which is contractible for almost
all $\ell$ (e.g. \cite[Theorem 5.7]{Sylow2017}). Since we have $|A|=\prod_{\ell}|A_{(\ell)}|$
(\remref{Cardinality_Arithmetic}), the $\mathcal{C}$-amenability
of $A$ follows from the $\mathcal{C}$-amenability of all the $A_{(\ell)}$-s
(see \propref{Height_Everything}(1)). Finally, as in the proof of
\propref{Height_Everything}, (2) follows from (1) and \corref{Amenable_Acyclic}
applied to $\Omega A$ and (3) follows from (2) and \propref{Acyclic_Trivial}.
\end{proof}
\begin{rem}
We note that the nilpotence condition in \thmref{Height_Everything_p_Local}(1)
is vacuous for simply connected spaces and hence relevant only for
height $n=1$. However, in this case it can not be dropped. Indeed,
by \exaref{Amenable_Cancelation}, at the prime $p=2$ we have $|B\Sigma_{3}|_{1}=\frac{2}{3},$
which is not-invertible. With a little more effort one can show that
if $\Ht(\mathcal{C})\le1$, then $|A|$ is invertible for every connected
$\pi$-finite space $A$, such that that the $p$-Sylow subgroup of
$\pi_{1}A$ is normal. In other words, the $p$-primary \emph{fusion} in
the fundamental group is the only obstruction for the invertibility
of $|A|$. 
\end{rem}

\subsection{Semiadditive Redshift}

By \exaref{Categorical_Cardinality}, the $\infty$-category $\acat_{\fin m}$
is $m$-semiadditive and hence given $\mathcal{C}\in\acat_{\fin m}$,
we can discuss $\htt(\mathcal{C})$ for various primes $p$, which
is the height of $\mathcal{C}$ as an object of $\acat_{\fin m}$.
However, if $\mathcal{C}$ itself is $p$-typically higher semiadditive,
then we have also defined $\Ht(\mathcal{C})$, as the height of the objects of $\mathcal{C}$. These two notions of height do \emph{not}
coincide. Rather, we shall now show that $\htt(\mathcal{C})$ exceeds
$\Ht(\mathcal{C})$ exactly by one. As the semiadditive height generalizes
the chromatic height, this can be viewed as a particular manifestation
of the ``redshift'' principle. Roughly speaking, categorification
tends to shift the height up by one. Before we begin, we need a general
categorical lemma:
\begin{lem}
\label{lem:Invertible_Monad}Let $F\colon\mathcal{C}\adj\mathcal{D}\colon G$
be an adjunction. If $GF$ is an equivalence, then $F$ is fully faithful. 
\end{lem}

\begin{proof}
The functor $F$ is fully faithful if and only if the unit $\Id\oto uGF$
is an isomorphism. The unit is part of a monad structure on $GF$
making it a monoid in the homotopy category of $\fun(\mathcal{C},\mathcal{C})$,
whose monoidal structure is given by composition. It therefore suffices
to show that given a monoid $M$ in any monoidal (ordinary) category,
if $M$ is invertible with respect to the monoidal structure, then
the unit map $\one\oto uM$ is an isomorphism. We observe that the
multiplication map $M\otimes M\oto mM$ is always a left inverse of
$1\otimes u$. Thus, since $M\otimes(-)$ is an equivalence, it follows
that $u$ admits a left inverse as well. To conclude the proof, it
remains to show that $u$ admits a right inverse. Since $M$ is invertible,
it is in particular dualizable, and the dual $M^{\vee}$ of $M$ is
its inverse. More precisely, the duality datum maps $\one\oto{\eta}M\otimes M^{\vee}$
and $M^{\vee}\otimes M\oto{\varepsilon}\one$ are isomorphisms and
exhibit $M^{\vee}$ as the inverse of $M$. Consider the following
commutative diagram
\[
\xymatrix{M\otimes\one\ar[d]_{1\otimes\eta\otimes1}^{\wr}\ar[rr]^{1\otimes u} &  & M\otimes M\ar[d]_{\wr}^{1\otimes\eta\otimes1}\ar@{=}[rrd]\\
M\otimes M\otimes M^{\vee}\otimes\one\ar[d]_{m\otimes1\otimes1}\ar[rr]^{1\otimes1\otimes1\otimes u} &  & M\otimes M\otimes M^{\vee}\otimes M\ar[d]^{m\otimes1\otimes1}\ar[rr]^{1\otimes1\otimes\varepsilon} &  & M\otimes M\ar[d]^{m}\\
M\otimes M^{\vee}\otimes\one\ar[rr]^{1\otimes1\otimes u} &  & M\otimes M^{\vee}\otimes M\ar[rr]_{\sim}^{1\otimes\varepsilon} &  & M
}
\]
The clockwise composition is the identity and hence so is the counter-clockwise
composition. It follows that the map
\[
M\otimes M^{\vee}\otimes\one\oto{1\otimes1\otimes u}M\otimes M^{\vee}\otimes M
\]
admits a right inverse. Since the functor $M\otimes M^{\vee}\otimes(-)$
is an equivalence, it follows that $u$ admits a right inverse.
\end{proof}
We can now state and prove the following ``Semiadditive Redshift''
result, which can be informally summarized as ``$\htt=\Ht+1$''.
\begin{thm}
[Semiadditive Redshift]\label{thm:Redshift} Let $0\le n\le m \leq \infty$ be
integers and let $\mathcal{C}\in\acat_{\fin{(m+1)}}$. If $\mathcal{C}$
is $p$-typically $m$-semiadditive, then 
\begin{enumerate}
\item $\Ht(\mathcal{C})\le n$ if and only if $\htt(\mathcal{C})\le n+1$. 
\item $\Ht(\mathcal{C})>n$ if and only if $\htt(\mathcal{C})>n+1$.
\end{enumerate}
In particular, an $\infty$-semiadditive $\infty$-category is of
height $n$ if and only if, as an object of $\acat_{\infty\text{-}\mathrm{fin}}$
(and hence also $\acat^{\sad{\infty}}$), it is of height $n+1$. 
\end{thm}

\begin{proof}
Denote $B=B^{n+1}C_{p}$. For (1), we observe that $\Ht(\mathcal{C})\le n$
if and only if $B$ is $\mathcal{C}$-acyclic (\propref{Height_Characterization}),
namely that $B^{*}$ is fully faithful. On the other hand, the natural
endomorphism $\pn{n+1}=|B|$ of the identity functor of $\acat_{\fin{(m+1)}}$,
acts on $\mathcal{C}$ by the functor $B_{!}B^{*}\colon\mathcal{C}\to\mathcal{C}$
(\exaref{Categorical_Cardinality}). Thus, $\htt(\mathcal{C})\le n+1$,
if and only if $B_{!}B^{*}$ is invertible. Now, if $B^{*}$ is fully
faithful, then $B_{!}B^{*}=\Id_{\mathcal{C}}$ and is in particular
invertible. On the other hand, if $B_{!}B^{*}$ is invertible, then
$B^{*}$ is fully faithful by the dual of \lemref{Invertible_Monad}.
Thus, $\Ht(\mathcal{C})\le n$ if and only if $\htt(\mathcal{C})\le n+1$.

For (2), we first assume that $\Ht(\mathcal{C})>n$, and show that
$\htt(\mathcal{C})>n+1$. Given $\mathcal{D}\in\acat_{\fin{(m+1)}}$
such that $\htt(\mathcal{D})\le n+1$, we need to show that every
$(m+1)$-finite colimit preserving functor $F\colon\mathcal{D}\to\mathcal{C}$
must be zero. For every $X\in\mathcal{D}$, we have $X[B]=B_{!}B^{*}X\iso X$.
Since $F$ is $(m+1)$-finite colimit preserving, we get also $F(X)[B]\iso F(X)$.
Thus, by \propref{Amenable_Acyclic_Object}, we get $\htt(F(X))\le n$.
Since $\Ht(\mathcal{C})>n$, the only object of height $\le n$ in
$\mathcal{C}$ is zero and hence $F(X)=0$. Thus, $F$ is zero, which
proves that $\htt(\mathcal{C})>n+1$. Conversely, assume $\htt(\mathcal{C})>n+1$,
to show that $\Ht(\mathcal{C})>n$, consider the full subcategory
$\mathcal{C}_{\le n}\ss\mathcal{C}$ spanned by objects of height
$\le n$ in $\mathcal{C}$, which is also $p$-typically $m$-semiadditive
(\propref{Height_Pieces_Sad}). By definition, $\Ht(\mathcal{C}_{\le n})\le n$,
and hence by (1), $\htt(\mathcal{C}_{\le n})\le n+1$, so the inclusion
functor $\mathcal{C}_{\le n}\into\mathcal{C}$ must be zero. It follows
that $\mathcal{C}_{\le n}$ is zero and therefore $\Ht(\mathcal{C})>n$. 

It follows from (1) and (2) that if $\Ht(\mathcal{C})=n$, then $\htt(\mathcal{C})=n+1$
when $\mathcal{C}$ is considered as an object of $\acat_{\infty\text{-}\mathrm{fin}}$.
The parenthetical remark follows from \propref{Height_Subcategory} and \propref{Sad_Sad}.
\end{proof}
Recall from \defref{CoCart_CMon}, that for every $\mathcal{C}\in\acat^{\sad m}$,
the space of objects $\mathcal{C}^{\simeq}$ is endowed with an $m$-commutative
coCartesian monoid structure making it an object of the $m$-semiadditive
$\infty$-category $\CMon_{m}$. \thmref{Redshift} has the following
corollary:
\begin{cor}
\label{cor:CMon_Redshift}Let $\mathcal{C}\in\acat^{\sad{\infty}}$,
such that $\Ht(\mathcal{C})\le n$. The space of objects $\mathcal{C}^{\simeq}$
with the higher coCartesian structure
satisfies $\htt(\mathcal{C}^{\simeq})\le n+1$, as an object of $\CMon_{\infty}$.
\end{cor}

\begin{proof}
We have seen in \thmref{Redshift}, that when we consider $\mathcal{C}$
as an object of $\acat_{\fin{\infty}}$, we have $\htt(\mathcal{C})\le n+1$.
The space of objects $\mathcal{C}^{\simeq}\in\CMon_{\infty}$ can
be identified with the image of $\mathcal{C}$ under the $\infty$-semiadditive
functor
\[
\hom(\mathcal{S}_{\fin{\infty}},-)\colon\acat_{\fin{\infty}}\to\CMon_{\infty}.
\]
Thus, by \propref{Height_Functoriality}, we have $\htt(\mathcal{C}^{\simeq})\le n+1$. 
\end{proof}
\begin{example}
\label{exa:Chromatic_Redshift}Let $R$ be a $T(n)$-local ring spectrum.
By \corref{Height_Modules} and \thmref{Height_Chrom}, the $\infty$-commutative monoid $\Mod_{R}(\Sp_{T(n)})^{\simeq}$
is of height $\le n+1$ in $\CMon_{\infty}$. In particular, this
applies to Morava $E$-theory $R=E_{n}$. This suggests a relation
between the ``semiadditive redshift'' of \thmref{Redshift} and
the ``chromatic redshift'' in algebraic $K$-theory of Ausani-Rognes
(see \cite{RedshiftRognes,RedshiftGood}). We shall explore this connection further in a future work.
\end{example}

The proof of \thmref{Redshift} relies ultimately on the fact that
$B^{n}C_{p}$ is $\mathcal{C}$\emph{-amenable} if and only if $B^{n+1}C_{p}$
is $\mathcal{C}$\emph{-acyclic} (\corref{Amenable_Acyclic}). In
\propref{Acyclic_Trivial} we categorified this fact by showing that
$B^{n+1}C_{p}$ is $\mathcal{C}$\emph{-acyclic}, if and only if $B^{n+2}C_{p}$
is $\mathcal{C}$\emph{-trivial}. Similarly, \thmref{Redshift} can
be categorified as follows. Let $\acat^{n\text{-}\mathrm{ht}}\ss\acat_{\infty\text{-}\mathrm{fin}}$
be the full subcategory spanned by the $\infty$-semiadditive $\infty$-categories
$\mathcal{C}$ such that $\Ht(\mathcal{C})=n$. 
\begin{lem}
\label{lem:Cat_Ht_Colim}The full subcategory $\acat^{n\text{-}\mathrm{ht}}\ss\acat_{\infty\text{-}\mathrm{fin}}$
is closed under $\pi$-finite colimits.
\end{lem}

\begin{proof}
We have 
\[
\acat^{n\text{-}\mathrm{ht}}\ss\acat^{\sad{\infty}}\ss\acat_{\infty\text{-}\mathrm{fin}}.
\]

By \thmref{Redshift}, we have \[\acat^{n\text{-}\mathrm{ht}}=(\acat^{\sad{\infty}})_{n+1}.\]
Thus, by \propref{Height_Pieces_Sad}, $\acat^{n\text{-}\mathrm{ht}}$
is closed under limits in $\acat^{\sad{\infty}}$. Additionally, $\acat^{\sad{\infty}}$
is closed under limits in $\acat_{\infty\text{-}\mathrm{fin}}$ by
\propref{Sad_Sad}. Therefore, $\acat^{n\text{-}\mathrm{ht}}$ is
closed under limits in $\acat_{\infty\text{-}\mathrm{fin}}$. Since
$\acat_{\infty\text{-}\mathrm{fin}}$ is $\infty$-semiadditive, it
follows that $\acat^{n\text{-}\mathrm{ht}}$ is also closed under
$\pi$-finite colimits in $\acat_{\infty\text{-}\mathrm{fin}}$.
\end{proof}
Hence, in particular, $\acat^{n\text{-}\mathrm{ht}}$ admits $\pi$-finite
colimits and is therefore an object\emph{ }of the $\infty$-semiadditive
$\infty$-category $\widehat{\acat}_{\infty\text{-}\mathrm{fin}}$
of large $\infty$-categories, which admit $\pi$-finite colimits
and functors preserving them. 
\begin{prop}
\label{prop:Categorified_Redshift}The $\infty$-category $\acat^{n\text{-}\mathrm{ht}}$
is an object of height $n+2$ in $\widehat{\acat}_{\infty\text{-}\mathrm{fin}}$.
\end{prop}

\begin{proof}
By \lemref{Cat_Ht_Colim}, $\acat^{n\text{-}\mathrm{ht}}$ is closed
under $\pi$-finite colimits in $\acat_{\infty\text{-}\mathrm{fin}}$
and hence is $\infty$-semiadditive by \propref{Ambi_Closure}. Thus,
by \thmref{Redshift}, it suffices to show that for every $\mathcal{C}\in\acat^{n\text{-}\mathrm{ht}}$
we have $\htt(\mathcal{C})=n+1$ as an object of $\acat_{\infty\text{-}\mathrm{fin}}$,
and hence also as an object of $\acat^{n\text{-}\mathrm{ht}}$ (\propref{Height_Subcategory}).
This follows again from \thmref{Redshift} and the fact that $\Ht(\mathcal{C})=n$. 
\end{proof}

\section{Stability}

So far, we have been considering general higher semiadditive $\infty$-categories.
In this section, we specialize to the \emph{stable} world. First, using
the general results on height from the previous section, we shall
show that every stable higher semiadditive $\infty$-category decomposes
completely according to height. Second, inspired by \cite{LurieRep}, we study  semisimplicity properties of local systems valued in general stable $\infty$-categories
of semiadditive height $n$. Finally, we show that $\Sp_{K(n)}$ and
$\Sp_{T(n)}$ are indeed of semiadditive height $n$.

\subsection{Recollement}

A central tool, which will be used several times in the following subsections, is that of a \emph{recollement}
of stable $\infty$-categories following \cite{barwick2016note}. In this preliminary
subsection, we collect several general facts regarding this notion.
First, we provide criteria for a recollement to be \emph{split}, in the sense
that it has trivial ``gluing data''. Second, we show that a decreasing
intersection of a chain of recollements is again a recollement. Finally, we give special attention
to recollements arising from ``divisible'' and ``complete'' objects
with respect to a natural endomorphism of the identity functor in the sense of \defref{Div_Comp}. 

\subsubsection{(Split) Recollement}

We begin by recalling the notion of recollement in the context of
stable $\infty$-categories following the exposition in \cite[Section A.8.1]{ha}. 
\begin{defn}
\label{def:Orthogonal_Complement}Let $\mathcal{C}$ be a stable $\infty$-category
and $\mathcal{C}_{\circ}\ss\mathcal{C}$ a full stable subcategory.
We define the \textbf{right} \textbf{orthogonal} \textbf{complement}
$\mathcal{C}_{\circ}^{\bot}\ss\mathcal{C}$ to be the full subcategory
consisting of objects $Y\in\mathcal{C},$ such that $\map(X,Y)\simeq\pt$
for all $X\in\mathcal{C}_{\circ}$. 
\end{defn}

Recall from \cite[Proposition A.8.20]{ha}, that if the inclusion $\mathcal{C}_{\circ}\into\mathcal{C}$
admits both a left adjoint $L$  and a right adjoint $R$, then $\mathcal{C}$
is a recollement of $\mathcal{C}_{\circ}$ and $\mathcal{C}_{\circ}^{\perp}$
in the sense of \cite[Section A.8.1]{ha}. In particular, the $\infty$-category $\mathcal{C}$
can be identified with the $\infty$-category of sections of the cartesian
fibration over $\Delta^{1}$, classified by the functor $L|_{\mathcal{C}_{\circ}^{\perp}}\colon\mathcal{C}_{\circ}^{\perp}\to\mathcal{C}_{\circ}.$
\begin{defn}
We shall say that an inclusion of stable $\infty$-categories $\mathcal{C}_{\circ}\into\mathcal{C}$,
that admits both a left and a right adjoint, exhibits $\mathcal{C}$
as a \textbf{recollement} of $\mathcal{C}_{\circ}$ and $\mathcal{C}_{\circ}^{\perp}$.
\end{defn}

Recollements in stable homotopy theory typically arise from smashing
localizations, and amount to the existence of various fracture squares:
\begin{example}
[Arithmetic and chromatic squares]\label{exa:Chromatic_Recollement}For
$\mathcal{C}=\Sp_{(p)}$, the inclusion of the full subcategory $\mathcal{C}_{\circ}=\Sp_{\bb Q}\ss\Sp_{(p)}$
admits both a left and a right adjoint, and we have 
\[
\Sp_{\bb Q}^{\bot}=\widehat{\Sp}_{p}\ss\Sp_{(p)},
\]
is the full subcategory spanned by $p$-complete spectra. The recollement
statement in this case recovers the classical $p$-local arithmetic
square for spectra. More generally, the full subcategory 
\[
L_{n}^{f}\Sp\coloneqq\Sp_{\oplus_{k=0}^{n}T(k)}\ss\Sp_{(p)}
\]
exhibits $\Sp_{(p)}$ as a recollement of $L_{n}^{f}\Sp$ and $\Sp_{F(n+1)}$,
where $F(n+1)$ is a finite spectrum of type $n+1$. In particular,
the inclusion $L_{n-1}^{f}\Sp\ss L_{n}^{f}\Sp$ exhibits $L_{n}^{f}\Sp$
as a recollement of $L_{n-1}^{f}\Sp$ and $\Sp_{T(n)}$. This recovers
the classical telescopic fracture square at height $n$. By the Smash
Product Theorem, 
\[
L_{n}\Sp\coloneqq\Sp_{\oplus_{k=0}^{n}K(k)}\ss\Sp_{(p)}
\]
is also a smashing localization, and similarly, the inclusion $L_{n-1}\Sp\ss L_{n}\Sp$
exhibits $L_{n}\Sp$ as a recollement of $L_{n-1}\Sp$ and $\Sp_{K(n)}$.
This recovers the classical chromatic fracture square at height $n$. 
\end{example}

Given a recollement $\mathcal{C}_{\circ}\ss\mathcal{C}$, the functor
$L|_{\mathcal{C}_{\circ}^{\bot}}\colon\mathcal{C}_{\circ}^{\bot}\to\mathcal{C}_{\circ}$
encodes the ``gluing data'' in the construction of $\mathcal{C}$
from the $\infty$-categories $\mathcal{C}_{\circ}$ and $\mathcal{C}_{\circ}^{\bot}$.
A particularly simple instance of a recollement is when this gluing
data is trivial:
\begin{prop}
\label{prop:Split_Recollement}Given a recollement $\mathcal{C}_{\circ}\ss\mathcal{C}$,
the following are equivalent:
\begin{enumerate}
\item The functor $L|_{\mathcal{C}_{\circ}^{\bot}}\colon\mathcal{C}_{\circ}^{\bot}\to\mathcal{C}_{\circ}$
is zero. 
\item The left adjoints of the inclusions $\mathcal{C}_{\circ},\mathcal{C}_{\circ}^{\perp}\ss\mathcal{C}$
induce an equivalence $\mathcal{C}\iso\mathcal{C}_{\circ}\times\mathcal{C}_{\circ}^{\bot}$.
\item The left and right adjoints $L$ and $R$ of $\mathcal{C}_{\circ}\into\mathcal{C}$
are isomorphic.
\end{enumerate}
\end{prop}

\begin{proof}
We prove (1)$\implies$(2)$\implies$(3)$\implies$(1). The implication
(1)$\implies$(2) follows from the identification of $\mathcal{C}$
with the $\infty$-category of sections of the cartesian fibration
classified by $L|_{\mathcal{C}_{\circ}^{\bot}}$. For the implication
(2)$\implies$(3), observe that the inclusion $\mathcal{C}_{\circ}\into\mathcal{C}_{\circ}\times\mathcal{C}_{\circ}^{\bot}$
can be identified with the functor $(\Id,0)$, for which the projection
$\mathcal{C}_{\circ}\times\mathcal{C}_{\circ}^{\bot}\to\mathcal{C}_{\circ}$
is both a left and a right adjoint. Finally, for every $X\in\mathcal{C}_{\circ}^{\bot}$,
we have $R(X)=0$. Thus, assuming (3) we have $R\simeq L$ and so
$L|_{\mathcal{C}_{\circ}^{\bot}}=0$, which proves (1).
\end{proof}
\begin{defn}
\label{def:Split_recollement}We say that a recollement $\mathcal{C}_{\circ}\into\mathcal{C}$
is \textbf{split} if it satisfies the equivalent conditions of \propref{Split_Recollement}.
\end{defn}

\subsubsection{Recollement chains}

We shall now study the behavior of \emph{chains} of recollements. 
\begin{defn}
\label{def:Recollement_Chain}For $\mathcal{C}\in\acat_{\st}$, we
say that a descending chain of full subcategories, 
\[
\dots\subseteq\mathcal{C}_{(n)}\ss\dots\subseteq{\cal C}_{(2)}\subseteq{\cal C}_{(1)}\subseteq{\cal C}_{(0)}\ss{\cal C}
\]
is a \textbf{recollement} \textbf{chain}, if each inclusion $\mathcal{C}_{(n)}\ss\mathcal{C}$
exhibits $\mathcal{C}$ as a recollement of $\mathcal{C}_{(n)}$ and
$\mathcal{C}_{(n)}^{\perp}$. We also set 
\[
\mathcal{C}_{(\infty)}\coloneqq\bigcap_{n\in\mathbb{N}}{\cal C}_{(n)}.
\]
\end{defn}

It turns out that under mild conditions, $\mathcal{C}_{(\infty)}\ss\mathcal{C}$
is itself a recollement.
\begin{lem}
\label{lem:Recollement_Intersection}Let ${\cal C}\in\acat_{\st}$,
which admits sequential limits and colimits. For a recollement chain
$\dots\subseteq{\cal C}_{(2)}\subseteq{\cal C}_{(1)}\subseteq{\cal C}_{(0)}\ss{\cal C}$,
the inclusion $\mathcal{C}_{(\infty)}\ss\mathcal{C}$ exhibits $\mathcal{C}$
as a recollement of $\mathcal{C}_{(\infty)}$ and $\mathcal{C}_{(\infty)}^{\perp}$.
\end{lem}

\begin{proof}
It suffices to show that the inclusion $\mathcal{C}_{(\infty)}\into\mathcal{C}$
admits a left and a right adjoint. By symmetry, it suffices to consider
only the left adjoint. By \cite[Proposition 5.2.7.8]{htt}, for every $X\in\mathcal{C}$
we need to construct an object $L_{\infty}X\in\mathcal{C}_{(\infty)}$
and a morphism $X\oto{\eta}L_{\infty}X$, such that for all $Y\in\mathcal{C}_{(\infty)}$,
the map
\[
\map(L_{\infty}X,Y)\oto{(-)\circ\eta}\map(X,Y)
\]
is an isomorphism. Let $L_{n}\colon\mathcal{C}\to\mathcal{C}_{(n)}$
be the left adjoint of the inclusion $\mathcal{C}_{(n)}\into\mathcal{C}$
and denote by $\Id\oto{\eta_{n}}L_{n}$ the corresponding unit (i.e.
localization) map, where we suppress the embedding functor $\mathcal{C}_{(n)}\ss\mathcal{C}$.
Since the $\mathcal{C}_{(n)}$-s are nested, we have canonical isomorphisms
$L_{n}L_{n-1}\iso L_{n}$, and we abuse notation by denoting the composition
$L_{n-1}\oto{\eta_{n}}L_{n}L_{n-1}\iso L_{n}$ also by $\eta_{n}$.
We now define 
\[
L_{\infty}X\coloneqq\colim(X\oto{\eta_{1}}L_{1}X\oto{\eta_{2}}L_{2}X\oto{\eta_{3}}\dots)
\]
and take $X\oto{\eta}L_{\infty}X$ to be the cone map from the first
object to the colimit in the diagram defining $L_{\infty}X$. For every
$Y\in\mathcal{C}_{(\infty)}\ss\mathcal{C}_{(n)}$, the map 
\[
\map(L_{n}X,Y)\oto{(-)\circ\eta_{n}}\map(X,Y)
\]
is an isomorphism for each $n\in\bb N$. Thus, by taking the limit
over $n$, we get an isomorphism
\[
\lim_{\bb N}\map(L_{n}X,Y)\iso\lim_{\bb N}\map(X,Y)\simeq\map(X,Y).
\]
Precomposing with the isomorphism
\[
\map(L_{\infty}X,Y)=\map(\colim_{\bb N}L_{n}X,Y)\iso\lim_{\bb N}\map(L_{n}X,Y),
\]
we get an isomorphism
\[
\map(L_{\infty}X,Y)\iso\map(X,Y).
\]
Unwinding the definitions, this isomorphism is given by precomposition
with $\eta$. 
\end{proof}
To identify the right orthogonal complement of $\mathcal{C}_{(\infty)}$
in $\mathcal{C}$, we need the following general categorical fact:
\begin{lem}
\label{lem:Kernel_Image_Orthogonal}Let $F\colon\mathcal{C}\to\mathcal{D}$
be a functor in $\acat_{\st}$, and denote by $\ker(F)\ss\mathcal{C}$
the full subcategory spanned by the objects $X$, for which $F(X)=0$.
If $F$ admits a fully faithful right adjoint $G\colon\mathcal{D}\into\mathcal{C}$,
then $\mathrm{Im}(G)=\ker(F)^{\perp}$.
\end{lem}

\begin{proof}
In one direction, for $X\in\mathrm{Im}(G)$, we have $X=G(Y)$ for
some $Y\in\mathcal{D}$. Hence, for every $Z\in\ker(F)$ we have
\[
\map(Z,X)\simeq\map(Z,G(Y))\simeq\map(F(Z),Y)\simeq\map(0,Y)\simeq\pt.
\]
Thus, $\mathrm{Im}(G)\ss\ker(F)^{\perp}$. Conversely, let $\Id\oto uGF$
and $FG\oto c\Id$ be the unit and counit of the adjunction respectively.
Since $G$ is fully faithful, $c$ is an isomorphism. By the zig-zag
identities and 2-out-of-3, the map $F(u)$ is also an isomorphism.
Now, for every $X\in\mathcal{C}$ consider the fiber sequence
\[
X_{0}\to X\oto uGF(X).
\]
On the one hand, since $F(u)$ is an isomorphism, $F(X_{0})=0$ and hence $X_{0}\in\ker(F)$.
On the other hand, if $X\in\ker(F)^{\perp}$, then since $GF(X)\in\mathrm{Im}(G)\ss\ker(F)^{\perp}$,
we also have $X_{0}\in\ker(F)^{\perp}$ and thus $X_{0}=0$. This
implies that $u$ is an isomorphism and so $X\simeq GF(X)\in\mathrm{Im}(G)$.
\end{proof}
Given a recollement chain as in \defref{Recollement_Chain}, for each
$n\in\bb N$, we have a fully faithful embedding $\mathcal{C}_{(n)}^{\perp}\into\mathcal{C}$
with left adjoint $P_{n}\colon\mathcal{C}\to\mathcal{C}_{(n)}^{\perp}$.
We abuse notation by suppressing the inclusion $\mathcal{C}_{n}^{\perp}\ss\mathcal{C}$
and the canonical isomorphisms $P_{n}P_{n+1}\simeq P_{n}$. We thus obtain a tower 
\[
\dots\oto{P_{n+1}}\mathcal{C}_{(n)}^{\perp}\oto{P_{n}}\dots\oto{P_{2}}{\cal C}_{(2)}^{\perp}\oto{P_{1}}{\cal C}_{(1)}^{\perp}\oto{P_{0}}{\cal C}_{(0)}^{\perp}
\]
of $\infty$-categories under $\mathcal{C}$, which induces a functor
\[
P_{\infty}\colon\mathcal{C}\to\lim{\cal C}_{(n)}^{\perp}.
\]

\begin{prop}
\label{prop:Recallement_Chain}Let ${\cal C}\in\acat_{\st}$ which
admits sequential limits and colimits. Given a recollement chain $\dots\subseteq{\cal C}_{(2)}\subseteq{\cal C}_{(1)}\subseteq{\cal C}_{(0)}\ss{\cal C}$,
the functor ${\displaystyle P_{\infty}\colon{\cal C}\to\lim{\cal C}_{(n)}^{\bot}}$
admits a fully faithful right adjoint, whose essential image is $\mathcal{C}_{(\infty)}^{\perp}$.
Thus, ${\cal C}$ is a recollement of ${\displaystyle \mathcal{C}_{(\infty)}}$
and $\lim{\cal C}_{(n)}^{\bot}$. 
\end{prop}

\begin{proof}
By \lemref{Recollement_Intersection}, the inclusion $\mathcal{C}_{(\infty)}\into\mathcal{C}$
exhibits $\mathcal{C}$ as a recollement of $\mathcal{C}_{(\infty)}$
and its right orthogonal complement and hence it suffices to identify
$\mathcal{C}_{(\infty)}^{\perp}$. The objects of $\mathcal{C}_{(\infty)}$
are precisely the $X\in\mathcal{C}$ for which $P_{\infty}(X)=0$.
Thus, by \lemref{Kernel_Image_Orthogonal}, it suffices to show that
$P_{\infty}$ admits a fully faithful right adjoint. Since $P_{n}$
is a left adjoint for all $n$, by \cite[Theorem B]{AD}, the functor $P_{\infty}$
is a left adjoint and we denote its right adjoint by $G_{\infty}\colon\lim{\cal C}_{(n)}^{\bot}\to\mathcal{C}$.
We show that $G_{\infty}$ is fully faithful using the explicit description
of the adjunction $P_{\infty}\dashv G_{\infty}$ given in \cite{AD}.
An object of $\lim{\cal C}_{(n)}^{\bot}$ consists of a sequence of
objects $X_{n}\in\mathcal{C}_{(n)}^{\perp}$ together with structure
isomorphisms $P_{n}X_{n+1}\iso X_{n}$. We shall write $\{X_{n}\}\in\lim{\cal C}_{(n)}^{\bot}$
suppressing the structure isomorphisms. Composing the structure isomorphisms
of $\{X_{n}\}$ with the corresponding unit (i.e. localization) maps
$\Id\oto{u_{n}}P_{n}$, we get maps as follows:
\[
f_{n}\colon X_{n+1}\oto{u_{n}}P_{n}X_{n+1}\iso X_{n}.
\]
By \cite[Theorem B]{AD}, the functor $G_{\infty}$ can be described explicitly
on objects by the following formula:
\[
G_{\infty}(\{X_{n}\})\simeq\lim(\dots\oto{f_{n}}X_{n}\oto{f_{n-1}}\dots\oto{f_{2}}X_{2}\oto{f_{1}}X_{1}\oto{f_{0}}X_{0}).
\]

To prove that $G_{\infty}$ is fully faithful, it suffices to show
that the counit $P_{\infty}G_{\infty}\oto c\Id$ is an isomorphism.
Since the collection of projection functors $\pi_{k}\colon\lim{\cal C}_{(n)}^{\bot}\to\mathcal{C}_{(k)}^{\bot}$
for all $k\in\bb N$ is jointly conservative, it suffices to show
that
\[
P_{k}(\lim X_{n})\oto{\pi_{k}(c)}X_{k}
\]
is an isomorphism for all $k\in\bb N$ and $\{X_{n}\}\in\lim{\cal C}_{(n)}^{\bot}$.
By \cite[Theorem 5.5]{AD}, we can describe $\pi_{k}(c)$ as the composition 
\[
P_{k}(\lim X_{n})\to P_{k}(X_{k})\iso X_{k},
\]
where the first map is induced by the canonical projection $\lim X_{n}\to X_{k}$.
By cofinality, we can assume that the limit is taken over $n\ge k$.
By definition, for each $n\in\bb N$, the fiber of $X_{n+1}\oto{f_{n}}X_{n}$
lies in $\mathcal{C}_{(n)}\ss\mathcal{C}_{(k)}$. Since $\mathcal{C}_{(k)}$
is closed under sequential limits, it follows that the fiber of $\lim X_{n}\to X_{k}$
lies in $\mathcal{C}_{(k)}$ and hence becomes an isomorphism after
applying $P_{k}$. This concludes the proof that $G_{\infty}$ is
fully faithful and hence the proof of the claim.
\end{proof}
\begin{cor}
\label{cor:Recollement_Chain_Decomposition}Let ${\cal C}\in\acat_{\st}$
which admits sequential limits and colimits with a recollement chain
$\dots\subseteq{\cal C}_{(2)}\subseteq{\cal C}_{(1)}\subseteq{\cal C}_{(0)}\ss{\cal C}$.
If $\mathcal{C}_{(\infty)}=0$ then ${\cal C}\simeq\lim{\cal C}_{(n)}^{\bot}$.
\end{cor}

\begin{proof}
By \propref{Recallement_Chain}, ${\cal C}$ is a recollement of $\mathcal{C}_{(\infty)}=0$
and of $\lim{\cal C}_{(n)}^{\bot}$, so that ${\cal C}\simeq\lim{\cal C}_{(n)}^{\bot}$. 
\end{proof}

\subsubsection{Divisible and complete recollement}

One way to get a recollement is by taking the divisible and complete
objects with respect to a natural endomorphism of the identity functor.
That is, given a stable $\infty$-category $\mathcal{C}$ and $\Id_{\mathcal{C}}\oto \alpha\Id_{\mathcal{C}}$,
we have the full subcategories $\mathcal{C}[\alpha^{-1}]$ and $\widehat{\mathcal{C}}_{\alpha}=\mathcal{C}[\alpha^{-1}]^{\perp}$
of $\mathcal{C}$ (\defref{Div_Comp}). Assuming $\mathcal{C}$ admits
sequential limits and colimits, the inclusion $\mathcal{C}[\alpha^{-1}]\into\mathcal{C}$
admits both a left adjoint $L$ and a right adjoint $R$, given respectively
by ``inverting $\alpha$'' on $X$
\[
LX=\colim(X\oto{\alpha}X\oto{\alpha}X\oto{\alpha}\dots)
\]
and by taking the ``$\alpha$-divisible part'' of $X$

\[
RX=\lim(\dots\oto{\alpha}X\oto{\alpha}X\oto{\alpha}X).
\]

\begin{rem}
We warn the reader that although the above statements are well known
and fairly intuitive, they are not as tautological as one might think.
In particular, they might fail if $\mathcal{C}$ is not assumed to
be stable (or at least additive). We refer the reader to \cite[Appendix
C]{BNT}, for a comprehensive treatment of a closely related situation. 
\end{rem}

Note that an object $Y\in\mathcal{C}$ is $\alpha$-complete if and
only if $RY=0$ if and only if $Y\simeq\lim Y/\alpha^{r}$. In fact,
the $\alpha$\textbf{-completion} functor
\[
Y\mapsto\widehat{Y}_{\alpha}\coloneqq\lim Y/\alpha^{r},
\]
is the left adjoint to the inclusion $\widehat{\mathcal{C}}_{\alpha}\into\mathcal{C}$. 
\begin{prop}
\label{prop:Div_Comp_Recollemenet}Let $\mathcal{C}\in\acat_{\st}$
which admits sequential limits and colimit and let $\Id_{\mathcal{C}}\oto \alpha\Id_{\mathcal{C}}$.
Then $\mathcal{C}$ is a recollement of $\mathcal{C}[\alpha^{-1}]$
and $\widehat{\mathcal{C}}_{\alpha}$.
\end{prop}

\begin{proof}
It follows from the discussion above that $\mathcal{C}[\alpha^{-1}]\into\mathcal{C}$
admits both adjoints.
\end{proof}
Our next goal is to give a characterization of when the said recollement
is \emph{split} in terms of the natural endomorphism $\alpha$. 
\begin{defn}
We say that a natural endomorphism $\beta\colon\Id_{\mathcal{C}}\to\Id_{\mathcal{C}}$
is a \textbf{semi-inverse} of $\alpha$, if for every $\alpha$-divisible
$X$, the map $\beta_{X}$ is an inverse of $\alpha_{X}$. 
\end{defn}

The usefulness of the notion of semi-inverse is in that it allows
us to characterize completeness in terms of divisibility:
\begin{prop}
\label{prop:Semi_Inevrse_Divisible} Let $\mathcal{C}\in\acat_{\st}$
which admits sequential limits and colimits. For every $\alpha,\beta\colon\Id_{\mathcal{C}}\to\Id_{\mathcal{C}}$, if an object $Y\in\mathcal{C}$ is $\alpha$-complete, then it is
$(1-\alpha\beta)$-divisible. If $\beta$ is a semi-inverse of $\alpha$,
then the converse holds as well. 
\end{prop}

\begin{proof}
Note that all natural endomorphisms of $\Id_{\mathcal{C}}$ commute
by the interchange law, so in particular $\alpha\beta=\beta\alpha$.
For an $\alpha$-complete object $Y\in\mathcal{C}$ we have $Y=\lim Y/\alpha^{r}$.
For every $r\in\bb N$, the map $\alpha^{2r}$ is zero on $Y/\alpha^{r}$
and hence $1-\alpha\beta$ is invertible on $Y/\alpha^{r}$. By passing
to the limit, $1-\alpha\beta$ is invertible on $Y$. Conversely,
assume that $(1-\alpha\beta)$ acts invertibly on $Y$. If $\beta$ is
a semi-inverse of $\alpha$, then for every $\alpha$-divisible $X$,
the map $(1-\alpha\beta)$ acts as zero on $X$. Thus the pointed
space $\map(X,Y)$ must be contractible as $(1-\alpha\beta)$ acts
both invertibly and as zero on it. This implies that $Y$ is $\alpha$-complete.
\end{proof}
The above lemma leads us to the following characterization of split recollement:
\begin{prop}
\label{prop:Recollement}Let $\mathcal{C}$ be a stable $\infty$-category
which admits sequential limits and colimits. The recollement associated
with a natural endomorphism $\Id_{\mathcal{C}}\oto{\alpha} \Id_{\mathcal{C}}$
is split if and only if $\alpha$ admits a semi-inverse. In which
case,
\[
\mathcal{C}\simeq\mathcal{C}[\alpha^{-1}]\times\widehat{\mathcal{C}}_{\alpha}.
\]
\end{prop}

\begin{proof}
Let $\beta$ be a semi-inverse of $\alpha$. To show that $LY=0$
for all $Y\in\widehat{\mathcal{C}}_{\alpha}$ it suffices to show
that $\map(Y,X)$ is contractible for all $X\in\mathcal{C}[\alpha^{-1}]$.
By definition, $1-\alpha\beta$ is zero on $X$, so it suffices to
observe that $1-\alpha\beta$ is invertible on $Y$ by \propref{Semi_Inevrse_Divisible}.
Conversely, if $\mathcal{C}\simeq\mathcal{C}[\alpha^{-1}]\times\widehat{\mathcal{C}}_{\alpha}$,
we have for every $X\in\mathcal{C}$ a natural decomposition $X\simeq X[\alpha^{-1}]\oplus\widehat{X}_{\alpha}$
with $X[\alpha^{-1}]\in\mathcal{C}[\alpha^{-1}]$ and $\widehat{X}_{\alpha}\in\widehat{\mathcal{C}}_{\alpha}$.
In this case, the map $\beta=(\alpha|_{X[\alpha^{-1}]})^{-1}\oplus0_{\widehat{X}_{\alpha}}$
is a semi-inverse of $\alpha$. 
\end{proof}
\begin{rem}
In \exaref{Chromatic_Recollement}, the recollement $\Sp_{\bb Q}\ss\Sp_{(p)}$
corresponds to the endomorphism $\Id_{\Sp_{(p)}}\oto p\Id_{\Sp_{(p)}}$.
However, not every recollement arises in such a way. For example,
for $n\ge1$ the recollement $L_{n}\Sp\ss\Sp_{(p)}$ is not induced
by any endomorphism $\alpha$ of the identity functor. We do note
however, that every \emph{split} recollement $\mathcal{C}_{\circ}\ss\mathcal{C}$
must arise from an endomorphism $\alpha$ of $\Id_{\mathcal{C}}$,
because we can take $\alpha$ to be the idempotent $\varepsilon\colon\Id_{\mathcal{C}}\to\Id_{\mathcal{C}}$
projecting onto $\mathcal{C}_{\circ}$. In this case, $\varepsilon$
itself is a semi-inverse of $\varepsilon$. 
\end{rem}

\subsection{Height Decomposition}

Let $\mathcal{C}$ be now a stable $m$\emph{-semiadditive} $\infty$-category.
We shall use the general machinery of (split) recollement to show
that $\mathcal{C}$ splits into a product of $\infty$-categories
according to height. By definition, an object $X\in\mathcal{C}$ is
of height $\le n$ if it is $\pn n$-divisible. Similarly, $X$ is
of height $>n$ if it is $\pn n$-complete, which by \propref{Height_Sense},
is if and only if it is complete with respect to all of $p=\pn 0,\pn 1,\dots,\pn n$.
Accordingly, 
\begin{prop}
\label{prop:Sad_Height_Recollement} Let $\mathcal{C}\in\acat_{\st}^{\sad m}$
and let $0\le n\le m$. If $\mathcal{C}$ admits sequential limits and
colimits, then $\mathcal{C}$ is a recollement of $\mathcal{C}_{\le n}$
and $\mathcal{C}_{>n}$. 
\end{prop}

\begin{proof}
The full subcategory $\mathcal{C}_{\le n}\ss\mathcal{C}$ consists
of the $\pn n$-divisible objects and $\mathcal{C}_{>n}\ss\mathcal{C}$
is the full subcategory of $\pn n$-complete objects. Thus, the result
follows from \propref{Div_Comp_Recollemenet}.
\end{proof}
Our next goal is to show that under suitable assumptions, this recollement
is in fact \emph{split}. For this we need the following:
\begin{prop}
\label{prop:Box_Semi_Inverse}Let $\mathcal{C}\in\acat_{\st}^{\sad m}$
and assume it admits sequential limits and colimits. For all $n=0,\dots,m-1$,
the map $\pn{n+1}$ is a semi-inverse of $\pn n$. In particular,
for $X\in\mathcal{C}$, we have $\htt(X)>n$ if and only if $X$ is
$(1-\pn n\pn{n+1})$-divisible.
\end{prop}

\begin{proof}
If 
\[
\pn n=|B^{n}C_{p}|=|\Omega B^{n+1}C_{p}|
\]
is invertible on $X\in\mathcal{C}$, then by \propref{Amenable_Acyclic_Object},
$\pn{n+1}=|B^{n+1}C_{p}|$ is the inverse of $\pn n$ on $X$. Thus,
$\pn{n+1}$ is a semi-inverse of $\pn n$. Therefore, the claim follows
from \propref{Semi_Inevrse_Divisible}.
\end{proof}
Using \propref{Box_Semi_Inverse} we can improve on \propref{Height_Functoriality}
in the stable case as follows:
\begin{cor}
\label{cor:Stable_Height_Functoriality}Let $F\colon\mathcal{C}\to\mathcal{D}$
be a functor in $\acat^{\sad m}$ and assume $\mathcal{C}$ and $\mathcal{D}$
are stable. For every $X\in\mathcal{C}$ and $0\le n\le m$, if $X$
is of height $\le n$ or $>n-1$, then so is $F(X)$. The converse
holds if $F$ is conservative.
\end{cor}

\begin{proof}
The facts about height $\le n$ follow from \propref{Height_Functoriality}.
The facts about height $>n-1$ follow similarly using \propref{Box_Semi_Inverse}.
Namely, that $X$ is of height $>n-1$, if and only if $(1-\pn{n-1}\pn n)$
acts invertibly on it.
\end{proof}
\begin{rem}
\corref{Stable_Height_Functoriality} does not cover the case
$\htt(X)>m$. This indeed can not be guaranteed even in the stable
case as witnessed by \exaref{Height_Decrease}. 
\end{rem}

Similarly, we can improve on \corref{Height_Monoidal} as follows:
\begin{cor}
\label{cor:Stable_Height_Monoiidal}Let $\mathcal{C}\in\acat_{\st}$
be $p$-typically $m$-semiadditively monoidal $\infty$-category.
For every $0\le n\le m$, we have $\Ht(\mathcal{C})\le n$ or $\Ht(\mathcal{C})>n-1$
if and only if $\htt_{\mathcal{C}}(\one)\le n$ or $\htt_{\mathcal{C}}(\one)>n-1$ respectively. 
\end{cor}

\begin{proof}
Given $X\in\mathcal{C}$, the functor $X\otimes(-)\colon\mathcal{C}\to\mathcal{C}$
is $p$-typically $m$-semiadditive. Thus, the claim follows from
\corref{Stable_Height_Functoriality}.
\end{proof}
\begin{rem}
Again, in \corref{Stable_Height_Monoiidal} it is \emph{not} true
that if $\htt(\one)>m$ then $\Ht(\mathcal{C})>m$. For example, $\Mod_{\widehat{\bb S}_{p}}(\Sp)$
is a presentably symmetric monoidal $0$-semiadditive $\infty$-category,
whose unit $\widehat{\bb S}_{p}$ has height $>0$, although the $\infty$-category
itself does not. 
\end{rem}

The following is our main structure theorem for stable higher semiadditive
$\infty$-categories:
\begin{thm}
[Height Decomposition]\label{thm:Height_Decomposition}Let $\mathcal{C}\in\acat_{\st}^{\sad m}$
for some $0\le m\le\infty$.
\begin{enumerate}
\item For $m<\infty$, and $\mathcal{C}$ which is idempotent complete,
the inclusions $\mathcal{C}_{0},\dots,\mathcal{C}_{m-1},\mathcal{C}_{>m-1}\ss\mathcal{C}$
determine an equivalence of $\infty$-categories
\[
\mathcal{C}\simeq\mathcal{C}_{0}\times\dots\times\mathcal{C}_{m-1}\times\mathcal{C}_{>m-1}.
\]
Moreover, $\mathcal{C}_{0}$ is $p$-typically $\infty$-semiadditive
and $\mathcal{C}_{1},\dots,\mathcal{C}_{m-1}$ are $\infty$-semiadditive. 
\item For $m=\infty$, and $\mathcal{C}$ which admits sequential limits
and colimits, $\mathcal{C}$ is a recollement of $\mathcal{C}_{\infty}$
and $\prod_{n\in\bb N}\mathcal{C}_{n}$. In particular, if $\mathcal{C}_{\infty}=0$,
then $\mathcal{C}\simeq\prod_{n\in\bb N}\mathcal{C}_{n}.$
\end{enumerate}
If in addition $\mathcal{C}$ is $m$-semiadditively $\mathcal{O}$-monoidal
for some $\infty$-operad $\mathcal{O},$ then $\mathcal{C}_{>m-1}$
and $\mathcal{C}_{n}$ for all $n=0,\dots,m-1$, are compatible with
the $\mathcal{O}$-monoidal structure and the equivalences in both
(1) and (2) promote naturally to an equivalence of $\mathcal{O}$-monoidal
$\infty$-categories.
\end{thm}

\begin{proof}
(1) We first prove the claim under the additional assumption that
$\mathcal{C}$ admits all sequential limits and colimits (and hence
in particular idempotent complete). First, by \propref{Box_Semi_Inverse},
the map $\pn m$ is a semi-inverse of $\pn{m-1}$. Hence, by \propref{Recollement},
we obtain a direct product decomposition $\mathcal{C}\simeq\mathcal{C}_{\le m-1}\times\mathcal{C}_{>m-1}$.
The category $\mathcal{C}_{\le m-1}$ is itself $m$-semiadditive
(\propref{Height_Pieces_Sad}) and admits all sequential limits and
colimits. Thus, we can continue decomposing $\mathcal{C}_{\le m-1}$
inductively and get $\mathcal{C}_{\le n}\simeq\mathcal{C}_{\le n-1}\times\mathcal{C}_{n}$
for all $n=0,\dots,m-1$. Finally, by \propref{Height_Bootstrap},
each $\mathcal{C}_{\le n}$ is in fact $p$-typically $\infty$-semiadditive.
By \propref{Height_Pieces_Sad}, $\mathcal{C}_{n}$ is also $p$-typically
$\infty$-semiadditive and since it is also $p$-complete for all
$n\ge1$, it is in particular $p$-local, and hence $\infty$-semiadditive
by \propref{p_Local_n_Semiadd}.

For a general $\mathcal{C}$ as in the claim, we use a semiadditive
version of the Yoneda embedding to reduce to the presentable case.
Namely, we shall show in \propref{Tsadi_Yoneda}, that there exists
a presentable stable $m$-semiadditive $\infty$-category $\widehat{\mathcal{C}}$
and an $m$-semiadditive fully faithful embedding $\mathcal{C}\into\widehat{\mathcal{C}}$.
By \corref{Stable_Height_Functoriality}, for each $n=0,\dots,m-1$,
we have fully faithful embedding $\mathcal{C}_{n}\into\widehat{\mathcal{C}}_{n}$
and $\mathcal{C}_{>m-1}\into\widehat{\mathcal{C}}_{>m-1}$, hence
also 
\[
\mathcal{C}_{0}\times\dots\times\mathcal{C}_{m-1}\times\mathcal{C}_{>m-1}\into\widehat{\mathcal{C}}_{0}\times\dots\times\widehat{\mathcal{C}}_{m-1}\times\widehat{\mathcal{C}}_{>m-1}\simeq\widehat{\mathcal{C}}.
\]
By the left cancellation property of fully faithful embeddings, we
get a fully faithful embedding 
\[
\mathcal{C}_{0}\times\dots\times\mathcal{C}_{m-1}\times\mathcal{C}_{>m-1}\into\mathcal{C}.
\]

For each object $X\in\mathcal{C}$, the height $n=0,\dots,m-1$ and
$>m-1$ components of $X$ in $\widehat{\mathcal{C}}$, are retracts
of $X$ in $\widehat{\mathcal{C}}$. Thus, if $\mathcal{C}$ is idempotent
complete, these components belong to $\mathcal{C}$. It follows that
the above fully faithful embedding is also essentially surjective.

(2) For every $n<\infty$, we have by (1), that ${\cal C}\simeq{\cal C}_{\le n}\times{\cal C}_{>n}$.
Hence, we can switch the roles of $\mathcal{C}_{\le n}$ and $\mathcal{C}_{>n}$,
and consider the embedding $\mathcal{C}_{>n}\ss\mathcal{C}$ as exhibiting
$\mathcal{C}$ as a recollement of $\mathcal{C}_{>n}$ and $(\mathcal{C}_{>n})^{\bot}=\mathcal{C}_{\le n}$.
We thus obtain a recollement chain 
\[
\dots\ss\mathcal{C}_{>2}\subseteq{\cal C}_{>1}\subseteq{\cal C}_{>0}\ss{\cal C}.
\]
By definition, $\mathcal{C}_{\infty}=\bigcap\limits _{n\in\bb N}\mathcal{C}_{>n}$,
and so by \propref{Recallement_Chain}, $\mathcal{C}$ is a recollement
of $\mathcal{C}_{\infty}$ and 
\[
\lim_{\mathbb{n\in N}}(\mathcal{C}_{\le n})\simeq\lim_{n\in\bb N}\left(\prod_{0\le k\le n}\mathcal{C}_{k}\right)\simeq\prod_{n\in\bb N}\mathcal{C}_{n}.
\]
Finally, assume that $\mathcal{C}$ is $m$-semiadditively $\mathcal{O}$-monoidal.
The full subcategories $\mathcal{C}_{\le n}$ and $\mathcal{C}_{>n}$ for
all $n=0,\dots,m-1$ consist of objects which are $\pn n$-divisible
and $(1-\pn n\pn{n+1})$-divisible respectively. It follows that $\mathcal{C}_{\le n}$
and $\mathcal{C}_{>n}$ are compatible with the $\mathcal{O}$-monoidal
structure and hence so is $\mathcal{C}_{n}$ for all $n=0,\dots,m-1$.
Thus, by \cite[Proposition 2.2.1.9]{ha}, the $\mathcal{C}_{n}$-s and $\mathcal{C}_{>m-1}$
inherit an $\mathcal{O}$-monoidal structure such that the projections
$\mathcal{C}\to\mathcal{C}_{n}$ and $\mathcal{C}\to\mathcal{C}_{>m-1}$
are $\mathcal{O}$-monoidal.
\end{proof}
We conclude with some remarks regarding the sharpness of \thmref{Height_Decomposition}.
In the case $m<\infty$, the fact that $\mathcal{C}$ is $m$-semiadditive
also implies that $\mathcal{C}$ is a recollement of $\mathcal{C}_{\le m}$
and $\mathcal{C}_{>m}$, but there is no guaranty that the ``gluing
data'' is trivial. That is, that $\mathcal{C}$ decomposes as a direct
product\emph{ }of $\mathcal{C}_{\le m}$ and $\mathcal{C}_{>m}$.
Indeed, consider the $0$-semiadditive $\infty$-category $\mathcal{C}=\Sp_{(p)}.$
The case $m=0$ corresponds to the recollement $\Sp_{\bb Q}\ss\Sp_{(p)}$
of \exaref{Chromatic_Recollement}. In this case $\mathcal{C}_{>0}=\widehat{\Sp}_{p}$
and the gluing data is \emph{not} trivial, as the rationalization
of the $p$-completion does not vanish in general. Having said that,
for $m\ge1$ we do not know whether there even exists a stable $p$-local
presentable $\infty$-category that is $m$-semiadditive, but not
$(m+1)$-semiadditive \cite[Conjecture 1.1.5]{Ambi2018}. 

In the case $m=\infty$, we do
not know whether there exists a stable $\infty$-semiadditive $\infty$-category
$\mathcal{C}$ for which $\mathcal{C}_{\infty}\neq0$. We hence propose the following:

\begin{conjecture}
[Height Finiteness]\label{conj:Finitness}
For every $\mathcal{C}\in\Pr_{\st}^{\sad \infty}$, the full subcategory 
\[
\mathcal{C}_{\infty}\coloneqq\bigcap_{n\ge0}\mathcal{C}_{>n}\ss\mathcal{C},
\]
of objects of height $\infty$, is trivial.
\end{conjecture}

\subsection{Semisimplicity}

Classical representation theory tells us that in characteristic $0$,
representations of a finite group $G$ are \emph{semisimple}. The
$\infty$-category of $\mathcal{C}$-representations of $G$ in any
$\infty$-category $\mathcal{C}$ is equivalent to the $\infty$-category
$\mathcal{C}^{BG}$ of $\mathcal{C}$-valued local systems on $BG$.
From the point of view of higher semiadditivity, characteristic $0$
corresponds to semiadditive height $0$, and so it is natural to consider
the analogous situation for higher heights. We shall show that given
a stable $\infty$-semiadditive $\infty$-category $\mathcal{C}$
of height $n$, certain analogous semisimplicity phenomena hold for $\mathcal{C}^{A}$
for every $\pi$-finite $n$\emph{-connected} space $A$ (e.g. $A=B^{n+1}G$). For the case $\mathcal{C} = \Mod_{K(n)}(\Sp)$ these ideas were discussed in \cite{LurieRep}.
\subsubsection{Splitting local systems }

The main result of this subsection is the following relation between acyclic maps and
split recollement in the stable setting:
\begin{prop}
\label{prop:Acyclic_Semisimple}Let $\mathcal{C}\in\acat_{\st}$ and
let $A\oto qB$ be a $\mathcal{C}$-acyclic and weakly $\mathcal{C}$-ambidextrous
map. The functor $q^{*}\colon\mathcal{C}^{B}\to\mathcal{C}^{A}$ is
fully faithful and exhibits $\mathcal{C}^{A}$ as a recollement of
$\mathcal{C}^{B}$ and $(\mathcal{C}^{B})^{\bot}\ss\mathcal{C}^{A}$.
The recollement is split if and only if $q$ is $\mathcal{C}$-ambidextrous,
in which case there is a canonical equivalence:
\[
\mathcal{C}^{A}\simeq\mathcal{C}^{B}\times(\mathcal{C}^{B})^{\bot}.
\]
\end{prop}

\begin{proof}
By definition of acyclicity, $q^{*}$ is fully faithful. Hence, the
recollement is split if and only the left and right adjoints $q_{!}$
and $q_{*}$ respectively of $q^{*}$ are isomorphic (\propref{Split_Recollement}).
If $q$ is $\mathcal{C}$-ambidextrous then $q_{!}$ and $q_{*}$
are isomorphic and the recollement is split. Conversely, if we have $q_{!}\simeq q_{*}$,
then $q_{*}$ preserves all $q$-colimits and hence $q$ is $\mathcal{C}$-ambidextrous (\propref{Ambi_Criterion}).
\end{proof}
As a special case we obtain:
\begin{thm}
\label{thm:Height_Semisimple}Let $\mathcal{C}\in\acat_{\st}^{\sad{\infty}}$
be $p$-local such that $\Ht(\mathcal{C})=n$. For every map of spaces $A\oto qB$
with $n$-connected $\pi$-finite fibers, we have a canonical equivalence
\[
\mathcal{C}^{A}\simeq\mathcal{C}^{B}\times(\mathcal{C}^{B})^{\bot}.
\]
\end{thm}

\begin{proof}
Since $\mathcal{C}$ is of height $n$, the map $q$ is $\mathcal{C}$-acyclic
by \thmref{Height_Everything_p_Local}(2) and \propref{Acyclicity_Fiberwise}.
Hence, the claim follows from \propref{Acyclic_Semisimple}.
\end{proof}
In the case $B= \pt$, we can interpret \thmref{Height_Semisimple} from the perspective of ``higher representation theory'' (see  \cite{LurieRep}). For every space $A$, we call an  object $X\in\mathcal{C}^{A}$  \textbf{unipotent} if it belongs
to the full subcategory $\mathcal{C}_{\mathrm{unip}}^{A}\ss\mathcal{C}^{A}$
generated by colimits from the trivial representations (i.e. constant
local systems). It can be shown that every object $X\in\mathcal{C}^{A}$ fits into an essentially unique cofiber sequence
\[
X_{\mathrm{unip}}\to X\to X^{0} \qquad(*)
\]
with $X_{\mathrm{unip}}\in\mathcal{C}_{\mathrm{unip}}^{A}$ and $X^{0}\in (\mathcal{C}_{\mathrm{unip}}^{A})^{\perp}$. 
If $A$
is $n$-connected, then  \thmref{Height_Semisimple} implies that $X_{\mathrm{unip}}$ is constant and $(*)$
canonically splits. 

\subsubsection{Transfer idempotents}

Given a stable $\infty$-category $\mathcal{C}$ and an equivalence
$\mathcal{C}\simeq\mathcal{C}^{\prime}\times\mathcal{C}^{\prime\prime}$,
every object $X\in\mathcal{C}$ has an essentially unique decomposition
$X\simeq X^{\prime}\oplus X^{\prime\prime}$, such that $X^{\prime}\in\mathcal{C}^{\prime}$
and $X^{\prime\prime}\in\mathcal{C}^{\prime\prime}$. This allows
us to define a natural endomorphism $\varepsilon\colon\Id_{\mathcal{C}}\to\Id_{\mathcal{C}}$
by the formula 
\[
X^{\prime}\oplus X^{\prime\prime}\oto{\Id_{X^{\prime}}\oplus0_{X^{\prime\prime}}}X^{\prime}\oplus X^{\prime\prime}.
\]
The natural endomorphism $\varepsilon$ is idempotent and realizes
internally to $\mathcal{C}$ the projection onto the essential image
of $\mathcal{C}^{\prime}$ in $\mathcal{C}$. Our next goal is to
provide an explicit description of the idempotent $\varepsilon$ for
the split recollement $q^{*}\colon\mathcal{C}^{B}\into\mathcal{C}^{A}$
of \propref{Acyclic_Semisimple}. To help guide the intuition, we
begin with a closely related elementary example:
\begin{example}
\label{exa:Rational_Semisimple}Let $\mathcal{C}=\mathrm{Vec}_{\bb Q}$
and let $G$ be a finite group with a normal subgroup $N\triangleleft G$.
The map $q\colon BG\to B(G/N)$ induces a fully faithful embedding
\[
q^{*}\colon\mathrm{Vec}_{\bb Q}^{B(G/N)}\to\mathrm{Vec}_{\bb Q}^{BG}
\]
which is the ``inflation'' functor that takes a vector space $V$
with a $G/N$-action to $V$ itself with the $G$-action induced via
$G\to G/N$. The essential image of $q^{*}$ consists of $G$-representations
on which the subgroup $N\triangleleft G$ acts trivially. The adjoints
$q_{!}$ and $q_{*}$ can be identified with the vector spaces of
$N$-coinvairant and $N$-invariants respectively, equipped with the
residual $G/N$-action. Thus, the full subcategory $(\mathrm{Vec}_{\bb Q}^{B(G/N)})^{\bot}\ss\mathrm{Vec}_{\bb Q}^{BG}$
is spanned by the $G$-representations without non-trivial $N$-fixed
vectors and we have an equivalence
\[
\mathrm{Vec}_{\bb Q}^{BG}\simeq\mathrm{Vec}_{\bb Q}^{B(G/N)}\times(\mathrm{Vec}_{\bb Q}^{B(G/N)})^{\perp}.
\]
This equivalence is realized explicitly as follows. Since $\mathrm{Vec}_{\bb Q}^{BG}$
is semi-simple, we can split each $G$-representation $V$ uniquely
as $V\simeq V^{N}\oplus V^{N\text{-}\mathrm{free}}$ with $V^{N}$
the space of $N$-fixed vectors and $V^{N\text{-}\mathrm{free}}\in(\mathrm{Vec}_{\bb Q}^{B(G/N)})^{\bot}$.
Consider the natural endomorphism:
\[
\alpha\colon\Id\to(-)_{N}\oto{\nm_{q}}(-)^{N}\to\Id,
\]
and its normalization $\varepsilon\coloneqq|N|^{-1}\alpha$. Unwinding
the definitions, we get
\[
\varepsilon(x)=\frac{1}{|N|}\sum_{g\in N}gx\quad\in V,
\]
which is an explicit formula for the $G$-equivariant projection $\varepsilon\colon V\to V$
onto the subspace $V^{N}\ss V$. 
\end{example}

In a similar fashion, we have:
\begin{prop}
\label{prop:Transfer_Idempotent}Let $\mathcal{C}\in\acat_{\st}$
and let $A\oto qB$ be a $\mathcal{C}$-acyclic and $\mathcal{C}$-ambidextrous
map. Consider
\[
\alpha\colon\Id_{\mathcal{C}^{A}}\oto{u_{!}}q^{*}q_{!}\oto{\nm_{q}}q^{*}q_{*}\oto{c_{*}}\Id_{\mathcal{C}^{A}}.
\]
The natural endomorphism $\varepsilon\coloneqq q^{*}(|q|)\alpha$
of $\Id_{\mathcal{C}^{A}}$ is idempotent, and it realizes the projection
onto the essential image of $q^{*}\colon\mathcal{C}^{B}\into\mathcal{C}^{A}$. 
\end{prop}

\begin{proof}
First, for every $Y\in(\mathcal{C}^{B})^{\bot},$ we have $q_{*}Y=0$
and hence $\alpha|_{(\mathcal{C}^{B})^{\bot}}=0$. Next, consider
the commutative diagram (see \defref{Cardinality}):
\[
\xymatrix{q^{*}\ar[r]^{u_{!}}\ar@{=}[rd] & q^{*}q_{!}q^{*}\ar[d]_{\wr}^{c_{!}}\ar[rr]_{\sim}^{\nm_{q}} &  & q^{*}q_{*}q^{*}\ar[r]^{c_{*}} & q^{*}\\
 & q^{*}\ar[rr]_{\sim}^{|q|^{-1}} &  & q^{*}\ar[u]_{u_{*}}^{\wr}\ar@{=}[ru]
}
\]
The composition along the top row computes the restriction of $\alpha$
along $q^{*}$, which is therefore invertible and coincides with $q^{*}(|q|^{-1})$.
It follows that $q^{*}(|q|)\alpha$ is the identity on the essential
image of $\mathcal{C}^{B}$ and is zero on $(\mathcal{C}^{B})^{\bot}$.
Thus, $q^{*}(|q|)\alpha$ equals the idempotent $\varepsilon$ which
projects onto the essential image of $\mathcal{C}^{A}$.
\end{proof}
\begin{rem}
\label{rem:Principal_Transfer_Idempotent}We note that if we furthermore
assume that $A\oto qB$ is \emph{principal} and $B$ is connected,
then by \propref{Amenable_Descent}, we have 
\[
q^{*}|q|=q^{*}B^{*}|F|=A^{*}|F|.
\]
Namely, $q^{*}|q|$ is constant with value $|F|$. Moreover, by \corref{Amenable_Acyclic}
and \propref{Acyclicity_Fiberwise}, the space $\Omega F$ is $\mathcal{C}$-amenable
and thus we get $|F|=|\Omega F|^{-1}$. To conclude, we can also write
$\varepsilon=|\Omega F|^{-1}\alpha$. 
\end{rem}

Our main motivation for \propref{Transfer_Idempotent}, is to generalize
\exaref{Rational_Semisimple} to higher heights: 
\begin{example}
\label{exa:Transfer_Idempotent}Let $\mathcal{C}\in\acat_{\st}^{\sad{\infty}}$
be of height $n$. For an abelian $p$-group $G$ with a subgroup
$N\le G$, we consider the fiber sequence 
\[
B^{n+1}N\to B^{n+1}G\oto qB^{n+1}(G/N).
\]
The map $q$ satisfies the conditions of \thmref{Height_Semisimple},
and hence $\mathcal{C}^{B^{n+1}G}$ splits as a product of $\mathcal{C}^{B^{n+1}(G/N)}$
and its right orthogonal complement. Moreover, by \propref{Transfer_Idempotent}
and \remref{Principal_Transfer_Idempotent}, the idempotent $\varepsilon$
is given by $\varepsilon=|B^{n}N|^{-1}\alpha$. The case $n=0$ is
a ``derived version'' of \exaref{Rational_Semisimple}.
\end{example}

We note that \exaref{Transfer_Idempotent} is essentially the only
interesting case of \thmref{Height_Semisimple}, as in view of \propref{Height_Everything}(3),
for any $n$-connected space $A$, we have $\mathcal{C}^{A}\simeq\mathcal{C}^{B^{n+1}\pi_{n+1}A}$.

\subsection{Chromatic Examples}

The main source of stable higher semiadditive $\infty$-categories is chromatic homotopy theory. In this subsection, we shall address the semiadditive height of such $\infty$-categories, using the properties of nil-conservative functors and the Nilpotence Theorem (see \cite{nilp2}). We begin by recalling the following notion:

\begin{defn}
[{\cite[Definition 4.4.1]{Ambi2018}}]
\label{def:Nil_Conservativity}We call a functor $F\colon\mathcal{C}\to\mathcal{D}$
in $\alg(\Pr_{\st})$ \textbf{nil-conservative}, if for every ring
$R\in\alg(\mathcal{C})$, if $F(R)=0$ then $R=0$\footnote{This notion is closely related  to the notion of ``nil-faithfulness'' defined in \cite{BalmerNil}.}.
\end{defn}

We also recall that nil-conservative functors are conservative on the full subcategory of dualizable objects. In particular, we have the following consequence regarding height:

\begin{prop}
\label{prop:Height_Detection} Let $F\colon{\cal C}\to{\cal D}$ be
a map in $\alg(\Pr_{\st})$. If \emph{$\mathcal{C}$} is $m$-semiadditive,
then so is $\mathcal{D}$. Furthermore, for every $0\le n\le m$,
if ${\cal C}$ is of height $\le n$ or $>n-1$ then so is ${\cal D}$
and the converse holds if $F$ is nil-conservative.
\end{prop}

\begin{proof}
Since $\mathcal{C}$ is $m$-semiadditive and $F\colon{\cal C}\to{\cal D}$
is a monoidal $m$-semiadditive functor, by \cite[Corollary 3.3.2(2)]{Ambi2018}, $\mathcal{D}$
is $m$-semiadditive as well. By \corref{Stable_Height_Monoiidal},
the height of $\mathcal{C}$ (resp. $\mathcal{D}$) is determined
by the height of $\one_{\mathcal{C}}$ (resp. $\one_{\mathcal{D}}$).
Moreover, $\htt(\one)\le n$ if and only if $\pn n\in\pi_{0}\one$
is invertible and $\htt(\one)>n-1$, if and only if $(1-\pn{n-1}\pn n)\in\pi_{0}\one$
is invertible. Thus, the claim follows from \cite[Corollary 4.4.5]{Ambi2018}.
\end{proof}

As a special case, we get:
\begin{cor}
\label{cor:Height_Modules}Let $\mathcal{C}\in\calg(\Pr_{\st})$ be
$m$-semiadditive and let $R\in\calg(\mathcal{C})$. For every integer $0\le n\le m$,
if ${\cal C}$ is of height $\le n$ or $>n-1$ then so is $\Mod_{R}(\mathcal{C})$.
The converse holds if tensoring with $R$ is nil-conservative.
\end{cor}

\begin{proof}
The claim follows from \propref{Height_Detection} for the colimit
preserving symmetric monoidal functor $R\otimes(-)\colon\mathcal{C}\to\Mod_{R}(\mathcal{C})$. 
\end{proof}

We now apply the above to higher semiadditive $\infty$-categories arising in chromatic homotopy theory. We begin with a special case in which cardinalities of Eilenberg-MacLane spaces can be computed explicitly.

\begin{prop}
\label{prop:En_Height}The $\infty$-category $\Theta_{n}=\Mod_{E_{n}}(\Sp_{K(n)})$
satisfies $\Ht(\Theta_{n})=n$.
\end{prop}

\begin{proof}
For $n=0$, the claim is clear, so assume $n\ge1$. This follows from
the explicit formula $\pn k=p^{\binom{n-1}{k}}$ given in \propref{Cardinality_EM}.
Indeed, for $k\ge n$, we have $\pn k=1$ and hence invertible. For
$0\le k<n$, the element $\pn k$ is a non-zero power of $p$, and
hence every $X\in\Theta_{n}$ is complete with respect to it as the
$\infty$-category $\Theta_{n}$ is $p$-complete. 
\end{proof}

In \cite[Theorem C]{Ambi2018}  we have shown that for a homotopy ring spectrum $R$,
the $\infty$-category $\Sp_{R}$ is $1$-semiadditive if and only
if it is $\infty$-semiadditive if and only if $\supp(R)=\{n\}$ for
some integer $n$. We now show that this $n$ is in fact the semiadditive height of $\Sp_R$.

\begin{thm}
\label{thm:Height_Chrom}Let $R$ be a homotopy ring spectrum\footnote{In fact, it suffices to assume that $R$ is a \emph{weak} \emph{ring}
in the sense of \cite[Definition 5.1.4]{Ambi2018}.}. If $\supp(R)=\{n\}$ for some integer $n$, then $\Ht(\Sp_{R})=n$.
In particular, we have 
\[
\Ht(\Sp_{K(n)})=\Ht(\Sp_{T(n)})=n.
\]
\end{thm}

\begin{proof}
We first consider the case $R=K(n)$. This follows by applying \corref{Height_Modules}
to the faithful commutative algebra $E_{n}\in\calg(\Sp_{K(n)})$ and
the fact that $\Ht(\Theta_{n})=n$ (\propref{En_Height}). For a general
homotopy ring $R$ with $\supp(R)=\{n\}$ (such as $R=T(n)$), consider
the functor $L_{K(n)}\colon\Sp_{R}\to\Sp_{K(n)}$. It is nil-conservative
by \cite[Proposition 5.1.15]{Ambi2018}. Since $\Sp_{K(n)}$ is of height
$n$, the claim follows from \propref{Height_Detection}. 
\end{proof}

\section{Modes}

In this section, we use the theory of idempotent algebras in $\Pr$,
which we call \emph{modes}, to further study the interaction of stability
and higher semiadditivity. 

\subsection{Idempotent Algebras}

We begin with a general discussion about idempotent algebras in symmetric
monoidal $\infty$-categories, as a means to encode properties, which
induce ``canonical structure''. 

\subsubsection{Definitions \& characterizations}

Following \cite[Definition 4.8.2.1]{ha}, given a symmetric monoidal $\infty$-category
${\cal C}$, we say that a morphism $\one\oto uX$ in $\mathcal{C}$
exhibits $X$ as an \textbf{idempotent object} of ${\cal C}$, if
\[
X\simeq X\otimes\one\oto{1\otimes u}X\otimes X
\]
is an isomorphism. By \cite[Proposition 4.8.2.9]{ha}, an idempotent object $\one\oto uX$
admits a unique commutative algebra structure for which $u$ is the
unit. Conversely, the unit $\one\oto uR$ of a commutative algebra
$R$, exhibits it as an idempotent object if and only if the multiplication
map $R\otimes R\oto mR$ is an isomorphism. In this case we call $R$
an \textbf{idempotent} \textbf{algebra}. More precisely, the functor
$\calg(\mathcal{C})\to\mathcal{C}_{\one/}$ which forgets the algebra
structure and remembers only the unit map, induces an equivalence
of $\infty$-categories from the full subcategory of idempotent algebras  $\calg^{\mathrm{idem}}(\mathcal{C}) \subseteq \calg(\mathcal{C})$
to the full subcategory of idempotent objects \cite[Proposition  4.8.2.9]{ha}. The fundamental
feature of an idempotent algebra $R$, is that the forgetful functor
$\Mod_{R}(\mathcal{C})\to\mathcal{C}$ is fully faithful, and its
essential image consists of those objects for which the map $Y\otimes\one\oto{1\otimes u}Y\otimes R$
is an isomorphism \cite[Proposition 4.8.2.10]{ha}. Thus, it is a \emph{property} of an object
in $\mathcal{C}$ to have the\emph{ structure} of an $R$-module.
We shall say that $R$ \textbf{classifies} the property of being an
$R$-module. 
\begin{example}
For $\mathcal{C}=\mathrm{Ab}$, the idempotent algebras are classically
known as \emph{solid} \emph{rings} \cite[Definition 2.1]{BKCore}. These include
for example $\bb Q$ and $\bb F_{p}$. We note that for $\mathcal{C}=\Sp$,
the ring $\bb Q$ is still idempotent, classifying the property of
being rational, but $\bb F_{p}$ is \emph{not} idempotent. The idempotent
rings in $\Sp$ correspond precisely to the \emph{smashing} \emph{localizations}.
\end{example}

Given an idempotent ring $R\in\calg(\mathcal{C}),$ the forgetful
functor $\Mod_{R}(\mathcal{C})\to\mathcal{C}$ admits a left adjoint
\[
R\otimes(-)\colon\mathcal{C}\to\Mod_{R}(\mathcal{C}).
\]
This is a \emph{localization} functor, which can be thought of as
forcing the property classified by $R$ in a universal way. In line
with the standard terminology for localizations of spectra, we set:
\begin{defn}
\label{def:Smashing}Let $L\colon\mathcal{C}\to\mathcal{D}$ be a
map in $\calg(\cat)$. We say that $L$ is a \textbf{smashing} \textbf{localization}
if there exists an idempotent algebra $R$ in $\mathcal{C}$ and an
isomorphism $\Mod_{R}(\mathcal{C})\iso\mathcal{D}$ in $\calg(\cat)$,
such that $L$ is the composition 
\[
\mathcal{C}\oto{R\otimes(-)}\Mod_{R}(\mathcal{C})\iso\mathcal{D}.
\]
\end{defn}

We note that for a smashing localization $L\colon\mathcal{C}\to\mathcal{D}$,
there is always a fully faithful right adjoint $F\colon\mathcal{D}\to\mathcal{C}$,
which is lax symmetric monoidal \cite[Corollary 7.3.2.7]{ha}, and we can identify the
idempotent algebra $R$ with $FL(\one_{\mathcal{C}})$. 
To characterize smashing localizations, we first introduce some terminology. 
Let $L\colon\mathcal{C}\to\mathcal{D}$ in $\calg(\cat)$,
which admits a (lax symmetric monoidal) right adjoint $F\colon\mathcal{D}\to\mathcal{C}$. For every $X\in\mathcal{C}$ and $Y\in\mathcal{D}$, we have a natural
map 
\[
X\otimes F(Y)\oto{\alpha}F(L(X)\otimes Y),
\]
which is the mate of
\[
L(X\otimes F(Y))\simeq L(X)\otimes LF(Y)\oto{L(X)\otimes c_{Y}}L(X)\otimes Y,
\]
where $c\colon LF\to\Id$ is the counit of the adjunction. We say
that the adjunction $L\dashv F$ satisfies the \textbf{projection
formula} if the map $\alpha$ is an isomorphism for all $X,Y$. The
map $\alpha$ is compatible with the unit of the adjunction $\Id\oto uFL$
in the following sense:

\begin{lem}
\label{lem:Unit_Proj_Formula} For all $X,Z\in\mathcal{C}$, the following
diagram is commutative
\[
\xymatrix{ & X\otimes Z\ar[ld]_{X\otimes u_{Z}}\ar[rd]^{u_{X\otimes Z}}\\
X\otimes FL(Z)\ar[r]^{\alpha\quad} & F(L(X)\otimes L(Z))\ar[r]^{\quad \sim} & FL(X\otimes Z)
}
\]
\end{lem}

\begin{proof}
Passing to the mates under the adjunction $L\dashv F$ and unwinding
the definitions, this follows from the zigzag identities.
\end{proof}
We can now characterize the smashing localizations among all localizations
as those that satisfy the projection formula (compare with \cite[Proposition 5.29]{MNN17}):
\begin{prop}
\label{prop:Idemp_Proj_Formula}A map $L\colon\mathcal{C}\to\mathcal{D}$
in $\calg(\cat)$ is a smashing localization if and only if it admits
a fully faithful right adjoint $F\colon\mathcal{D}\to\mathcal{C}$
(i.e. it is a localization) and the adjunction $L\dashv F$ satisfies
the projection formula.
\end{prop}

\begin{proof}
In the ``only if'' direction, the right adjoint of $R\otimes(-)\colon\mathcal{C}\to\Mod_{R}(\mathcal{C})$
is the forgetful functor $F\colon\Mod_{R}(\mathcal{C})\to\mathcal{C}$,
which is fully faithful since $R$ is idempotent. For $X\in\mathcal{C}$
and $M\in\Mod_{R}(\mathcal{C})$, the projection formula transformation
is the composition
\[
X\otimes F(M)\oto{u_{(X\otimes F(M))}}FL(X\otimes F(M))\iso F(L(X)\otimes LF(M))\oto{L(X)\otimes c_{M}}F(L(X)\otimes M).
\]

The map $c_{M}$ is an isomorphism since $F$ is fully faithful. The
map $u_{(X\otimes F(M))}$ is an isomorphism because $X\otimes F(M)$
admits a structure of an $R$-module and hence in the essential image
of $L$.

For the ``if'' direction, consider the commutative ring 
\[
R=FL(\one_{\mathcal{C}})=F(\one_{\mathcal{D}}).
\]
By the projection formula, we have a natural isomorphism
\[
X\otimes R\simeq X\otimes F(\one_{\mathcal{D}})\simeq F(L(X)\otimes\one_{\mathcal{D}})\simeq FL(X).
\]
Since $F$ is fully faithful, the unit map $\Id\oto uFL$ exhibits
$FL\colon\mathcal{C}\to\mathcal{C}$ as a localization (as in \cite[Proposition 5.2.7.4]{htt}). Moreover, by \lemref{Unit_Proj_Formula},
the unit map $\Id\oto uFL$ is induced from tensoring with the unit
$\one_{\mathcal{C}}\oto uR$ of $R$. Thus, by  \cite[Proposition 4.8.2.4]{ha}, $R$ is
an idempotent ring. Moreover, by \cite[Proposition 4.8.2.10]{ha}, the forgetful functor
$\Mod_{R}(\mathcal{C})\to\mathcal{C}$ is a symmetric monoidal equivalence
onto the essential image of $FL$ with the localized symmetric monoidal
structure, which is equivalent to $\mathcal{D}$.
\end{proof}
\begin{rem}
\label{rem:Smashing_Recollement}Let $\mathcal{C}$ be presentably
symmetric monoidal and let $R$ be an idempotent algebra in $\mathcal{C}$.
The fully faithful forgetful functor $F\colon\Mod_{R}(\mathcal{C})\into\mathcal{C}$
admits also a \emph{right }adjoint. Hence, if $\mathcal{C}$ is moreover
\emph{stable}, then $F$ exhibits $\mathcal{C}$ as a recollement
of $\Mod_{R}(\mathcal{C})$ and its right orthogonal complement.
\end{rem}

\subsubsection{Poset structure}

Another nice characterization of idempotent algebras is as the $(-1)$-cotruncated
objects of $\calg(\mathcal{C})$: 
\begin{prop}
\label{prop:Idemp_Poset}Let $\mathcal{C}\in\calg(\cat)$. A commutative
algebra $R\in\calg(\mathcal{C})$ is idempotent, if and only if
for all $S\in\calg(\mathcal{C})$ the
space $\map_{\calg(\mathcal{C})}(R,S)$ is either empty or contractible. Moreover, it is non-empty if and
only if $S$ is an $R$-module.
\end{prop}

\begin{proof}
By \cite[Proposition 3.2.4.7]{ha}, the tensor product of commutative algebras is the coproduct
in $\calg(\mathcal{C})$. Moreover, the multiplication map 
\[
R\sqcup R=R\otimes R\oto mR
\]
is the categorical fold map. Thus, it follows by the dual of \cite[Lemma .5.5.6.15]{htt},
that $R$ is idempotent if and only if it is $(-1)$-cotruncated.
We note that \cite[Lemma .5.5.6.15]{htt} requires the $\infty$-category to admit
finite limits; however, the only limit used in the proof is the one defining the diagonal map, which in our case corresponds to
the coproduct defining the fold map. Now, if there exists a map $R\to S$
in $\calg(\mathcal{C})$, then $S$ is an $R$-algebra and in particular
an $R$-module. Conversely, for every $S\in\calg(\mathcal{C})$, we
have maps
\[
R=R\otimes\one\oto{1\otimes u_{S}}R\otimes S
\]
\[
S=\one\otimes S\oto{u_{R}\otimes1}R\otimes S
\]
in $\calg(\mathcal{C})$. If $S$ is an $R$-module, then the map
$u_{R}\otimes1$ is an isomorphism, and thus 
\[
(u_{R}\otimes1)^{-1}\circ(1\otimes u_{S})\colon R\to S
\]
is a map in $\calg(\mathcal{C})$.
\end{proof}
We also have the following non-commutative analogue of \propref{Idemp_Poset}:
\begin{prop}
\label{prop:Idemp_Alg}Let $\mathcal{C}\in\calg(\cat)$ and let $R\in\calg^{\mathrm{idem}}(\mathcal{C})$. For every algebra $S\in\alg(\mathcal{C})$, the space $\map_{\alg(\mathcal{C})}(R,S)$
is either empty or contractible and it is non-empty if and only if
$S$ is an $R$-module.
\end{prop}

\begin{proof}
As in the proof of \propref{Idemp_Poset}, $S$ is an $R$-module
if and only if there exists a map $R\to S$ in $\alg(\mathcal{C})$.
If $S$ is an $R$-module, then $S\in\alg(\Mod_{R}(\mathcal{C}))$
and 
\[
\map_{\alg(\mathcal{C})}(R,S)\simeq\map_{\alg(\Mod_{R}(\mathcal{C}))}(R,S)\simeq\pt,
\]
since $R$ is the initial object of $\alg_{R}(\mathcal{C})=\alg(\Mod_{R}(\mathcal{C}))$.
\end{proof}
As a consequence of \propref{Idemp_Poset}, the $\infty$-category $\calg^{\mathrm{idem}}(\mathcal{C})$ of idempotent algebras
is a  \emph{poset}. This poset admits binary meets:
\begin{prop}
\label{prop:Idemp_Tensor}Let $\mathcal{C}\in\calg(\cat)$ and $R,S \in  \calg^{\mathrm{idem}}(\mathcal{C})$. Then $R\otimes S$ is an
idempotent algebra which classifies the conjunction of the properties
classified by $R$ and $S$.
\end{prop}

\begin{proof}
Let $\one\oto{u_{R}}R$ be the unit of $R$ and $\one\oto{u_{S}}S$
the unit of $S$. We consider the composition
\[
u_{R\otimes S}\colon\one\simeq\one\otimes\one\oto{u_{R}\otimes1}R\otimes\one\oto{1\otimes u_{S}}R\otimes S.
\]
After tensoring with $R$ the first map becomes an isomorphism and
after tensoring with $S$ the second map becomes an isomorphism.
Thus, $u_{R\otimes S}$ exhibits $R\otimes S$ as an idempotent object.
Given an $R\otimes S$-module $M$, tensoring $\one\oto{u_{R\otimes S}}R\otimes S$
with $M$ is an isomorphism. Thus, $M\simeq R\otimes S\otimes M$
and hence $M$ is both an $R$-module and an $S$-module. Conversely,
tensoring $\one\oto{u_{R\otimes S}}R\otimes S$ with $M$ is given
by the composition:

\[
\one\otimes\one\otimes M\oto{u_{R}\otimes1\otimes1}R\otimes\one\otimes M\oto{1\otimes u_{S}\otimes1}R\otimes S\otimes M
\]
Hence, if $M$ is both an $R$-module and an $S$-module, then tensoring
$\one\oto{u_{R}}R$ and $\one\oto{u_{S}}S$ with $M$ is an isomorphism,
and so tensoring $\one\oto{u_{R\otimes\mathcal{S}}}R\otimes S$ with
$M$ is an isomorphism as well. Thus, $M$ is an $R\otimes S$-module.
\end{proof}
Idempotent objects are also closed under sifted colimits in $\mathcal{C}$.
\begin{prop}\label{prop:sifted_mode}
Let $\mathcal{C}\in\calg(\cat)$, and let $\mathcal{J}$ be a sifted $\infty$-category  such that $\mathcal{C}$ is compatible with $\mathcal{J}$-indexed colimits. Then
\begin{enumerate}
\item The $\infty$-category $\calg^{\mathrm{idem}}(\mathcal{C})$ admits $\mathcal{J}$-indexed colimits.
\item The forgetful functor $\calg^{\mathrm{idem}}(\mathcal{C}) \to \mathcal{C}$ preserves $\mathcal{J}$-indexed colimits.
\item Given a functor $F\colon \mathcal{J} \to \calg^{\mathrm{idem}}(\mathcal{C})$, the idempotent algebra $\colim  F$ classifies the conjunction of the properties classified by $F(j)$ for all $j \in \mathcal{J}$.
\end{enumerate}
\end{prop}
\begin{proof}
By \cite[Corollary 3.2.3.2]{ha} the $\infty$-category $\calg(\mathcal{C})$ admits $\mathcal{J}$-indexed colimits and the forgetful functor $\calg(\mathcal{C}) \to \mathcal{C}$ preserves  $\mathcal{J}$-indexed colimits. 
Thus, to prove (1) and (2) it  suffices to show that $\calg^{\mathrm{idem}}(\mathcal{C}) \subseteq \calg(\mathcal{C})$ is closed under $\mathcal{J}$-indexed colimits. 
Since $\mathcal{J}$ is sifted, the diagonal map $\mathcal{J} \to \mathcal{J}\times \mathcal{J}$ is cofinal. Therefore,  given a functor $F\colon \mathcal{J} \to \calg(\mathcal{C})$,  the multiplication map 
\[
\colim F \otimes \colim F \to \colim F
\]
can be identified with the colimit of the multiplication maps $F(j)\otimes F(j) \to F(j)$ for $j \in \mathcal{J}$. Hence, if $F(j)$ is an idempotent algebra for every $j\in \mathcal{J}$, so is $\colim F$. We shall now prove (3). 
First,  for every $j \in \mathcal{J}$ there is a canonical algebra map $F(j) \to \colim F$. Thus, every ($\colim F$)-module admits an $F(j)$-module structure for every $j \in \mathcal{J}$. 
It remains to prove that if $M \in \mathcal{C}$ admits an $F(j)$-module  structure for every $j \in \mathcal{J}$, then $$ M\otimes \one \xrightarrow{1 \otimes u} M \otimes \colim F$$ is an isomorphism. 
Since $\mathcal{C}$ is compatible with $\mathcal{J}$-indexed colimits, the map $1 \otimes u$ above is the colimit of the isomorphisms $ M\otimes \one \xrightarrow{1 \otimes u_j}  M \otimes  F(j)$.
\end{proof}
Consequently, under mild conditions on $\mathcal{C}$, the poset $\calg^{\mathrm{idem}}(\mathcal{C})$ admits arbitrary  meets.
\begin{cor}\label{cor:coco_mode}
Let $\mathcal{C}\in\calg(\cat)$  which is compatible with filtered colimits. Then the poset $\calg^{\mathrm{idem}}(\mathcal{C})$  is cocomplete. Moreover,  given a functor $F\colon \mathcal{J} \to \calg^{\mathrm{idem}}(\mathcal{C})$, the idempotent algebra $$\colim  F = \sup \limits_{ j \in \mathcal{J}}F(j)$$ classifies the conjunction of the properties classified by $F(j)$ for all $j \in \mathcal{J}$.
\end{cor}
\begin{proof}
First, $\calg^{\mathrm{idem}}(\mathcal{C})$ admits an initial object which is $\one_\mathcal{C}$. Second, by \propref{Idemp_Tensor}, $\calg^{\mathrm{idem}}$ admits binary coproducts. Since every filtered $\infty$-category is sifted by \cite[Example 5.5.8.3]{htt} we get by \propref{sifted_mode} that $\calg^{\mathrm{idem}}(\mathcal{C})$ admits filtered colimits. Since $\calg^{\mathrm{idem}}(\mathcal{C})$ is a poset, we deduce that it is cocomplete. Furthermore, \propref{Idemp_Tensor} and \propref{sifted_mode} also imply that the colimit classifies the conjunction of the properties classified by the idempotent algebras in the diagram.
\end{proof}

Under some conditions, \emph{disjoint} idempotent algebras have also
binary joins.
\begin{prop}
\label{prop:Idem_Prod}Let $\mathcal{C}\in\calg(\cat)$ which is compatible with all small colimits and is $0$-semiadditive. Let $R,S$ be idempotent algebras in $\mathcal{C}$.
If $R\otimes S\simeq0$, then $R\times S$ is an idempotent algebra,
which classifies the property of an object $X\in\mathcal{C}$ to be
of the form $X_{R}\oplus X_{S}$, where $X_{R}$ is an $R$-module
and $X_{S}$ is an $S$-module.
\end{prop}

\begin{proof}
By assumption, the tensor product preserves binary coproducts in each variable. Since $\mathcal{C}$ is $0$-semiadditive, we get that the tensor product also preserves binary \emph{products} in each variable. Thus, 
\[
(R\otimes R)\times (S\otimes S)\simeq(R\times S)\otimes(R\times S)\oto{m_{R\times S}}R\times S
\]
coincides with $m_{R}\times m_{S}$, which is an isomorphism. Thus,
$m_{R\times S}$ is an isomorphism, and therefore, $R\times S$ is an
idempotent algebra. The projection maps $R\times S\to R$ and $R\times S\to S$
induce the extension of scalars functors
\[
\xymatrixrowsep{1em}
\xymatrixcolsep{4em}
\xymatrix{ & \Mod_{R\times S}(\mathcal{C})\ar[ld]_{F_{R}}\ar[rd]^{F_{S}}\\
\Mod_{R}(\mathcal{C}) &  & \Mod_{S}(\mathcal{C})
}
\]
which in turn induce a functor 
\[
F\colon\Mod_{R\times S}(\mathcal{C})\to\Mod_{R}(\mathcal{C})\times\Mod_{S}(\mathcal{C}).
\]
Since $F_{R}$ and $F_{S}$ are left adjoints, by \cite[Theorem B]{AD} the functor
$F$ admits a right adjoint $G$, which is given object-wise by $G(X_{R},X_{S})\simeq X_{R}\times X_{S}$.
To complete the proof of the claim, it would suffice to show that $G$
is an equivalence. We do this by showing that $G$ is conservative
and $F$ is fully faithful. By \lemref{Invertible_Monad}, in order
to show that $F$ is fully faithful, it suffices to show that $GF$
is an equivalence. For every $(R\times S)$-module $X$, we have
\[
GF(X)=(R\otimes_{R\times S}X)\times(S\otimes_{R\times S}X)\simeq(R\times S)\otimes_{R\times S}X\simeq X.
\]
To show that $G$ is conservative, it suffices to observe that the
underlying $\mathcal{C}$-object of $G(X_{R},X_{S})$ is the direct
sum $X_{R}\oplus X_{S}$ and both $X_{R}$ and $X_{S}$ are retracts
of $X_{R}\oplus X_{S}$.
\end{proof}

\subsection{Theory of Modes }

We now specialize the notion of idempotent algebras to the $\infty$-category
$\Pr$ of presentable $\infty$-categories and colimit preserving
functors. 

\subsubsection{Tensor of presentable $\infty$-categories}

Recall from \cite[Proposition 4.8.1.15]{ha}, that the $\infty$-category $\Pr$ admits a
closed symmetric monoidal structure. The unit is  $\mathcal{S}\in\Pr$, and for every $\mathcal{C},\mathcal{D}\in\Pr$, the internal
hom and tensor product are given respectively by 
\[
\hom(\mathcal{C},\mathcal{D})=\fun^{L}(\mathcal{C},\mathcal{D}),\qquad\mathcal{C}\otimes\mathcal{D}=\fun^{R}(\mathcal{C}^{\op},\mathcal{D})\simeq\fun^{R}(\mathcal{D}^{\op},\mathcal{C}).
\]

It is worth spelling out in what sense the above formula for the tensor
product is functorial. Given a functor $\mathcal{D}_{1}\oto F\mathcal{D}_{2}$
in $\Pr$ with right adjoint $G$, the induced functor $\mathcal{C}\otimes\mathcal{D}_{1}\oto{\Id_{\mathcal{C}}\otimes F}\mathcal{C}\otimes\mathcal{D}_{2}$
is the left adjoint of 
\[
\fun^{R}(\mathcal{C}^{\op},\mathcal{D}_{2})\oto{G\circ(-)}\fun^{R}(\mathcal{C}^{\op},\mathcal{D}_{1}).
\]
From this we get that tensoring with $\mathcal{C}$ preserves reflective
localizations:
\begin{lem}
\label{lem:FF_Right_Adj}Let $\mathcal{C}$ and $\mathcal{D}_{1}\oto F\mathcal{D}_{2}$
in $\Pr$. If $F$ admits a fully faithful (resp. conservative) right
adjoint, then so does $\Id_{\mathcal{C}}\otimes F$ .
\end{lem}

\begin{proof}
By the above formula for $\Id_{\mathcal{C}}\otimes F$, its right adjoint is given by
\[
\fun^{R}(\mathcal{C}^{\op},\mathcal{D}_{2})\oto{G\circ(-)}\fun^{R}(\mathcal{C}^{\op},\mathcal{D}_{1}),
\]
where $G$ is the right adjoint of $F$. Thus, if $G$ is fully faithful
(resp. conservative), then post composition with $G$ is fully faithful
(resp. conservative) as well.
\end{proof}
\begin{rem}
On the other hand, if $\mathcal{D}_{1}\oto F\mathcal{D}_{2}$ is itself
fully faithful or conservative, then $\Id_{\mathcal{C}}\otimes F$
need not be. For example, let $\Sp^{\cn}\ss\Sp$ be the full subcategory
of connective spectra. One can show that $\Sp^{\cn}\otimes\set\simeq\mathrm{Ab}$,
while $\Sp\otimes\set\simeq0$ (e.g. by \propref{Mode_Localization}).
Thus, tensoring the fully faithful inclusion $\Sp^{\cn}\into\Sp$
with the category $\set$, produces the zero functor $\mathrm{Ab}\to0$.
\end{rem}

Another general fact which we shall require is the preservation of
recollements under base change. Recall that $\Pr_{\st}\ss\Pr$ is
the full subcategory of stable presentable $\infty$-categories.
\begin{prop}
\label{prop:Recollement_Base_Change}Let $\mathcal{C}\in\Pr_{\st}$
and assume it is a recollement of $\mathcal{C}_{\circ}\ss\mathcal{C}$
and $\mathcal{C}_{\circ}^{\perp}\ss\mathcal{C}$. For every $\mathcal{D}\in\Pr$,
the morphism $\mathcal{C}_{\circ}\otimes\mathcal{D}\to\mathcal{C}\otimes\mathcal{D}$
exhibits $\mathcal{C}\otimes\mathcal{D}$ as a recollement of $\mathcal{C}_{\circ}\otimes\mathcal{D}$
and $(\mathcal{C}_{\circ}\otimes\mathcal{D})^{\perp}\simeq\mathcal{C}_{\circ}^{\perp}\otimes\mathcal{D}$.
\end{prop}

\begin{proof}
Let $\mathcal{C}_{\circ}\oto F\mathcal{C}$ be the inclusion functor.
We denote by $\mathcal{C}\oto L\mathcal{C}_{\circ}$ and $\mathcal{C}\oto R\mathcal{C}_{\circ}$
the left and right adjoints of $F$ respectively. We observe that
$\mathcal{C}_{\circ}$ is presentable as an accessible localization
of $\mathcal{C}$, and hence both functors $F$ and $L$ are morphisms
in $\Pr$. The adjunction $F\dashv R$ induces an adjunction
\[
F\circ(-)\colon\fun(\mathcal{D}^{\op},\mathcal{C}_{\circ})\adj\fun(\mathcal{D}^{\op},\mathcal{C})\colon R\circ(-).
\]
Since $F$ and $R$ are both right adjoints, this adjunction restricts
to an adjunction on the full subcategories spanned by the right adjoints
on both sides,
\[
F\circ(-)\colon\mathcal{D}\otimes\mathcal{C}_{\circ}\simeq\fun^{R}(\mathcal{D}^{\op},\mathcal{C}_{\circ})\adj\fun^{R}(\mathcal{D}^{\op},\mathcal{C})\simeq\mathcal{D}\otimes\mathcal{C}\colon R\circ(-).
\]
On the other hand, the left adjoint of $R\circ(-)$ is $F\otimes\mathcal{D}$
and the left adjoint of $F\circ(-)$ is $L\otimes\mathcal{D}$. It
follows that $L\otimes\mathcal{D}$ is the left adjoint of $F\otimes\mathcal{D}$.
To conclude, $F\otimes\mathcal{D}$ admits a right adjoint (as a morphism
in $\Pr)$ and also a left adjoint, given by $L\otimes\mathcal{D}$.
Furthermore, by \lemref{FF_Right_Adj}, $F\otimes\mathcal{D}$ is
also fully faithful. The $\infty$-categories $\mathcal{C}_{\circ}\otimes\mathcal{D}$
and $\mathcal{C}\otimes\mathcal{D}$ are stable by \cite[Proposition 4.8.2.18]{ha} because
$\mathcal{C}_{\circ}$ and $\mathcal{C}$ are. Hence, we deduce that
$F\otimes\mathcal{D}$ exhibits $\mathcal{C}\otimes\mathcal{D}$ as
a recollement of $\mathcal{C}_{\circ}\otimes\mathcal{D}$ and $(\mathcal{C}_{\circ}\otimes\mathcal{D})^{\perp}.$ 

It remains to identify $(\mathcal{C}_{\circ}\otimes\mathcal{D})^{\perp}$
with $\mathcal{C}_{\circ}^{\perp}\otimes\mathcal{D}$. Note that
$(\mathcal{C}_{\circ}\otimes\mathcal{D})^{\perp}$ is the full subcategory
of $\mathcal{C}\otimes\mathcal{D}$ spanned by objects on which the
right adjoint of $F\otimes\mathcal{D}$ is zero. This right adjoint
is given by 
\[
R\circ(-)\colon\fun^{R}(\mathcal{D}^{\op},\mathcal{C})\to\fun^{R}(\mathcal{D}^{\op},\mathcal{C}_{\circ}).
\]

On the other hand, the inclusion $F^{\perp}\colon\mathcal{C}_{\circ}^{\perp}\into\mathcal{C}$
is right adjoint to its left adjoint $L^{\perp}\colon\mathcal{C}\to\mathcal{C}_{\circ}^{\perp}$,
and hence the right adjoint of $L^{\perp}\otimes\mathcal{D}$ is given
by 
\[
F^{\perp}\circ(-)\colon\fun^{R}(\mathcal{D}^{\op},\mathcal{C}_{\circ}^{\perp})\to\fun^{R}(\mathcal{D}^{\op},\mathcal{C}).
\]
This is a fully faithful functor whose essential image consists precisely
of objects on which $R\circ(-)$ is zero. Thus, we have canonically
identified $(\mathcal{C}_{\circ}\otimes\mathcal{D})^{\perp}$ with
$\mathcal{C}_{\circ}^{\bot}\otimes\mathcal{D}$.
\end{proof}

\subsubsection{Definition \& examples of modes}
We are now ready to introduce the central notion of this section:
\begin{defn}
\label{def:Mode}A \textbf{mode} is an idempotent algebra in $\Pr$. We denote by $$\mathrm{Mode} \coloneqq \calg^{\mathrm{idem}}(\Pr)\ss\calg(\Pr)$$
the
full subcategory spanned by modes. 
\end{defn}

Applying the preceding results on idempotent algebras
to $\Pr$, we get the following:
\begin{prop}
\label{prop:Mode_Lattice}
\hfill 
\begin{enumerate}
\item $\mode$ is a (large) poset.
\item $\mode$ is co-complete. Moreover, the colimit of a diagram of modes classifies the conjunction  of the properties classified by the modes in the diagram.
\item $\mathcal{S}\in\mode$ is the initial mode and it classifies the
empty property (which is always satisfied).
\item $ \mathcal{M} \sqcup \mathcal{N} = \mathcal{M} \otimes \mathcal{N}$ for all $\mathcal{M},\mathcal{N}\in\mode$.
\item The forgetful functor $\mode \to \Pr$ preserves  sifted colimits.
\item  $0\in\mode$ is the terminal mode and it classifies the property
of being equivalent to $0$.
\item For $\mathcal{M},\mathcal{N}\in\mode$, if $\mathcal{M}\otimes\mathcal{N}=0$,
then their join is given by $\mathcal{M}\times\mathcal{N}$, and it
classifies the property of being a direct product of an $\mathcal{M}$-module
and an $\mathcal{N}$-module.
\end{enumerate}
\end{prop}

\begin{proof}
(1) follows from \propref{Idemp_Poset}. (2) follows from \corref{coco_mode}  and the fact that $\Pr$ is closed symmetric monoidal. (3) follows from the fact that $\mathcal{S}$ is the unit of $\Pr$. (4) follows from  \propref{Idemp_Tensor}. (5) follows from  \propref{sifted_mode}.
(6) follows from the fact that $0$ is a zero object of $\Pr$.
(7) follows from \propref{Idem_Prod}, since $\Pr$ is $0$-semiadditive (\cite[Example 4.3.11]{HopkinsLurie}). 
\end{proof}
In addition to the initial and terminal modes, we
also have the following (far from exhaustive) list of modes:\footnote{All these can be found in \cite[Section 4.8.2]{ha} with the exception of (3), which
can be deduced from \propref{Mode_Localization}.}
\begin{example}
\label{exa:Modes}
\hfill
\begin{enumerate}
\item $(0\to1)$ is the \emph{boolean }mode which classifies the property
of being equivalent to a poset (i.e. the mode of propositional logic).
\item $\set$ is the \emph{discrete }mode, which classifies the property
of being equivalent to an ordinary category (i.e. the mode of ordinary,
as opposed to ``higher'', mathematics). 
\item $\mathrm{Ab}$ is the \emph{discrete} \emph{additive }mode\emph{,
}which classifies the property of being equivalent to an ordinary
additive category. 
\item $\mathcal{S}_{*}$ is the \emph{pointed }mode, which classifies the
property of having a zero object. 
\item $\Sp$ is the \emph{stable }mode, which classifies the property of
being stable. 
\end{enumerate}
\end{example}

Given a mode $\mathcal{M}$, the fully faithful forgetful functor
$\Mod_{\mathcal{M}}(\Pr)\into\Pr$ admits a left adjoint (i.e. localization),
which is given by 
\[
\mathcal{C}\mapsto\mathcal{M}\otimes\mathcal{C}=\fun^{R}(\mathcal{M}^{\op},\mathcal{C}).
\]
This procedure should be thought of as forcing $\mathcal{C}$ to be
in the mode $\mathcal{M}$ in a universal way. 
\begin{example}
For the stable mode $\Sp$, the $\infty$-category $\Sp\otimes\,\mathcal{C}$
is the \emph{stabilization} $\Sp(\mathcal{C})\in\Pr_{\st}$ \cite[Example 4.8.1.23]{ha}.
Similarly, for the discrete mode $\set$, the $\infty$-category $\set\otimes\,\mathcal{C}$
is the $0$-truncation $\tau_{\le0}\mathcal{C}$ \cite[Remark 4.8.2.17]{ha}.

The general results for idempotent algebras have the following implication:
\end{example}

\begin{prop}
\label{prop:Mode_Hereditary}Let $\mathcal{M}$ be a mode and $\mathcal{C}\in\alg(\Pr)$
which is an $\mathcal{M}$-module. The fully faithful embedding $\Mod_{\mathcal{M}}(\Pr)\into\Pr$
induces an equivalence of $\infty$-categories 
\[
\LMod_{\mathcal{C}}(\Mod_{\mathcal{M}}(\Pr))\iso\LMod_{\mathcal{C}}(\Pr).
\]
In particular, every $\mathcal{D}\in\LMod_{\mathcal{C}}(\Pr)$ is an
$\mathcal{M}$-module.
\end{prop}

\begin{proof}
Since $\mathcal{C}$ is an $\mathcal{M}$-module, by \propref{Idemp_Alg},
there is a map of algebras $\mathcal{M}\to\mathcal{C}$ and the claim
follows. 
\end{proof}
The $(\infty,2)$-categorical structure of $\Pr$ allows further constructions
of modes beyond those provided by \propref{Mode_Lattice}. Primarily,
modes can be \emph{localized}. 

\subsubsection{Localization of modes}

Given a mode $\mathcal{M}$, every $\mathcal{M}$-module $\mathcal{C}$
is by definition left-tensored over $\mathcal{M}$, and hence in particular
\emph{enriched} over $\mathcal{M}$ \cite[Proposition 4.2.1.33]{ha}. For every $X,Y\in\mathcal{C},$
we denote by $\hom^{\mathcal{M}}(X,Y)$ the corresponding hom-object
in $\mathcal{M}$. The $\mathcal{M}$-enrichment of an $\mathcal{M}$-module
$\mathcal{C}$ can be explicitly described via the $\mathcal{M}$-enriched
Yoneda embedding:
\[
\mathcal{C}\simeq\mathcal{C}\otimes\mathcal{M}\simeq\fun^{R}(\mathcal{C}^{\op},\mathcal{M})\into\fun(\mathcal{C}^{\op},\mathcal{M}).
\]

\begin{defn}
Let $\mathcal{M}$ be a mode, let $\mathcal{M}_{\circ}\ss\mathcal{M}$
be a reflective full subcategory, and let $\mathcal{C}$ be an $\mathcal{M}$-module.
We say that an object $X\in\mathcal{C}$ is $\mathcal{M}_{\circ}$\textbf{-local}
if for every $Z\in\mathcal{C}$, the enriched hom-object $\hom^{\mathcal{M}}(Z,X)$
lies in $\mathcal{M}_{\circ}.$ Furthermore, we say that $\mathcal{C}$
itself is $\mathcal{M}_{\circ}$-local, if every object of $\mathcal{C}$
is $\mathcal{M}_{\circ}$-local.
\end{defn}

\begin{prop}
\label{prop:Mode_Localization}Let $\mathcal{M}$ be a mode and $\mathcal{M}_{\circ}\ss\mathcal{M}$
an accessible reflective full subcategory, which is compatible with
the symmetric monoidal structure of $\mathcal{C}$. Let $L$ be the
left adjoint of the inclusion $\mathcal{M}_{\circ}\into\mathcal{M}$.
The composition
\[
\mathcal{S}\oto u\mathcal{M}\oto L\mathcal{M}_{\circ}
\]
exhibits $\mathcal{M}_{\circ}$ as a mode. Moreover, for every $\mathcal{M}$-module
$\mathcal{C},$ the $\infty$-category $\mathcal{M}_{\circ}\otimes\mathcal{C}$
can be canonically identified with the full subcategory of $\mathcal{M}_{\circ}$-local
objects in $\mathcal{C}$. In particular, $\mathcal{M}_{\circ}$ classifies
the property of being an $\mathcal{M}_{\circ}$-local $\mathcal{M}$-module. 
\end{prop}

\begin{proof}
By \cite[Proposition 2.2.1.9]{ha}, the $\infty$-category $\mathcal{M}_{\circ}$ admits
a canonical symmetric monoidal structure, such that the left adjoint
$\mathcal{M}\oto L\mathcal{M}_{\circ}$ promotes to a symmetric monoidal
functor. In particular, the unit of this symmetric monoidal structure
$\mathcal{S}\oto{u_{\circ}}\mathcal{M}_{\circ}$ is the composition
$$\mathcal{S}\oto u\mathcal{M}\oto L\mathcal{M}_{\circ},$$
where $u$
is the unit of the symmetric monoidal structure of $\mathcal{M}$.
We need to show that $u_{\circ}\otimes\mathcal{M}_{\circ}$ is an
equivalence, or equivalently, that its right adjoint $G_{\circ}$ is
an equivalence. Let $\mathcal{M}_{\circ}\otimes\mathcal{M}_{\circ}\oto{m_{\circ}}\mathcal{M}_{\circ}$
be the tensor product functor. Since the composition 
\[
\mathcal{S}\otimes\mathcal{M}_{\circ}\oto{u_{\circ}\otimes\mathcal{M}_{\circ}}\mathcal{M}_{\circ}\otimes\mathcal{M}_{\circ}\oto{m_{\circ}}\mathcal{M}_{\circ}
\]
is an equivalence, so is the composition of the right adjoints. It
follows that $G_{\circ}$ is essentially surjective. To complete the proof, we shall
show that $G_{\circ}$ is also fully faithful. Write $u_{\circ}\otimes\mathcal{M}_{\circ}$
as the composition 
\[
\mathcal{S}\otimes\mathcal{M}_{\circ}\oto{u\otimes\mathcal{M}_{\circ}}\mathcal{M}\otimes\mathcal{M}_{\circ}\oto{L\otimes\mathcal{M}_{\circ}}\mathcal{M}_{\circ}\otimes\mathcal{M}_{\circ}.
\]
Let $G$ be the right adjoint of $u\otimes\mathcal{M}_{\circ}$. It
follows that $G_{\circ}$ is the composition of the right adjoint
of $L\otimes\mathcal{M}_{\circ}$ and $G$. Since $L$ admits a fully
faithful right adjoint, by \lemref{FF_Right_Adj}, the functor $L\otimes\mathcal{M}_{\circ}$
has a fully faithful right adjoint as well. Thus, it suffices to show
that $G$ is fully faithful. For this, consider the commutative diagram
\[
\xymatrix{\mathcal{S}\otimes\mathcal{M}\ar[d]_{u\otimes1}^{\wr}\ar[rr]^{1\otimes L} &  & \mathcal{S}\otimes\mathcal{M}_{\circ}\ar[d]^{u\otimes1}\\
\mathcal{M}\otimes\mathcal{M}\ar[rr]^{1\otimes L} &  & \mathcal{M}\otimes\mathcal{M}_{\circ}.
}
\]
Taking the right adjoints, we see that the composition of $G$ with
the right adjoint of $\mathcal{M}\otimes L$, which is fully faithful
by \lemref{FF_Right_Adj}, is fully faithful. It follows that $G$
must be fully faithful as well. This concludes the proof that $u_{\circ}$
exhibits $\mathcal{M}$ as a mode.

Now, we want to analyze the property classified by $\mathcal{M}_{\circ}$.
Since $L$ is symmetric monoidal, $\mathcal{M}_{\circ}$ is a commutative
algebra over $\mathcal{M}$. Thus, every $\mathcal{M}_{\circ}$-module
is, in particular, an $\mathcal{M}$-module. Given an $\mathcal{M}$-module
$\mathcal{C}$, it is an $\mathcal{M}_{\circ}$-module if and only
if the composition 
\[
\mathcal{C}\otimes\mathcal{S}\oto{\mathcal{C}\otimes u}\mathcal{C}\otimes\mathcal{M}\oto{\mathcal{C}\otimes L}\mathcal{C}\otimes\mathcal{M}_{\circ}
\]
is an equivalence. The first functor is an equivalence since $\mathcal{C}$
is an $\mathcal{M}$-module. Thus, by 2-out-of-3, the composition
is an equivalence if and only if $\mathcal{C}\otimes L$ is an equivalence.
The functor $\mathcal{C}\otimes L$ admits by \lemref{FF_Right_Adj}
a fully faithful right adjoint. To describe its essential image, we
consider the commutative diagram of Yoneda embeddings
\[
\xymatrix{\mathcal{C}\otimes\mathcal{M}_{\circ}\ar[d]^{\wr}\ar[rr] &  & \mathcal{C}\otimes\mathcal{M}\ar[d]^{\wr}\\
\fun^{R}(\mathcal{C}^{\op},\mathcal{M}_{\circ})\ar[d]\ar@{^{(}->}[rr] &  & \fun^{R}(\mathcal{C}^{\op},\mathcal{M})\ar[d]\\
\fun(\mathcal{C}^{\op},\mathcal{M}_{\circ})\ar@{^{(}->}[rr] &  & \fun(\mathcal{C}^{\op},\mathcal{M}).
}
\]
We see that $\mathcal{C}\otimes\mathcal{M}_{\circ}$ is identified
with the full subcategory of $\mathcal{M}_{\circ}$-local objects
in $\mathcal{C}$.
\end{proof}
\begin{rem}
\label{rem:Mode_Reflective}\propref{Mode_Localization} need not
hold for a reflective subcategory $\mathcal{M}_{\circ}\ss\mathcal{M}$,
which is not assumed in advance to be compatible with the symmetric
monoidal structure. Indeed, the inclusion of \emph{co-connective}
spectra $\Sp^{\mathrm{co-cn}}\ss\Sp$ is reflective with a left adjoint
$\tau_{\le0}$. However, we have 
\[
\Sp^{\mathrm{co-cn}}\otimes\Sp^{\mathrm{co-cn}}\simeq0.
\]
Indeed, $\Sp^{\mathrm{co-cn}}\otimes\Sp$ is the $\infty$-category
of spectrum objects in $\Sp^{\mathrm{co-cn}}$, and thus is zero.
But by \lemref{FF_Right_Adj}, the functor 
\[
0=\Sp^{\mathrm{co-cn}}\otimes\Sp\oto{\Id\otimes\tau_{\le0}}\Sp^{\mathrm{co-cn}}\otimes\Sp^{\mathrm{co-cn}}
\]
admits a fully faithful right adjoint, so $\Sp^{\mathrm{co-cn}}\otimes\Sp^{\mathrm{co-cn}}$
must be zero as well. 
\end{rem}

Many further examples of modes can be constructed using \propref{Mode_Localization}:
\begin{example}
Consider the subcategory $\mathcal{S}_{\le d}\ss\mathcal{S}$ of $d$-truncated
spaces. Every $\mathcal{C}\in\Pr$ is an $\mathcal{S}$-module and
an object $X\in\mathcal{C}$ is $\mathcal{S}_{\le d}$-local if and
only if it is $d$-truncated. Thus, $\mathcal{S}_{\le d}$ is a mode
which classifies the property that every object is $d$-truncated.
Namely, the property of being equivalent to a $(d+1)$-category (compare
with \remref{Mode_Reflective}). The cases $d=-2,$ $-1$ and $0$,
reproduce the terminal mode $0$, the boolean mode $(0\to1)$, and
the discrete mode $\set$ respectively. 
\end{example}

Another important family of examples is the Bousfield localizations: 
\begin{example}
\label{exa:Bousfield_Modes}For every $E\in\Sp,$ the full subcategory
$\Sp_{E}\ss\Sp$ of $E$-local spectra is a mode and we have that
$\Sp_{E_{1}}\simeq\Sp_{E_{2}}$ in $\calg(\Pr)$, if and only if $E_{1}$
and $E_{2}$ are Bousfield equivalent. For every $\mathcal{C}\in\Pr_{\st}$,
we write $\mathcal{C}_{E}\coloneqq\Sp_{E}\otimes\,\mathcal{C}$ and
$\mathcal{C}_{(p)}\coloneqq\Sp_{(p)}\otimes\,\mathcal{C}$ for a prime
$p$.
\end{example}

\begin{rem}
Given $E_{1},E_{2}\in\Sp$, if $\Sp_{E_{1}}\otimes\Sp_{E_{2}}=0$,
then $\Sp_{E_{1}}\times\Sp_{E_{2}}$ is a mode (\propref{Mode_Lattice}(7)).
However, it is usually not a localization of $\Sp$, and in particular, it is not the same as $\Sp_{E_{1}\oplus E_{2}}$. For example, for every $n\ge1$ we have 
\[
L_{n}^{f}\Sp\simeq\Sp_{\oplus_{k=0}^{n}T(k)}\not\simeq\prod_{k=0}^{n}\Sp_{T(k)}.
\]
Note that the right-hand side is $\infty$-semiadditive, while the
left-hand side is not even $1$-semiadditive. In \thmref{1Sad_Decomposition},
we shall show that in a sense, this is \emph{the }difference between
the left and right-hand sides.
\end{rem}

As with Bousfield localization, a particularly nice kind of localizations
of modes, is given by the ones which are \emph{smashing} in the sense
of \defref{Smashing}. The smashing localizations of modes have a very simple characterization:
\begin{prop}
\label{prop:Mode_Smashing_Localization}A localization of modes $L\colon \mathcal{M}\to \mathcal{N}$
is smashing, if and only if the (fully faithful) right adjoint of $L$
admits a further right adjoint.
\end{prop}

\begin{proof}
In one direction, the forgetful functor $\Mod_{R}(\mathcal{M})\to\mathcal{M}$
admits a right adjoint for every $R\in\calg(\mathcal{M})$ by \cite[Remark 4.2.3.8]{ha}.
Conversely, by \propref{Idemp_Proj_Formula}, it suffices to show
that if the right adjoint $F$ of $L$ admits a further right adjoint,
the adjunction $F\dashv L$ satisfies the projection formula. Since
$F$ is then colimit preserving, the natural transformation in the
projection formula
\[
X\otimes F(Y)\oto{\alpha}F(L(X)\otimes Y)
\]
is a natural transformation between two functors $\mathcal{M}\times\mathcal{N}\to\mathcal{M}$,
which are colimit preserving in each variable. Equivalently, these
are colimit preserving functors $\mathcal{M}\otimes\mathcal{N}\to\mathcal{M}$ \cite[Section 4.8.1]{ha}.
Thus, it suffices to check that $\alpha$ becomes an isomorphism after
whiskering along the equivalence $$\mathcal{N}\simeq\mathcal{S}\otimes\mathcal{N}\iso\mathcal{M}\otimes\mathcal{N}.$$
This amounts to verifying the case $X=\one_{\mathcal{M}}$, in which,
by unwinding the definitions, $\alpha$ is the identity and so in
particular an isomorphism.
\end{proof}
\begin{rem}
Every mode $\mathcal{M}$ provides by definition a smashing localization
$\Pr\to\Mod_{\mathcal{M}}(\Pr)$. Furthermore, every map of modes
$\mathcal{M}\to\mathcal{N}$ induces a smashing localization of the
$\infty$-categories of modules $\Mod_{\mathcal{M}}(\Pr)\to\Mod_{\mathcal{N}}(\Pr)$.
\propref{Mode_Smashing_Localization} characterizes those smashing
localizations of $\Mod_{\mathcal{M}}(\Pr)$, which arise from smashing
localizations of $\mathcal{M}.$
\end{rem}

A particular instance of mode localizations arises from 
divisible and complete objects with respect to a natural endomorphism
of the identity functor. More generally,
\begin{prop}
\label{prop:Div_Comp_Mode}Let $\mathcal{C}\in\Pr$ and $\Id_{\mathcal{C}}\oto{\alpha, }\Id_{\mathcal{C}}$.
The $\infty$-categories $\mathcal{C}[\alpha^{-1}]$ and $\widehat{\mathcal{C}}_{\alpha}$
are accessible localizations of $\mathcal{C}$ and hence in particular
presentable. If moreover $\mathcal{C}\in\calg(\Pr)$,  and $\alpha$
is given by tensoring with $\one\oto{\alpha_{\one}}\one$, then the
full subcategories $\mathcal{C}[\alpha^{-1}],\widehat{\mathcal{C}}_{\alpha}\ss\mathcal{C}$
are compatible with the symmetric monoidal structure of $\mathcal{C}$
and are thus symmetric monoidal localizations of $\mathcal{C}$. 
\end{prop}

\begin{proof}
To show that $\mathcal{C}[\alpha^{-1}]$ and $\widehat{\mathcal{C}}_{\alpha}$
are accessible localizations of $\mathcal{C}$, we use \cite[Propostion 5.5.4.15]{htt},
by which it suffices to realize them as the full subcategories of
$S$-local objects with respect to a suitable (small) set of morphisms
in $\mathcal{C}$. Since $\mathcal{C}$ is presentable, it is $\kappa$-compactly
generated for some cardinal $\kappa$. In particular, the subcategory
$\mathcal{C}^{\kappa}\ss\mathcal{C}$ of $\kappa$-compact objects
is essentially small and generates $\mathcal{C}$ under colimits.
For $\mathcal{C}[\alpha^{-1}]$, we take $S$ to be the collection
of maps $X\oto{\alpha}X$ for $X$ in (a set of representatives of) $\mathcal{C}^{\kappa}$. For $\widehat{\mathcal{C}}_{\alpha}$,
we can take $S^{\prime}$ to be the collection of maps $0\to Z$ for
$Z$ in (a set of representatives of) $\tau$-compact objects for
$\tau$ large enough so that $\mathcal{C}[\alpha^{-1}]$ is $\tau$-compactly
generated and the inclusion $\mathcal{C}[\alpha^{-1}]\into\mathcal{C}$
is $\tau$-accessible.

For $\mathcal{C}\in\calg(\Pr)$, the assumption on $\alpha$
implies that for all $X,Y\in\mathcal{C}$ we have $\alpha_{X\otimes Y}=X\otimes\alpha_{Y}$
and similarly for the adjoint $\alpha_{\hom(X,Y)}=\hom(X,\alpha_{Y})$.
In particular, the class of $\alpha$-divisible objects is closed
under tensoring and exponentiation by any object of $\mathcal{C}$.
Now, the class of morphisms $\overline{S}$ in $\mathcal{C}$, which
are mapped to isomorphisms under the localization $\mathcal{C}\to\mathcal{C}[\alpha^{-1}]$,
is the set of maps $X\oto fY$ in $\mathcal{C}$ such that for every
$W\in\mathcal{C}[\alpha^{-1}]$, the map 
\[
\map(Y,W)\oto{(-)\circ f}\map(X,W)
\]
is an isomorphism. For any $Z\in\mathcal{C}$, we have that $\hom(Z,W)\in\mathcal{C}[\alpha^{-1}]$.
Thus, by adjointness, $Z\otimes f\in\overline{S}$. The argument for
$\widehat{\mathcal{C}}_{\alpha}$ is similar but simpler. It again
suffices to show that for $Z\in\mathcal{C}$ and $W\in\widehat{\mathcal{C}}_{\alpha}$,
the object $\hom^{\mathcal{C}}(W,Z)$ is $\alpha$-complete. This
follows from the fact that for any $\alpha$-divisible $X$, the object
$W\otimes X$ is $\alpha$-divisible. 
\end{proof}

\subsection{Modes of Semiadditivity}

In this subsection, we apply the general theory of modes to study the interaction
of stability and higher semiadditivity. In particular, we introduce
and study the mode which classifies the property of being stable,
$\infty$-semiadditive and of semiadditive height $n$, and compare
it with $\Sp_{T(n)}$. 
\subsubsection{Semiadditivity \& stability}

It is a fundamental result of \cite{Harpaz}, that higher semiadditivity
is classified by a mode. More precisely,  by \lemref{CMon_Pr} the forgetful functor $$\CMon_{m} \to \CMon_{-2} \simeq\mathcal{S}$$
admits a left adjoint $\one_{\Pr} = \mathcal{S} \to \CMon_{m}$. 
We consider $\CMon_{m}$ as an object of $\Pr_{\one/}$ via this left adjoint. 

\begin{prop}
\label{prop:CMon_Tensor} Let $-2 \leq m \leq \infty$. $\CMon_{m}$ is a mode, which classifies
$m$-semiadditivity. Moreover, for every $\mathcal{C}\in\Pr$, there
is a canonical equivalence\footnote{This can be compared with the fact that the stabilization $\Sp\otimes\,\mathcal{C}$
can be identified with $\Sp(\mathcal{C})$, the $\infty$-category
of spectrum objects in $\mathcal{C}$.}
\[
\CMon_{m}\otimes\,\mathcal{C}\simeq\CMon_{m}(\mathcal{C}).
\]
\end{prop}
\begin{proof}
We first treat the case $m <\infty $. The fact that $\CMon_{m}$ is a mode, which classifies $m$-semiadditivity is exactly \cite[Corollary 5.21]{Harpaz}. 
Consider now the inclusion $\iota\colon \Pr^{\sad m} \hookrightarrow \Pr$. 
To prove that    
$\CMon_{m}\otimes\,\mathcal{C}\simeq\CMon_{m}(\mathcal{C})$,
note that  \cite[Corollary 5.21]{Harpaz} and  \cite[Corollary 5.18]{Harpaz} identify  the left-hand side and the right-hand side respectively as left adjoints to $\iota$.  

We now consider the case $m=\infty$. For verious $k$, the left adjoints of the forgetful functors $\CMon_{k+1} \to \CMon_{k}$ can be considered as maps in $\calg^{\mathrm{idem}}(\Pr)\subseteq \Pr_{\one/}$.  Thus, by \lemref{CMon_Pr} and  \propref{Mode_Lattice}(5), we have that
\[
\CMon_{\infty} \simeq \colim_k \CMon_{k}\quad \in \quad \calg^{\mathrm{idem}}(\Pr)
\]
is a mode, classifying the property of being $k$-semiadditive for every $k$. In other words, $\CMon_{\infty}$ is a mode classifying $\infty$-semiadditivity.   Finally, since $\Pr$ is closed symmetric monoidal, for every $\mathcal{C}\in\Pr$ we have

\[
\CMon_{\infty} \otimes\, \mathcal{C} \simeq (\colim_k \CMon_{k} )\otimes \mathcal{C} \simeq \colim_k (\CMon_{k}\otimes\, \mathcal{C}  ) \simeq \colim_k \CMon_{k}(\mathcal{C})   \simeq 
\CMon_{\infty}(C).
\]
\end{proof}
\begin{rem}
The mere fact that $m$-semiadditivity is classified by a mode has already
non-trivial implications. For example, given $\mathcal{C}\oto F\mathcal{D}$
in $\alg(\Pr)$, we get by \propref{Mode_Hereditary}, that if $\mathcal{C}$
is $m$-semiadditive then so is $\mathcal{D}$ (compare \cite[Corollary  3.3.2(2)]{Ambi2018}).
\end{rem}

To study the interaction of higher semiadditivity and stability we
introduce the stable $m$-semiadditive mode:
\begin{defn}
For every $0\le m\le\infty$, we define the $\infty$-category of
$m$-commutative monoids in spectra:
\[
\tsadi^{[m]}\coloneqq\CMon_{m}(\Sp).
\]
We denote $\tsadi^{[\infty]}$ simply by $\tsadi$ (and observe that
$\tsadi^{[0]}\simeq\Sp$). 
\end{defn}

It is an immediate consequence of \propref{CMon_Tensor} and \propref{Mode_Lattice}(4),
that for all values $0\le m\le\infty$, the $\infty$-category $\tsadi^{[m]}$
is a mode which classifies the property of being $m$-semiadditive and
stable. 

The mode $\tsadi^{[m]}$ plays an analogous role to $\CMon_{m}$ for
\emph{stable} $m$-semiadditive $\infty$-categories (even for the
non-presentable ones). In particular, the $\CMon_{m}$-enriched Yoneda
embedding of an $m$-semiadditive $\infty$-category can be further
lifted to $\tsadi^{[m]}$, provided that the $\infty$-category is
stable. To see this, for a stable $\infty$-category $\mathcal{C}$,
we already have a limit preserving spectral Yoneda embedding (e.g. as can be deduced from \cite[Corollary 1.4.2.23]{ha})
$$\Yo^{\Sp}\colon \mathcal{C}\into\fun(\mathcal{C}^{\op},\Sp).$$
By \propref{CMon_Universal_Property},
if $\mathcal{C}$ is in addition $m$-semiadditive, we have a canonical
equivalence 
\[
\fun^{\fin m}(\mathcal{C}^{\op},\CMon_{m}(\Sp))\simeq\fun^{\fin m}(\mathcal{C}^{\op},\Sp)).
\]
Thus, we get a $\tsadi^{[m]}$-enriched Yoneda embedding functor
\[
\Yo^{\tsadi^{[m]}}\colon\mathcal{C}\to\fun^{\fin m}(\mathcal{C}^{\op},\tsadi^{[m]})\ss\fun(\mathcal{C}^{\op},\tsadi^{[m]}).
\]
We note that this functor is fully faithful, exact and $m$-semiadditive.
Using the $\tsadi^{[m]}$-enriched Yoneda embedding, we can characterize
the semiadditive height of an object.
\begin{prop}
\label{prop:Yoneda_Height}Let ${\cal C}\in\acat_{\st}^{\sad m}$
and let $X\in{\cal C}$. For every $0\le n\le m,$ the object $X$
is of height $\le n$ (resp. $>n-1$) if and only if $\hom^{\tsadi^{[m]}}(Z,X)\in\tsadi^{[m]}$
is of height $\le n$ (resp. $>n-1$), for every object $Z\in\mathcal{C}$.
\end{prop}

\begin{proof}
Using the $\tsadi^{[m]}$-enriched Yoneda embedding for $\mathcal{C}^{\op}$,
the functors 
\[
\hom^{\tsadi^{[m]}}(Z,-)\colon\mathcal{C}\to\tsadi^{[m]}\quad,\quad Z\in\mathcal{C}
\]
are $m$-semiadditive and jointly conservative. Thus, by \corref{Stable_Height_Functoriality},
$X$ is of height $\le n$ or $>n-1$ if and only if $\hom^{\tsadi^{[m]}}(Z,X)$
is so for all $Z\in\mathcal{C}$.
\end{proof}
Another useful application of the $\tsadi^{[m]}$-enriched Yoneda
embedding was already used in the proof of \thmref{Height_Decomposition}:
\begin{prop}
\label{prop:Tsadi_Yoneda}For every $\mathcal{C}\in\acat_{\st}^{\sad m}$,
there exists $\widehat{\mathcal{C}}\in\Pr_{\st}^{\sad m}$ and a fully
faithful, $m$-semiadditive and exact embedding $\mathcal{C}\into\widehat{\mathcal{C}}$.
\end{prop}

\begin{proof}
We take 
$$\Yo^{\tsadi^{[m]}}\colon\mathcal{C}\into\fun(\mathcal{C}^{\op},\tsadi^{[m]}) \eqqcolon \widehat{\mathcal{C}},$$
which satisfies all the required properties.
\end{proof}

\subsubsection{Modes of semiadditive height}

We can further concentrate on stable higher semiadditive $\infty$-categories
of particular semiadditive height. We shall now show that this property
is also classified by a mode. 
\begin{thm}
\label{thm:Tsadi_n_Mode}For every prime $p$ and $0\le n\le\infty$,
there exists a mode $\tsadi_{n}$ which classifies the property of
being stable $p$-local $\infty$-semiadditive of height $n$\footnote{We keep $p$ implicit in the notation $\tsadi_{n}$, by analogy with
$\Sp_{T(n)}$ and $\Sp_{K(n)}$.}. 
Moreover, $\tsadi_{n}$ can be canonically identified with $\tsadi_{(p),n}^{[m]}$
for every $n\le m\le\infty$.
\end{thm}

\begin{proof}
We first consider the case $n<\infty$. For every $m\ge n$, we have
that 
\[
\tsadi_{(p),n}^{[m]}=\widehat{(\tsadi_{(p)}^{[m]})}_{\pn{n-1}}[\pn n^{-1}]
\]
is a symmetric monoidal localization of $\tsadi_{(p)}^{[m]}$ (\propref{Div_Comp_Mode}).
By \propref{Yoneda_Height}, for every $p$-local $\mathcal{C}\in\Pr_{\st}^{\sad m}$,
an object $X\in\mathcal{C}$ is $\tsadi_{(p),n}^{[m]}$-local, if
and only if $\htt(X)=n$. Hence, we can apply \propref{Mode_Localization}
to deduce that $\tsadi_{(p),n}^{[m]}$ is itself a mode classifying
the property of being stable $p$-local $m$-semiadditive and of height
$n$. Finally, by \thmref{Height_Everything_p_Local}, every such
$\infty$-category is $\infty$-semiadditive and hence we can take
$\tsadi_{n}$ to be $\tsadi_{(p),n}^{[m]}$. 

For $n=\infty$, we have
\[
\tsadi_{(p),\infty}=\bigcap_{n\in\bb N}\tsadi_{(p),>n}.
\]
Each $\tsadi_{(p),>n}=(\widehat{\tsadi_{(p)}})_{\pn n}$ is a symmetric
monoidal localization of $\tsadi_{(p)}^{[m]}$ (\propref{Div_Comp_Mode}).
Thus, $\tsadi_{(p),\infty}$ is also an accessible reflective subcategory
of $\tsadi_{(p)}$ (\cite[Proposition 5.4.7.10]{htt}), which is compatible with the symmetric
monoidal structure. It follows from \propref{Mode_Localization},
that $\tsadi_{(p),\infty}$ is a mode. Moreover, $\tsadi_{(p),\infty}$
classifies the property of being stable $p$-local $\infty$-semiadditive
and of height $>n$ for all $n$, which is the same as being of height $\infty$. 
\end{proof}
\begin{example}
\label{exa:Tsadi0}In the case $n=0$, we have 
\[
\tsadi_{0}\simeq\tsadi_{(p),0}^{[0]}\simeq\Sp_{(p),0}\simeq\Sp_{\bb Q}.
\]
\end{example}

As with the $\Sp_{T(n)}$-s, the modes $\tsadi_{n}$ are disjoint for different $n$-s.
\begin{prop}
\label{prop:Tasdi_n_Disjoint}For all $n\neq k$, we have $\tsadi_{n}\otimes\tsadi_{k}=0$.
\end{prop}

\begin{proof}
The $\infty$-category $\tsadi_{n}\otimes\tsadi_{k}$ is $\infty$-semiadditive
in which every object is of height both $n$ and $k$, and hence must
be the zero object.
\end{proof}
Another aspect in which $\tsadi_{n}$ resembles $\Sp_{T(n)}$, is
that it kills all bounded above spectra:
\begin{prop}
\label{prop:Tsadi_Bounded_Vanishing}For $1\le n < \infty$, the map of modes
$\Sp\to\tsadi_{n}$ vanishes on all bounded above spectra. 
\end{prop}

\begin{proof}
Denote by $F\colon\Sp\to\tsadi_{n}$ the map of modes. The class of
spectra on which $F$ vanishes is closed under colimits and desuspensions in $\Sp$.
Hence, by a standard devissage argument, it suffices to show that
$F$ vanishes on $\bb Q$ and $\bb F_{\ell}$ for all primes $\ell$.
First, $\bb Q$ and $\bb F_{\ell}$ for $\ell\neq p$ are $p$-divisible.
Since $F$ is $0$-semiadditive, $F(\bb Q)$ and $F(\bb F_{\ell})$
are $p$-divisible as well, but all objects of $\tsadi_{n}$ are $p$-complete,
and so $F(\bb Q)=F(\bb F_{\ell})=0$. Thus, it remains to show that $F(\mathbb{F}_p)=0$. For every $k\in \mathbb{N}$ we denote
by $\overline{\bb S}[B^{k}C_{p}]$ the fiber of the fold map $\bb S[B^{k}C_{p}]\to\bb S.$
Applying $F$, we get a fiber sequence 
\[
F(\overline{\bb S}[B^{k}C_{p}])\to\one_{\tsadi_{n}}[B^{k}C_{p}]\to\one_{\tsadi_{n}}.
\]
Since $\tsadi_{n}$ is by definition of height $n$, it follows from
\propref{Height_Everything}(2), that the second map above is an isomorphism
for $k\ge n+1$ and hence $F(\overline{\bb S}[B^{k}C_{p}])=0$. We
observe that $\bb F_{p}$ can be written as a filtered colimit $\bb F_{p}=\colim\Sigma^{-k}\overline{\bb S}[B^{k}C_{p}]$.
Thus, we also get 
\[
F(\bb F_{p})=\colim\Sigma^{-k}F(\overline{\bb S}[B^{k}C_{p}])=0.
\]
\end{proof}

\begin{cor}
\label{cor:Tsadi_Conservative}For $1\le n < \infty$, the right adjoint of
the map of modes $\mathcal{S}\to\tsadi_{n}$ is conservative.
\end{cor}

\begin{proof}
The map of modes $\mathcal{S}\to\tsadi_{n}$ is given by the composition
\[
\mathcal{S}\oto{\bb S[-]}\Sp\oto{F_{1}}\Sp_{(p)}\oto{F_{2}}\tsadi_{(p)}^{[n]}\oto{F_{3}}\tsadi_{n}.
\]
The right adjoints $G_{1}$ and $G_{3}$ of $F_{1}$ and $F_{3}$
respectively, are fully faithful embeddings and hence in particular
conservative. The right adjoint $G_{2}$ of $F_{2}$ can be identified
with the functor
\[
\tsadi_{(p)}^{[n]}=\CMon_{n}(\Sp_{(p)})\to\Sp_{(p)}
\]
which evaluates at the point. Thus, $G_{2}$ is conservative by  \cite[Lemma 5.17]{Harpaz}.
Combining the above, the right adjoint of the functor 
\[
F=F_{3}F_{2}F_{1}\colon\Sp\to\tsadi_{n}
\]
is given by $G=G_{1}G_{2}G_{3}$, and is therefore conservative.

Now, the right adjoint of $\mathcal{S}\to\tsadi_{n}$, is given by
the composition of the right adjoints $\tsadi_{n}\oto G\Sp\oto{\Omega^{\infty}}\mathcal{S}$.
Let $X\oto fY$ be a map in $\tsadi_{n}$ with fiber $Z$, such that
$\Omega^{\infty}G(f)$ is an isomorphism. It follows that $\Omega^{\infty}G(Z)=0$
and hence $G(Z)$ is co-connective. By \propref{Tsadi_Bounded_Vanishing},
we get $FG(Z)=0$ and hence $GFG(Z)=0$. By the zig-zag identities,
$G(Z)$ is a retract of $GFG(Z)$ and hence we also get $G(Z)=0$.
Finally, since $G$ is conservative, we get $Z=0$ and hence $X\oto fY$
is an isomorphism. This concludes the proof that $\Omega^{\infty}G$,
the right adjoint of $\mathcal{S}\to\tsadi_{n}$, is conservative.
\end{proof}

\begin{rem}
\conjref{Finitness} is equivalent to the statement that $\tsadi_{\infty}=0$.

\end{rem}

\subsection{1-Semiadditive Decomposition}

As we recalled in \exaref{Chromatic_Recollement}, while for $n\ge1$ the $\infty$-category
$L_{n}^{f}\Sp$ itself is not even $1$-semiadditive,
it is an iterated recollement of the $\infty$-categories
$\Sp_{T(k)}$ for $k=0,\dots,n$, which are  $\infty$-semiadditive. The theory of modes allows us to
\emph{enforce} $m$-semiadditivity on $L_{n}^{f}\Sp$ in a universal
way, by tensoring it with the $m$-semiadditive mode $\CMon_{m}$.
We shall show that enforcing even $1$-semiadditivity on $L_{n}^{f}\Sp$,
``dissolves the glue'' which holds the monochromatic layers together,
and decomposes it into a product of $\Sp_{T(k)}$ for $k=0,\dots,n$. 
\begin{rem}
The fact that $L_{n}^{f}\Sp$ is a recollement of $\Sp_{T(k)}$ for
$k=0,\dots,n$, is closely related to the fact that for $k_0<k_1$, every
exact functor $\Sp_{T(k_0)}\to\Sp_{T(k_1)}$ must be zero. The fact that
the recollement becomes \emph{split} after imposing $1$-semiadditivity
follows from the fact that every $1$-semiadditive functor
$\Sp_{T(k_1)}\to\Sp_{T(k_0)}$ must be zero (\corref{Tn_One_Semiadditivity}).
Note that by \corref{Stable_Height_Functoriality}, a $k_1$-semiadditive
functor $\Sp_{T(k_1)}\to\Sp_{T(k_0)}$ must be zero because every object
of $\Sp_{T(k_1)}$ is of height $k_1$, and thus must be sent to the only
object of height $k_1$ in $\Sp_{T(k_0)}$, which is zero. The main ``non-formal''
ingredient in the proof of \thmref{1Sad_Decomposition} (which makes
essential use of the theory developed in \cite[Section 4]{Ambi2018}) is that it suffices to assume merely $1$-semiadditivity. 
\end{rem}

\subsubsection{Divisible and complete cardinalities}

We begin with some general observations and constructions. Every $\mathcal{C}\in\Pr_{\st}^{\sad 1}$
is a module over $\tsadi^{[1]}$ and thus, every $\alpha\in\pi_{0}(\one_{\tsadi^{[1]}})$
induces a natural transformation $\Id_{\mathcal{C}}\oto{\alpha_{\mathcal{C}}}\Id_{\mathcal{C}}$.
If in addition $\mathcal{C}$ is presentably symmetric monoidal, then
there is a unique functor $\tsadi^{[1]}\to\mathcal{C}$ in $\calg(\Pr)$
(see \propref{Idemp_Poset}), which induces a map $\pi_{0}(\one_{\tsadi^{[1]}})\to\pi_{0}(\one_{\mathcal{C}})$.
The natural transformation $\alpha_{\mathcal{C}}$ can be identified
with the one induced by the image of $\alpha$ in $\pi_{0}(\one_{\mathcal{C}})$.
We now restrict to a particular subset of $\pi_{0}(\one_{\tsadi^{[1]}})$
consisting of elements which have a simple description, as they are
of an unstable origin.
\begin{defn}
\label{def:R1}Let $\mathcal{R}_{1}\ss\pi_{0}(\one_{\tsadi^{[1]}})$
be the $\bb Z$-linear span of elements of the form $|BG|\in\pi_{0}(\one_{\tsadi^{[1]}})$,
for $G$ a finite group.
\end{defn}

We observe that the action of $\mathcal{R}_{1}$ on the objects of any stable $1$-semiadditive
$\infty$-category $\mathcal{C}$ is natural with respect to $1$-semiadditive
functors.
\begin{prop}
\label{prop:R1_Linearity}Let $F\colon\mathcal{C}\to\mathcal{D}$
be a $1$-semiadditive functor between stable presentable $1$-semiadditive
$\infty$-categories. For every $\alpha\in\mathcal{R}_{1}$, we have
$F(\alpha_{\mathcal{C}})=\alpha_{\mathcal{D}}$.
\end{prop}

\begin{proof}
By $0$-semiadditivity, it suffices to consider the case $\alpha=|BG|$
for $G$ a finite group, which follows from the fact that $F$ is
$1$-semiadditive.
\end{proof}
\begin{rem}
If $F$ is further assumed to be \emph{colimit} \emph{preserving},
then $F(\alpha_{\mathcal{C}})=\alpha_{\mathcal{D}}$ for all elements $\alpha\in\pi_{0}(\one_{\tsadi^{[1]}})$.
However, we shall be interested in applying \propref{R1_Linearity}
to functors $F$, which are not a priori colimit preserving (e.g.
right adjoints). 
\end{rem}

Recall from \cite[Theorem 4.3.2]{Ambi2018}, that for every $1$-semiadditive stable presentably
symmetric monoidal $\infty$-category $\mathcal{C}$, the ring $\pi_{0}(\one_{\mathcal{C}})$
is equipped with a canonical additive $p$-derivation 
\[
\delta\colon\pi_{0}(\one_{\mathcal{C}})\to\pi_{0}(\one_{\mathcal{C}}).
\]

\begin{prop}
\label{prop:R1_Closure}The subset $\mathcal{R}_{1}\ss\pi_{0}(\one_{\tsadi^{[1]}})$
is closed under multiplication and the additive $p$-derivation $\delta$
inside $\pi_{0}(\one_{\tsadi^{[1]}})$. Consequently, $\mathcal{R}_{1}$
is a semi-$\delta$-ring and the inclusion $\mathcal{R}_{1}\into\pi_{0}(\one_{\tsadi^{[1]}})$
is a homomorphism of semi-$\delta$-rings. 
\end{prop}

\begin{proof}
The closure under multiplication follows from the identity (\remref{Cardinality_Arithmetic}) 
\[
|BG||BH|=|B(G\times H)|
\]
and the closure under $\delta$ follows from the identity 
\[
\delta|BG|=|BG\wr C_{p}|-|B(C_{p}\times G)|,
\]
the formula (see \cite[Definition 4.1.1(1)]{Ambi2018})
\[
\delta(x+y)=\delta(x)+\delta(y)+\frac{x^{p}+y^{p}-(x+y)^{p}}{p}
\]
and the closure of $\mathcal{R}_{1}$ under multiplication. 
\end{proof}

\subsubsection{$1$-semiadditive decomposition}

We would now like to apply the above to the $\infty$-categories $\Sp_{T(n)}$.
By \cite[Proposition 4.3.4]{Ambi2018}, the construction of the additive $p$-derivation $\delta$
is functorial with respect to colimit preserving symmetric monoidal
functors. In particular, assuming that $n\ge1$, both maps 
\[
\pi_{0}(\one_{\tsadi^{[1]}})\to\pi_{0}\bb S_{T(n)}\oto u\pi_{0}E_{n}
\]
commute with $\delta$. By \propref{Goerss_Irena}, the image of $u$
is the canonical $\bb Z_{p}$ copy obtained from the map 
\[
\bb Z_{p}\simeq\pi_{0}\widehat{\bb S}_{p}\into\pi_{0}E_{n}.
\]
Since $\bb Z_p \subset \pi_{0}E_{n}$ is the image of a semi-$\delta$-ring map we obtain a surjective map  of semi-$\delta$-rings
\[u\colon \pi_{0}\bb S_{T(n)} \to \bb Z_p .\]
The semi-$\delta$-ring structure on $\bb Z_p$ can be explicitly described.
\begin{lem}
\label{lem:Zp_delta}
The unique semi-$\delta$-ring structure on $\bb Z_p$ is given by 
\[\delta(a) = \frac{a-a^p}{p},\qquad \forall a \in \mathbb{Z}_p.
\]
In particular, if $v_{p}(a)>0$, then $v_{p}(\delta(a)) = v_{p}(a)-1$.
\end{lem}
\begin{proof}
For $a \in \bb Z_p$ we define $\phi(a) =  \frac{a-a^p}{p}-\delta(a)$. By \cite[Definition 4.1.1(1)]{Ambi2018}, the function $\phi\colon \bb Z_p \to \bb Z_p$ is additive. By \cite[Proposition  4.1.11]{Ambi2018}, $\phi$ factors trough a map ${\bb Z_p}/{\bb Z_{(p)}} \to \bb Z_p$. Since ${\bb Z_p}/{\bb Z_{(p)}}$ is $p$-divisible, all such maps are zero.
\end{proof}
For $a\in\pi_{0}\bb S_{T(n)}$, we denote by $v_{p}(a)$ the $p$-adic
valuation of $u(a)\in\bb Z_{p}\ss\pi_{0}E_{n}$. 
\begin{prop}
\label{prop:Inv_Valuation}For every $n\ge1$ and $a\in\pi_{0}\bb S_{T(n)}$
we have:
\begin{enumerate}
\item If $v_{p}(a)=0$, then $\Sp_{T(n)}$ is $a$-divisible.
\item If $v_{p}(a)>0$, then $\Sp_{T(n)}$ is $a$-complete.
\end{enumerate}
\end{prop}

\begin{proof}
For (1), if $v_{p}(a)=0$, then $u(a)\in\pi_{0}E_{n}$ is invertible.
It follows that $a\in\pi_{0}\bb S_{T(n)}$ is invertible by nil-conservativity
\cite[Corollary 5.1.17]{Ambi2018}. For (2), let $a\in\pi_{0}\bb S_{T(n)}$
such that $v_{p}(a)>0$. We need to show that for every $0\neq X\in\Sp_{T(n)}$,
the element $a$ acts non-invertibly on $X$. Since tensoring with
$T(n)$ is conservative, we may replace $X$ by $T(n)\otimes X$.
Without loss of generality, we may choose $T(n)$ to be a ring spectrum.
Thus, $T(n)\otimes X$ is a $T(n)$-module and the action of $a$
on it is by its image under the map $\pi_{0}\bb S_{T(n)}\to\pi_{0}T(n)$,
which we denote by $\overline{a}$. Thus, it would suffice to show
that $\overline{a}\in\pi_{0}T(n)$ is \emph{nilpotent}. By the Nilpotence
Theorem, it suffices to show that the image of $\overline{a}$ under
the map 
\[
\pi_{0}T(n)\to\pi_{0}(T(n)\otimes E_{n})
\]
is nilpotent, as $T(n)\otimes K(j)=0$ for all $j=n+1,\dots,\infty$.
Consider the commutative diagram
\[
\xymatrix{\pi_{0}\bb S_{T(n)}\ar[d]\ar[rr]^{u} &  & \pi_{0}E_{n}\ar[d]\\
\pi_{0}T(n)\ar[rr] &  & \pi_{0}(T(n)\otimes E_{n})
}
\]
We see that the image of $\overline{a}$ in $\pi_{0}(T(n)\otimes E_{n})$
is contained in the image of $u$, which lies in $\bb Z_{p}\ss\pi_{0}E_{n}$
(\propref{Goerss_Irena}). Since $T(n)$ is $p$-power torsion,
the map 
\[
\bb Z_{p}\ss\pi_{0}E_{n}\to\pi_{0}(T(n)\otimes E_{n})
\]
factors through a finite quotient $\bb Z_{p}\onto\bb Z/p^{r}$ and
hence the image of $\overline{a}$ in $\pi_{0}(T(n)\otimes E_{n})$
is nilpotent. 
\end{proof}
\begin{prop}
\label{prop:Div_Comp_R1}For every $k\ge0$, there exists an element
$\alpha\in\mathcal{R}_{1}$, such that $\Sp_{T(n)}$ is $\alpha$-complete
for $n\le k$ and $\alpha$-divisible for $n>k$.
\end{prop}

\begin{proof}
First, we observe that it suffices to construct elements $\beta_{(k)}$
such that $\Sp_{T(k)}$ is $\beta_{(k)}$-complete and $\Sp_{T(n)}$
is $\beta_{(k)}$-divisible for $n>k$. Indeed, we can then define
$\alpha=\beta_{(0)}\cdots\beta_{(k)}$, which satisfies the required
properties. To construct $\beta\coloneqq\beta_{(k)}$ for a specific
$k\ge0$, we proceed as follows. For $k=0$, we take $\beta=p|BC_{p}|-1$.
We get that for $\Sp_{T(0)}=\Sp_{\bb Q}$, the element $\beta$ acts as zero,
while for $n>0$, the $\infty$-category $\Sp_{T(n)}$ is $p$-complete,
and hence by \propref{Semi_Inevrse_Divisible}, it is $\beta$-divisible.
Now, we treat the case $k\ge1$. For every $a\in\mathcal{R}_{1}\ss\pi_{0}(\one_{\tsadi^{[1]}})$,
we shall denote by $a_{n}$ its image in $\bb Z_{p}\ss\pi_{0}E_{n}$
under the map $\pi_{0}(\one_{\tsadi^{[1]}})\to\pi_{0}E_{n}$. By \propref{Inv_Valuation},
it suffices to construct $\beta$, such that $v_{p}(\beta_{k})>0$
and $v_{p}(\beta_{n})=0$ for all $n\ge k+1$. Using \propref{R1_Closure},
we define the element 
\[
\gamma=\delta^{k-1}|BC_{p}|\in\mathcal{R}_{1}.
\]
By \propref{Cardinality_EM}, we have $|BC_{p}|_{n}=p^{n-1}$ for all $n\in \mathbb{N}$. Thus, by \lemref{Zp_delta}, we have 
$v_{p}(\gamma_{n})=n-k$ for all $n\ge k$, and in particular $v_{p}(\gamma_{k})=0$.
It follows that there exists an integer $b\in\bb Z$ with $v_{p}(b)=0$,
such that $v_{p}(\gamma_{k}-b)>0$. We define $\beta\coloneqq\gamma-b\in\mathcal{R}_{1}.$
On the one hand, we have by construction  $v_{p}(\beta_{k})>0$, and on the other, $v_{p}(\beta_{n})=0$ for $n>k$ by the ultrametric property
of the $p$-adic valuation.
\end{proof}
\begin{cor}
\label{cor:Tn_One_Semiadditivity}For all $n>k\ge0$, every $1$-semiadditive
functor $F\colon\Sp_{T(n)}\to\Sp_{T(k)}$ is zero.
\end{cor}

\begin{proof}
By \propref{Div_Comp_R1}, there exists an element $a\in\mathcal{R}_{1}$,
such that $\Sp_{T(n)}$ is $a$-divisible and $\Sp_{T(k)}$ is $a$-complete.
Since $F$ is $1$-semiadditive, by \propref{R1_Linearity}, it must
preserve the action of $a$. It follows that for every $X\in\Sp_{T(n)}$,
the object $F(X)$ is both $a$-divisible and $a$-complete and hence
must be zero. 
\end{proof}
We are now ready to prove that $1$-semiadditivity forces $L_{n}^{f}\Sp$
to decompose into its monochromatic pieces.
\begin{thm}
[1-Semiadditive Decomposition]\label{thm:1Sad_Decomposition}For
every $n\ge0$, there is an equivalence of modes
\[
\CMon_{1}\otimes L_{n}^{f}\Sp\simeq\prod_{k=0}^{n}\Sp_{T(k)}.
\]
\end{thm}

\begin{proof}
For $n=0$, the claim is $\CMon_{1}\otimes\Sp_{\bb Q}\simeq\Sp_{\bb Q},$
which is true because $\Sp_{\bb Q}$ is $1$-semiadditive. It thus
suffices to prove by induction on $n\in\bb N$, that the functor  
\[
L_{n}^{f}\Sp\to L_{n-1}^{f}\Sp\times\Sp_{T(n)},
\]
given by the product of the respective symmetric monoidal localizations, becomes an equivalence after tensoring with
$\CMon_{1}$. We know that $L_{n}^{f}\Sp$ is a recollement of $L_{n-1}^{f}\Sp\ss L_{n}^{f}\Sp$
and $\Sp_{T(n)}$. It follows that $\CMon_{1}\otimes L_{n}^{f}\Sp$
is a recollement of $\CMon_{1}\otimes L_{n-1}^{f}\Sp$ and $\CMon_{1}\otimes\Sp_{T(n)}=\Sp_{T(n)}$
(\propref{Recollement_Base_Change}). By the inductive hypothesis, we have
\[
\CMon_{1}\otimes L_{n-1}^{f}\Sp\simeq\prod_{k=0}^{n-1}\Sp_{T(k)}.
\]
Thus, $\CMon_{1}\otimes L_{n}^{f}\Sp$ is a recollement of $\prod_{k=0}^{n-1}\Sp_{T(k)}$
and $\Sp_{T(n)}$. It is therefore suffices to show that the gluing
data given by the functor 
\[
L\colon\Sp_{T(n)}\to\prod_{k=0}^{n-1}\Sp_{T(k)}
\]
is zero (\propref{Split_Recollement}). We observe that $L$ is given
as a composition of a right and a left adjoint between $1$-semiadditive
$\infty$-categories and hence is $1$-semiadditive. Thus, by \corref{Tn_One_Semiadditivity},
we must have $L=0$ and hence the corresponding recollement is split.
To conclude, the localization functors $L_{n}^{f}\Sp\to\Sp_{T(k)}$
for $k=0,\dots,n$ induce a functor $L_{n}^{f}\Sp\to\prod_{k=0}^{n}\Sp_{T(k)}$,
which becomes an equivalence after tensoring with $\CMon_{1}$. In
particular, this is a symmetric monoidal equivalence and hence an
equivalence of modes.
\end{proof}
\begin{rem}
By tensoring \thmref{1Sad_Decomposition} with $L_{n}\Sp$, we also
get
\[
\CMon_{1}\otimes L_{n}\Sp\simeq\prod_{k=0}^{n}\Sp_{K(k)}.
\]
\end{rem}

\subsection{Semiadditive vs. Stable Height}

As we recalled in the introduction, for every $n\ge0$, the localization
functors of spectra $L_{n}^{f}$, $L_{F(n)}$ and $L_{T(n)}$ can
be thought of as restriction to heights $\le n$, $\ge n$ and $n$
respectively, as measured by the $v_{n}$-self maps. It is natural
to compare this notion of height with the semiadditive one considered
in this paper. In this subsection, we phrase the notion of height classified
by $L_{n}^{f}$, $L_{F(n)}$ and $L_{T(n)}$ (which for disambiguation we shall call \emph{stable height}) in the language of modes
and establish some comparison results with the notion of semiadditive
height. Using that, we shall prove the bounded version of the
``Bootstrap Conjecture'' (\thmref{Intro_Bounded_Bootstrap}), regarding
$1$-semiadditivity vs. $\infty$-semiadditivity for stable presentable
$\infty$-categories.

\subsubsection{Stable Height}

By \exaref{Bousfield_Modes}, the $\infty$-categories $L_{n}^{f}\Sp$,
$\Sp_{F(n)}$ and $\Sp_{T(n)}$ are themselves \emph{modes.} Our first
goal is to show that the properties classified by them can be profitably
reinterpreted in terms of the following notion:
\begin{defn}
\label{def:Stable_Height}Given a stable $\infty$-category $\mathcal{C}$,
for every $X\in\mathcal{C}$, we define and denote the \textbf{stable
height} of $X$ as follows:
\begin{enumerate}
\item $\htt_{\st}(X)\le n$, if $F(n+1)\otimes X=0$ for some (hence any)
finite spectrum $F(n+1)$ of type $n+1$. 
\item $\htt_{\st}(X)>n$, if $\map(Z,X)\simeq\pt$ for every $Z$ of stable
height $\le n$.
\item $\htt_{\st}(X)=n$, if $\htt_{\st}(X)\le n$ and $\htt_{\st}(X)>n-1$.
\end{enumerate}
By convention, $\htt_{\st}(X)>-1$ for all $X$, and $\htt_{\st}(X)\le-1$
if and only if $X=0$. We also extend the definition to $n=\infty$ as follows: For every $X\in\mathcal{C}$, we write $\htt_{\st}(X)=\infty$
if and only if $\htt_{\st}(X)>n$ for all $n$.
\end{defn}

\begin{rem}
Since $F(n+2)$ can be constructed as a cofiber of a self map of $F(n+1)$,
it is clear that $\htt_{\st}(X)\le n$ implies $\htt_{\st}(X)\le n+1$.
Consequently, $\htt_{\st}(X)>n$ also implies $\htt_{\st}(X)>n-1$.
\end{rem}

As with the semiadditive height, it is useful to consider the corresponding
subcategories of objects of stable height in a certain range:
\begin{defn}
\label{def:Stable_Height_Cat}Let $\mathcal{C}\in\acat_{\st}$ and
let $0\le n\le\infty$. We denote by $\mathcal{C}_{\le^{\st}n}$,
$\mathcal{C}_{>^{\st}n}$, and $\mathcal{C}_{n^{\st}}$ the full subcategories
of $\mathcal{C}$ spanned by objects of stable height $\le n$, $>n$,
and $n$. In addition, we write $\Ht_{\st}(\mathcal{C})\le n$, $>n$, or $n$,
if $\mathcal{C}=\mathcal{C}_{\le^{\st}n}$, $\mathcal{C}_{>^{\st}n}$,
or $\mathcal{C}_{n^{\st}}$ respectively. 
\end{defn}

\begin{example}
In the special case $\mathcal{C}=\Sp_{(p)},$ we have by definition
\[
\Sp_{(p),\le^{\st}n}=L_{n}^{f}\Sp,\qquad\Sp_{(p),>^{\st}n-1}=\Sp_{F(n)},\qquad\Sp_{(p),n^{\st}}=\Sp_{T(n)}.
\]
We note that the subcategory $\Sp_{(p),\infty^{\st}}\ss\Sp$ is rather large. First, for every $X\in L_{n}^{f}\Sp$, we have 
\[
L_{n}^{f}X\otimes\bb F_{p}\simeq L_{n}^{f}X\otimes L_{n}^{f}\bb F_{p}\simeq L_{n}^{f}X\otimes0\simeq0.
\]
Therefore, for every $\bb F_{p}$-module spectrum $M$, we have
\[
\map_{\Sp}(L_{n}^{f}X,M)\simeq\map_{\Mod_{\bb F_{p}}(\Sp)}(\bb F_{p}\otimes L_{n}^{f}X,M)\simeq\map_{\Mod_{\bb F_{p}}(\Sp)}(0,M)\simeq0,
\]
which implies that $M\in\Sp_{(p),\infty^{\st}}$. Since the subcategory $\Sp_{(p),\infty^{\st}}\ss\Sp$
is closed under limits, it contains all bounded below $p$-complete
spectra. In contrast, the $\infty$-category $\mathcal{C}=\prod_{n\in\bb N}\Sp_{T(n)}$
satisfies $\mathcal{C}_{\infty^{\st}}=0$, even though it contains
many objects of unbounded height. 
\end{example}

We now show that the modes $L_{n}^{f}\Sp$, $\Sp_{F(n)}$ and $\Sp_{T(n)}$
\emph{classify} the properties of having the corresponding stable
height.
\begin{prop}
\label{prop:Stable_Modes}For every $0\le n<\infty$, the modes $L_{n}^{f}\Sp$,
$\Sp_{F(n)}$ and $\Sp_{T(n)}$ classify the properties of being stable
$p$-local of stable height $\le n$, $>n-1$ and $n$ respectively. 
\end{prop}

\begin{proof}
We begin with $L_{n}^{f}\Sp$. It follows from \propref{Mode_Localization},
that $L_{n}^{f}\Sp$ is a mode, which classifies stable $L_{n}^{f}\Sp$-local
$\infty$-categories. Thus, it suffice to show that an object $X$
in a stable $\infty$-category $\mathcal{C}$ is $L_{n}^{f}\Sp$-local,
if and only if $\htt_{\st}(X)\le n$. By definition, $X$ is $L_{n}^{f}\Sp$-local,
if for every $Z\in\Sp$, the mapping spectrum $\hom_{\mathcal{C}}^{\Sp}(Z,X)$
belongs to $L_{n}^{f}\Sp$. Since the corepresentable functor $\hom_{\mathcal{C}}^{\Sp}(Z,-)\colon\mathcal{C}\to\Sp$
is exact, we have a canonical isomorphism
\[
\hom_{\mathcal{C}}^{\Sp}(Z,F(n+1)\otimes X)\simeq F(n+1)\otimes\hom_{\mathcal{C}}^{\Sp}(Z,X).
\]
It follows that if $\htt_{\st}(X)\le n$, then $X$ is $L_{n}^{f}\Sp$-local.
Since the collection of functors $\hom_{\mathcal{C}}^{\Sp}(Z,-)$
for all $Z\in\mathcal{C}$ is also jointly conservative, the converse
holds as well. 

We now move to $\Sp_{F(n)}$. We first show that if $\htt_{\st}(X)>n-1$,
then $\hom_{\mathcal{C}}^{\Sp}(Z,X)$ is $F(n)$-local for all $Z\in\mathcal{C}$.
For $A$ an $F(n)$-acyclic spectrum, $A\otimes Z$ is $F(n)$-acyclic
as well and hence $\htt_{\st}(A\otimes Z)\le n-1$. Since $\htt_{\st}(X)>n-1$,
it follows that
\[
\map_{\Sp}(A,\hom_{\mathcal{C}}^{\Sp}(Z,X))\simeq\map_{\mathcal{C}}(A\otimes Z,X)\simeq\pt,
\]
and hence $X$ is $\Sp_{F(n)}$-local. Conversely, assume that $X$
is $\Sp_{F(n)}$-local, and let $Z\in\mathcal{C}$ such that $\htt_{\st}(Z)\le n-1$.
We have 
\[
0\simeq\hom_{\mathcal{C}}^{\Sp}(F(n)\otimes Z,X)\simeq\bb DF(n)\otimes\hom_{\mathcal{C}}^{\Sp}(Z,X),
\]

where $\bb DF(n)$ is the Spanier-Whithead dual of $F(n)$, which
is itself of type $n$. Since $\hom_{\mathcal{C}}^{\Sp}(Z,X)$ is $F(n)$-local and hence
$\bb DF(n)$-local, we get $\hom_{\mathcal{C}}^{\Sp}(Z,X)=0$. This
implies that $\map_{\mathcal{C}}(Z,X)=\pt$ and so $\htt_{\st}(X)>n-1$.

Finally, for $\Sp_{T(n)}$ we observe that a spectrum is $T(n)$-local
if and only if it is both $F(n)$-local and $\bigoplus_{k=0}^{n}T(k)$
-local. Thus,
\[
\Sp_{T(n)}\simeq L_{n}^{f}\Sp\otimes\Sp_{F(n)}.
\]
Hence, it classifies the property of being stable $p$-local of both stable
height $\le n$ and $>n-1$, i.e exactly $n$. 
\end{proof}
Similarly, we also treat the case of stable height $\infty$.
\begin{cor}
\label{cor:Height_Inf_Mode}The $\infty$-category $\Sp_{(p),\infty^{\st}}$
is a mode, which classifies the property of being stable of stable
height $\infty$.
\end{cor}

\begin{proof}
We have 
\[
\Sp_{(p),\infty^{\st}}=\bigcap_{n\in\bb N}\Sp_{F(n)}\ss\Sp_{(p)},
\]
and thus is an accessible reflective subcategory of $\Sp_{(p)}$,
which is compatible with the symmetric monoidal structure \cite[Proposition .5.4.7.10]{htt}.
It follows from \propref{Mode_Localization}, that $\Sp_{(p),\infty^{\st}}$
is a mode. Moreover, $\Sp_{(p),\infty^{\st}}$ classifies the property
that the $\infty$-category is stable such that every object $X$ is $\Sp_{F(n)}$-local
for all $n\in\bb N$. By \propref{Stable_Modes}, the said condition
on $X$ is equivalent to $\htt_{\st}(X)\ge n$ for all $n$, which
is the same as $\htt_{\st}(X)=\infty$. Thus, $\Sp_{(p),\infty^{\st}}$
classifies the property of being stable of stable height $\infty$.
\end{proof}
\propref{Stable_Modes} has the following immediate corollaries regarding
stable height for stable presentable $\infty$-categories:
\begin{cor}
\label{cor:Stable_Height_Localization}For every $p$-local $\mathcal{C}\in\Pr_{\st}$,
we have canonical equivalences
\[
\mathcal{C}_{\le^{\st}n}\simeq\mathcal{C}\otimes L_{n}^{f}\Sp,\qquad\mathcal{C}_{>^{\st}n}\simeq\mathcal{C}\otimes\Sp_{F(n+1)},\qquad\mathcal{C}_{n^{\st}}\simeq\mathcal{C}\otimes\Sp_{T(n)}
\]
 for every $n\in\bb N$.
\end{cor}

\begin{proof}
This is a direct consequence of \propref{Stable_Modes} and \propref{Mode_Localization}. 
\end{proof}
\begin{cor}
\label{cor:Stable_Height_Recollement}Every $p$-local ${\cal C}\in\Pr_{\st}$
is a recollement of ${\cal C}_{\le^{\st}n}$ and ${\cal C}_{>^{\st}n}$
for every $n\in\bb N$.
\end{cor}

\begin{proof}
Since $\Sp$ is a recollement of $L_{n}^{f}\Sp$ and $(L_{n}^{f}\Sp)^{\perp}=\Sp_{F(n+1)}$,
we have by \propref{Recollement_Base_Change}, that $\mathcal{C}\simeq{\cal C}\otimes\Sp$
is a recollement of ${\cal C}\otimes L_{n}^{f}\Sp$ and $\mathcal{C}\otimes\Sp_{F(n+1)}$.
By \corref{Stable_Height_Localization}, we get that $\mathcal{C}$
is a recollement of ${\cal C}_{\le^{\st}n}$ and ${\cal C}_{>^{\st}n}$. 
\end{proof}
Let $\mathcal{C}$ be a stable presentable $\infty$-category. For
each $n\in\bb N,$ let $\mathcal{C}\oto{R_{\le n}}\mathcal{C}_{\le^{\st}n}$
be the right adjoint of the inclusion $\mathcal{C}_{\le^{\st}n}\into\mathcal{C}$
and $\mathcal{C}_{\le^{\st}n}\oto{L_{n}}\mathcal{C}_{n^{\st}}$ the
left adjoint of the inclusion $\mathcal{C}_{n^{\st}}\into\mathcal{C}_{\le^{\st}n}$.
We define
\[
P_{n}\colon\mathcal{C}\oto{R_{\le n}}\mathcal{C}_{\le^{\st}n}\oto{L_{n}}\mathcal{C}_{n^{\st}}.
\]

\begin{prop}
\label{prop:Finite_Stable_Height_Conservative}For every $p$-local
$\mathcal{C}\in\Pr_{\st}$ and $X\in\mathcal{C}$. If $P_{n}(X)=0$
for all $n\in\bb N$, then $X\in\mathcal{C}_{\infty^{\st}}$. In particular,
if $\mathcal{C}_{\infty^{\st}}=0$, the collection of functors $P_{n}\colon\mathcal{C}\to\mathcal{C}_{n^{\st}}$
is jointly conservative. 
\end{prop}

\begin{proof}
Let $X\in\mathcal{C}$, such that $P_{n}(X)=0$ for all $n$. We shall
show by induction that $X\in\mathcal{C}_{>^{\st}n}$ for all $n\in\bb N$,
and hence $X\in\mathcal{C}_{\infty^{\st}}$. Assuming by induction
that $X\in\mathcal{C}_{>^{\st}n-1}$ we have
\[
R_{\le n-1}R_{\le n}(X)=R_{\le n-1}(X)=0.
\]
Therefore $R_{\le n}(X)\in\mathcal{C}_{>^{\st}n-1}$ and hence 
\[
R_{\le n}(X)\in\mathcal{C}_{\le{}^{\st}n}\cap\mathcal{C}_{>^{\st}n-1}=\mathcal{C}_{n^{\st}}.
\]

It follows that 
\[
P_{n}(X)=L_{n}R_{\le n}(X)=R_{\le n}(X).
\]
Hence, for all $Z\in\mathcal{C}_{\le^{\st}n}$ we get 
\[
\map(Z,X)=\map(Z,R_{\le n}X)=\map(Z,P_{n}X)=\map(Z,0)=\pt
\]
and so $X\in\mathcal{C}_{>^{\st}n}$. We can take the base of the
induction to be $n=-1$, in which there is nothing to prove. 
\end{proof}
When further assuming $1$-semiadditivity, we get the following:
\begin{prop}
\label{prop:Bounded_Stable_Height_Decomposition}For every $p$-local
$\mathcal{C}\in\Pr_{\st}^{\sad 1}$, we have a canonical equivalence:
\[
\mathcal{C}_{\le^{\st}n}\simeq\mathcal{C}_{0^{\st}}\times\mathcal{C}_{1^{\st}}\times\cdots\times\mathcal{C}_{n^{\st}}.
\]
\end{prop}

\begin{proof}
By tensoring the equivalence of \thmref{1Sad_Decomposition} with
$\mathcal{C}$, we get an equivalence
\[
\mathcal{C}\otimes\CMon_{1}\otimes L_{n}^{f}\Sp\simeq\mathcal{C}\otimes\prod_{k=0}^{n}\Sp_{T(n)}\simeq\prod_{k=0}^{n}(\mathcal{C}\otimes\Sp_{T(n)}).
\]
Since $\mathcal{C}$ is already $1$-semiadditive, $\mathcal{C}\otimes\CMon_{1}\simeq\mathcal{C}$.
Thus, the claim follows from \corref{Stable_Height_Localization}.
\end{proof}

\subsubsection{Comparing heights}

For a stable higher semiadditive $\infty$-category, it is natural to compare the stable height with the semiadditive height. First, 
\begin{lem}
\label{lem:Height_Order}Let $\mathcal{C}\in\Pr_{\st}^{\sad{\infty}}$
be $p$-local. For all $n,k\in\bb N\cup\{\infty\}$, we have $(\mathcal{C}_{n^{\st}})_{k}\simeq(\mathcal{C}_{k})_{n^{\st}}$.
\end{lem}

\begin{proof}
Using \thmref{Tsadi_n_Mode} and \propref{Stable_Modes} (or \corref{Height_Inf_Mode}
for $n=\infty$) we get 
\[
(\mathcal{C}_{n^{\st}})_{k}\simeq\mathcal{C}\otimes\Sp_{(p),n^{\st}}\otimes\tsadi_{k}\simeq(\mathcal{C}_{k})_{n^{\st}}.
\]
\end{proof}
We next observe that for $n\in\bb N$ the $\infty$-category $\Sp_{(p),n^{\st}}=\Sp_{T(n)}$,
which is the mode of stable height $n$, is $\infty$-semiadditive
of semiadditive height $n$ (\thmref{Height_Chrom}). Therefore, there
is a map of modes $\tsadi_{n}\to\Sp_{T(n)}$, making $\Sp_{T(n)}$
an algebra over $\tsadi_{n}$.
\begin{prop}
\label{prop:Stable_Sad_Height}Let $\mathcal{C}\in\Pr_{\st}^{\sad{\infty}}$
be $p$-local. For all $n\in\bb N$ and $k\in\bb N\cup\{\infty\}$,
we have 
\[
(\mathcal{C}_{n^{\st}})_{k}\simeq(\mathcal{C}_{k})_{n^{\st}}\simeq\begin{cases}
\mathcal{C}_{n^{\st}} & k=n\\
0 & k\neq n
\end{cases}
\]
\end{prop}

\begin{proof}
On the one hand, for $k\neq n$ we have $\tsadi_{n}\otimes\tsadi_{k}=0$
(\propref{Tasdi_n_Disjoint}) and so $\Sp_{T(n)}\otimes\tsadi_{k}=0$.
On the other, $\Sp_{T(n)}\otimes\tsadi_{n}\simeq\Sp_{T(n)}$ as $\Sp_{T(n)}$
is a $\tsadi_{n}$-module. Tensoring these with $\mathcal{C}$ yields the claim.
\end{proof}
Given $\mathcal{C}\in\Pr_{\st}^{\sad{\infty}},$ tensoring the map
of modes $\tsadi_{n}\to\Sp_{T(n)}$ with $\mathcal{C}$, yields a
map $\mathcal{C}_{n}\to\mathcal{C}_{n^{\st}}$. 
\begin{prop}
\label{prop:Stable_Sad_Recollement}Let $\mathcal{C}\in\Pr_{\st}^{\sad{\infty}}$
be $p$-local. For every $n\in\bb N$, the map $\mathcal{C}_{n}\to\mathcal{C}_{n^{\st}}$
admits a fully faithful right adjoint $\mathcal{C}_{n^{\st}}\into\mathcal{C}_{n}$,
which exhibits $\mathcal{C}_{n}$ as a recollement of $\mathcal{C}_{n^{\st}}$
and $(\mathcal{C}_{n})_{\infty^{\st}}$ ($=(\mathcal{C}_{\infty^{\st}})_{n}$).
In particular, if $\mathcal{C}_{\infty^{\st}}=0$, then $\mathcal{C}_{n}\simeq\mathcal{C}_{n^{\st}}$.
\end{prop}

\begin{proof}
By \corref{Stable_Height_Recollement}, for every $N\ge n$, the $\infty$-category
$\mathcal{C}_{n}$ is a recollement of $(\mathcal{C}_{n})_{\le^{\st}N}$
and $(\mathcal{C}_{n})_{>^{\st}N}$. Applying \propref{Bounded_Stable_Height_Decomposition}
and \propref{Stable_Sad_Height} to $\mathcal{C}_{n}$ we obtain
\[
(\mathcal{C}_{n})_{\le^{\st}N}\simeq(\mathcal{C}_{n})_{0^{\st}}\times(\mathcal{C}_{n})_{1^{\st}}\times\cdots\times(\mathcal{C}_{n})_{N^{\st}}\simeq\mathcal{C}_{n^{\st}}.
\]
Consequently, 
\[
(\mathcal{C}_{n})_{>^{\st}N}=(\mathcal{C}_{n})_{>^{\st}N+1}=(\mathcal{C}_{n})_{>^{\st}N+2}=\dots=(\mathcal{C}_{n})_{\infty^{\st}}
\]
and hence, $\mathcal{C}_{n}$ as a recollement of $\mathcal{C}_{n^{\st}}=(\mathcal{C}_{n})_{\le^{\st}N}$
and $(\mathcal{C}_{n})_{\infty^{\st}}=(\mathcal{C}_{n})_{>^{\st}N}$.
\end{proof}
As a consequence, we get a tight connection between $\Sp_{T(n)}$ and
$\tsadi_{n}$.
\begin{cor}
\label{cor:Tsadi_Tn_Smashing}The map of modes $ \tsadi_{n}\to \Sp_{T(n)}$
is a smashing localization.
\end{cor}

\begin{proof}
This follows from \propref{Mode_Smashing_Localization} and \propref{Stable_Sad_Recollement}.
\end{proof}

\subsubsection{Bootstrap of Semiadditivity}

Based on the classification of higher semiadditive localizations of
$\Sp$ with respect to homotopy rings, the authors proposed in \cite[Conjecture 1.1.5]{Ambi2018}
the ``Bootstrap Conjecture'', stating that if a presentable stable
$p$-local $\infty$-category is $1$-semiadditive, then it is automatically
$\infty$-semiadditive. Using the $1$-semiadditive decomposition
of \thmref{1Sad_Decomposition} we now provide some partial results in the direction of proving this conjecture. First, given a $p$-local $\mathcal{C}\in\Pr_{\st}$
and $n\in\bb N$, the $\infty$-category $\mathcal{C}_{n^{\st}}$
is a module over $\Sp_{T(n)}$ and hence over $\tsadi_{n}$. It follows
that $\mathcal{C}_{n^{\st}}$ is $\infty$-semiadditive and of height
$n$. More generally, if $\mathcal{C}$ is $1$-semiadditive, then
by \propref{Bounded_Stable_Height_Decomposition}, we have $\mathcal{C}_{\le^{\st}n}\simeq\prod_{k=0}^{n}\mathcal{C}_{k^{\st}}$.
From this one can deduce that, if $\mathcal{C}$ is $1$-semiadditive and every object of $\mathcal{C}$ is of bounded stable height, then 
$\mathcal{C}$ is $\infty$-semiadditive. However, we shall show that having
such a bound on the stable height of the objects of $\mathcal{C}$
is an unnecessarily strong restriction, and it in fact suffices to
assume merely that $\mathcal{C}_{\infty^{\st}}=0$. 

Since every stable presentable $\infty$-category of stable height
exactly $n$ (for some $n$) is $\infty$-semiadditive, it has an
action of $\pi_{0}\tsadi^{[1]}$, and a fortiori of the subring $\mathcal{R}_{1}\ss\pi_{0}\tsadi^{[1]}$
(see \defref{R1}). We begin with a generalization of \propref{Div_Comp_R1}.
\begin{prop}
\label{prop:Div_Comp_R1_General}For every $n$, there exists $a\in\mathcal{R}_{1}$,
such that every $\mathcal{C}\in\Pr_{\st}^{\sad 1}$ of stable height
$k$ is $a$-complete if $k\le n$ and $a$-divisible if $k>n$.
\end{prop}

\begin{proof}
By \propref{Div_Comp_R1}, there exists $a\in\mathcal{R}_{1}$,
such that $\Sp_{T(k)}$ is $a$-complete if $k\le n$ and $a$-divisible
if $k>n$. Since $\mathcal{C}$ is an $\Sp_{T(k)}$-module, the action
of $a\in\mathcal{R}_{1}$ on $\mathcal{C}$ is via its image in $\pi_{0}\bb S_{T(k)}$.
Thus, it suffices to show that for every $a\in\pi_{0}\bb S_{T(k)}$,
if $\Sp_{T(k)}$ is $a$-complete (resp. $a$-divisible), then $\mathcal{C}$
is also $a$-complete (resp. $a$-divisible). We observe that $\Sp_{T(k)}$
is $a$-divisible if and only if $a$ is invertible, in which case
$\mathcal{C}$ is $a$-divisible as well. On the other hand, if $\Sp_{T(k)}$
is $a$-complete, then tensoring with $A\coloneqq\bb S_{T(k)}/a$
is conservative on $\Sp_{T(k)}$. The object $A$ is dualizable with
dual $A^{\vee}\simeq\Sigma^{-1}A$ and so tensoring with $A^{\vee}$
is also conservative. Finally, to show that $\mathcal{C}$ is $a$-complete,
it suffices to show that tensoring with $A$ is conservative on $\mathcal{C}$,
via the left tensoring of $\mathcal{C}$ over $\Sp_{T(k)}$. For every
$X,Y\in\mathcal{C}$, we have isomorphisms
\[
\hom_{\mathcal{C}}^{\Sp_{T(k)}}(A\otimes X,Y)\simeq\hom_{\Sp_{T(k)}}^{\Sp_{T(k)}}(A,\hom{}_{\mathcal{C}}^{\Sp_{T(k)}}(X,Y))\simeq A^{\vee}\otimes\hom_{\mathcal{C}}^{\Sp_{T(k)}}(X,Y).
\]
Hence, if $A\otimes X\simeq0$, then by the conservativity of $A^{\vee}\otimes(-)$,
we get that $\hom_{\mathcal{C}}^{\Sp_{T(k)}}(X,Y)=0$ for all $Y$
and hence $X=0$.
\end{proof}
From this we derive a strengthening of \corref{Tn_One_Semiadditivity}.
\begin{prop}
\label{prop:Stable_Height_Increase_Functoriality}Let ${\cal C},{\cal D}\in\Pr_{\st}^{\sad 1}$
be $p$-local. If $\Ht_{\st}({\cal D})\le n$, $\Ht_{\st}({\cal C})>n$
and $\mathcal{C}_{\infty^{\st}}=0$, then every 1-semidditive functor
$F\colon{\cal C}\to{\cal D}$ is zero. 
\end{prop}

\begin{proof}
We shall construct an element $a\in\mathcal{R}_{1}$ such that $\mathcal{C}$
is $a$-divisible and $\mathcal{D}$ is $a$-complete. Since $F$
is $1$-semiadditive, we shall get that it takes every $X\in{\cal C}$
to an object of ${\cal D}$ which is both $a$-complete and $a$-divisible
(\propref{R1_Linearity}). This will imply that $F(X)=0$ for all $X\in\mathcal{C}$
and hence that $F$ is zero. By \propref{Div_Comp_R1_General}, there is
an $a\in{\cal R}_{1}$ such that $\mathcal{C}_{k}$ is $a$-divisible
for $k>n$ and $\mathcal{D}_{k}$ is $a$-complete for $k\le n$.
By \propref{Bounded_Stable_Height_Decomposition}, we have 
\[
\mathcal{D}=\mathcal{D}_{\le^{\st}n}\simeq\prod_{0\le k\le n}\mathcal{D}_{k^{\st}}
\]
and hence $\mathcal{D}$ itself is $a$-complete. As for $\mathcal{C}$, by \propref{Finite_Stable_Height_Conservative}, we have a jointly conservative collection of functors $P_{n}\colon\mathcal{C}\to\mathcal{C}_{n^{\st}}$ for
$n\in\bb N$. Moreover, all the $P_{n}$-s are $1$-semiadditive, as a composition of a left and a right adjoint,
and thus $\mathcal{C}$ is $a$-divisible as well. 
\end{proof}
As a corollary, we get the following partial result in the direction
of \cite[Conjecture 1.1.5]{Ambi2018}:
\begin{thm}
\label{thm:Bounded_Bootstrap}Let ${\cal C}\in\Pr_{\st}$ be $p$-local,
such that $\mathcal{C}_{\infty^{\st}}=0$. If $\mathcal{C}$ is $1$-semiadditive,
then ${\cal C}$ is $\infty$-semiadditive. Moreover, in this case
$\mathcal{C}_{n^{\st}}=\mathcal{C}_{n}$ for all $0\le n\le\infty$
and there is a canonical decomposition
\[
\mathcal{C}\simeq\prod_{n\in\bb N}\mathcal{C}_{n^{\st}}\simeq\prod_{n\in\bb N}\mathcal{C}_{n}.
\]
\end{thm}

\begin{proof}
By \corref{Stable_Height_Recollement}, ${\cal C}$ is a recollement
of ${\cal C}_{\le^{\st}n}$ and ${\cal C}_{>^{\st}n}$. Since ${\cal C}_{>^{\st}n}$
is $1$-semiadditive and 
\[
({\cal C}_{>^{\st}n})_{\infty^{\st}}\simeq{\cal C}_{\infty^{\st}}=0,
\]
every 1-semiadditive functor $F\colon{\cal C}_{>^{\st}n}\to{\cal C}_{\le^{\st}n}$
is zero by \propref{Stable_Height_Increase_Functoriality}. It follows
that ${\cal C}$ is a \emph{split }recollement of\emph{ }${\cal C}_{\le^{\st}n}$
and ${\cal C}_{>^{\st}n}$, and hence a recollement of ${\cal C}_{>^{\st}n}$
and ${\cal C}_{\le^{\st}n}$ (i.e. we may switch the roles). By \corref{Recollement_Chain_Decomposition},
we get 
\[
{\cal C}\simeq\lim_{n\in\mathbb{N}}({\cal C}_{\le^{\st}n})\simeq\lim_{n\in\mathbb{N}}(\prod_{k\le n}{\cal C}_{k^{\st}})\simeq\prod_{n\in\mathbb{N}}{\cal C}_{n^{\st}}.
\]

For every $n\in\bb N$, the $\infty$-category ${\cal C}_{n^{\st}}$
is $\infty$-semiadditive of semiadditive height $n$, hence ${\cal C}$
itself is $\infty$-semiadditive and for every $0\le k\le\infty$
we have (\propref{Stable_Sad_Height})
\[
{\cal C}_{k}\simeq(\prod_{n\in\mathbb{N}}{\cal C}_{n^{\st}})_{k}\simeq\prod_{n\in\mathbb{N}}({\cal C}_{n^{\st}})_{k}\simeq{\cal C}_{k^{\st}}.
\]
\end{proof}

\bibliographystyle{alpha}
\phantomsection\addcontentsline{toc}{section}{\refname}\bibliography{AmbiHeight}

\def\cprime{$'$}
\begin{thebibliography}{HSSS18}

\bibitem[AKQ19]{RedshiftGood}
Gabriel Angelini-Knoll and JD~Quigley.
\newblock Chromatic complexity of the algebraic {K}-theory of $ y (n) $.
\newblock {\em arXiv preprint arXiv:1908.09164}, 2019.

\bibitem[AR08]{RedshiftRognes}
Christian Ausoni and John Rognes.
\newblock The chromatic red-shift in algebraic {K}-theory.
\newblock {\em L'Enseignement Math{\'e}matique}, 54(2):13--15, 2008.

\bibitem[Bae]{Baez}
John~C. Baez.
\newblock Euler characteristic versus homotopy cardinality.
\newblock http://math.ucr.edu/home/baez/cardinality/cardinality.pdf.

\bibitem[Bal16]{BalmerNil}
Paul Balmer.
\newblock Separable extensions in tensor-triangular geometry and generalized
  {Q}uillen stratification.
\newblock {\em Ann. Sci. {\'E}c. Norm. Sup{\'e}r.(4)}, 49(4):907--925, 2016.

\bibitem[BG16]{barwick2016note}
Clark Barwick and Saul Glasman.
\newblock A note on stable recollements.
\newblock {\em arXiv preprint arXiv:1607.02064}, 2016.

\bibitem[BG18]{BobkovaG}
Irina Bobkova and Paul~G. Goerss.
\newblock Topological resolutions in {$K(2)$}-local homotopy theory at the
  prime 2.
\newblock {\em Journal of Topology}, 11(4):917--956, 2018.

\bibitem[BGH17]{Agnes2017}
Agnes Beaudry, Paul~G Goerss, and Hans-Werner Henn.
\newblock Chromatic splitting for the ${K}(2) $-local sphere at $ p= 2$.
\newblock {\em arXiv preprint arXiv:1712.08182}, 2017.

\bibitem[BK72]{BKCore}
AK~Bousfield and DM~Kan.
\newblock The core of a ring.
\newblock {\em Journal of Pure and Applied Algebra}, 2(1):73--81, 1972.

\bibitem[BNT18]{BNT}
Ulrich Bunke, Thomas Nikolaus, and Georg Tamme.
\newblock The {B}eilinson regulator is a map of ring spectra.
\newblock {\em Advances in Mathematics}, 333:41--86, 2018.

\bibitem[Car84]{carlsson1984}
Gunnar Carlsson.
\newblock Equivariant stable homotopy and segal's burnside ring conjecture.
\newblock {\em Annals of Mathematics}, pages 189--224, 1984.

\bibitem[CM17]{ClausenAkhil}
Dustin Clausen and Akhil Mathew.
\newblock A short proof of telescopic {T}ate vanishing.
\newblock {\em Proceedings of the American Mathematical Society},
  145(12):5413--5417, 2017.

\bibitem[CSY18]{Ambi2018}
Shachar Carmeli, Tomer~M Schlank, and Lior Yanovski.
\newblock Ambidexterity in chromatic homotopy theory.
\newblock {\em arXiv preprint arXiv:1811.02057}, 2018.

\bibitem[GGN16]{GepUniv}
David Gepner, Moritz Groth, and Thomas Nikolaus.
\newblock Universality of multiplicative infinite loop space machines.
\newblock {\em Algebraic \& Geometric Topology}, 15(6):3107--3153, 2016.

\bibitem[GH15]{Rune}
David Gepner and Rune Haugseng.
\newblock Enriched $\infty$-categories via non-symmetric $\infty$-operads.
\newblock {\em Advances in mathematics}, 279:575--716, 2015.

\bibitem[GS96]{GState}
John~P.C. Greenlees and Hal Sadofsky.
\newblock The {T}ate spectrum of {$v_n$}-periodic complex oriented theories.
\newblock {\em Mathematische Zeitschrift}, 222(3):391--405, 1996.

\bibitem[Har17]{Harpaz}
Yonatan Harpaz.
\newblock Ambidexterity and the universality of finite spans.
\newblock arXiv preprint arXiv:1703.09764, 2017.

\bibitem[Hin16]{Hinich}
Vladimir Hinich.
\newblock Dwyer-{K}an localization revisited.
\newblock {\em Homology Homotopy Appl.}, 18(1):27--48, 2016.

\bibitem[HL13]{HopkinsLurie}
Michael Hopkins and Jacob Lurie.
\newblock Ambidexterity in {$K(n)$}-local stable homotopy theory.
\newblock {http://www.math.harvard.edu/~lurie/}, 2013.

\bibitem[HS96]{HState}
Mark Hovey and Hal Sadofsky.
\newblock {T}ate cohomology lowers chromatic {B}ousfield classes.
\newblock {\em Proceedings of the American Mathematical Society},
  124(11):3579--3585, 1996.

\bibitem[HS98]{nilp2}
Michael~J. Hopkins and Jeffrey~H. Smith.
\newblock Nilpotence and stable homotopy theory {II}.
\newblock {\em Ann. of Math. (2)}, 148(1):1--49, 1998.

\bibitem[HSSS18]{HSSS2018}
Marc Hoyois, Pavel Safronov, Sarah Scherotzke, and Nicol{\`o} Sibilla.
\newblock The categorified {G}rothendieck-{R}iemann-{R}och theorem.
\newblock {\em arXiv preprint arXiv:1804.00879}, 2018.

\bibitem[HY17]{AD}
Asaf Horev and Lior Yanovski.
\newblock On conjugates and adjoint descent.
\newblock {\em Topology and its Applications}, 232:140--154, 2017.

\bibitem[Kuh04]{Kuhn}
Nicholas~J. Kuhn.
\newblock Tate cohomology and periodic localization of polynomial functors.
\newblock {\em Inventiones mathematicae}, 157(2):345--370, 2004.

\bibitem[Lura]{ha}
Jacob Lurie.
\newblock Higher algebra.
\newblock {http://www.math.harvard.edu/~lurie/}.

\bibitem[Lurb]{LurieRep}
Jacob Lurie.
\newblock Representation theory in intermediate characteristic.
\newblock http://www.claymath.org/downloads/utah/Notes/Lurie2.pdf.

\bibitem[Lur09]{htt}
Jacob Lurie.
\newblock {\em Higher topos theory}, volume 170 of {\em Annals of Mathematics
  Studies}.
\newblock Princeton University Press, Princeton, NJ, 2009.

\bibitem[Lur19]{Ell3}
Jacob Lurie.
\newblock Elliptic cohomology {III}: Tempered cohomology.
\newblock {\em Available from the author{'}s webpage}, 2019.

\bibitem[MNN17]{MNN17}
Akhil Mathew, Niko Naumann, and Justin Noel.
\newblock Nilpotence and descent in equivariant stable homotopy theory.
\newblock {\em Advances in Mathematics}, 305:994--1084, 2017.

\bibitem[PS17]{Sylow2017}
Matan Prasma and Tomer~M Schlank.
\newblock Sylow theorems for $\infty$-groups.
\newblock {\em Topology and its Applications}, 222:121--138, 2017.

\end{thebibliography}

\end{document}